\pdfoutput=1

\documentclass[10pt]{amsart}
\usepackage[utf8]{inputenc}
\usepackage[T1]{fontenc}
\usepackage{mathtools}
\usepackage{amsmath,amsthm,amsfonts,amssymb,mathrsfs,pb-diagram,amscd}
\usepackage{amsmath}
\usepackage{amsxtra}
\usepackage{amscd}
\usepackage{amsthm}
\usepackage{amsfonts}
\usepackage{amssymb}
\usepackage{eucal}
\usepackage{xcolor}
\usepackage[enableskew,vcentermath]{youngtab}
\usepackage{bm}
\usepackage{leftindex}
\usepackage{graphics}
\usepackage{mathrsfs}
\usepackage{enumerate}
\usepackage{tikz}

\usetikzlibrary{decorations,decorations.pathmorphing,decorations.pathreplacing,decorations.markings}
\usepackage[pagebackref=false]{hyperref}
\DeclareMathOperator\Std{Std}
\DeclareMathOperator\res{res}

\allowdisplaybreaks[4]

\def\O{\mathscr{O}}
\def\K{\mathscr{K}}
\def\hO{{\hat{\O}}}
\def\hK{{\hat{\K}}}
\newtheorem{thm}{Theorem}[section]
\theoremstyle{plain}

\newtheorem{lem}[thm]{Lemma}

\newtheorem{prop}[thm]{Proposition}
\newtheorem{cor}[thm]{Corollary}

\theoremstyle{definition}
\newtheorem{defn}[thm]{Definition}
\newtheorem{example}[thm]{Example}

\theoremstyle{remark}
\newtheorem{rem}[thm]{Remark}
\newtheorem{conjecture}[thm]{Conjecture}

\definecolor{A}{rgb}{.75,1,.75}

\numberwithin{equation}{section}

\newcommand{\C}{\mathbb C}
\newcommand{\Z}{\mathbb Z}
\newcommand{\N}{\mathbb N}

\newcommand{\mb}{\mathtt{b}}
\newcommand{\mt}{\mathfrak{t}}

\newcommand{\supp}{\text{supp}}

\newcommand{\undla}{\underline{\lambda}}
\newcommand{\undmu}{\underline{\mu}}
\newcommand{\undQ}{\underline{Q}}

\newcommand{\Add}{{\rm Add}}
\newcommand{\Rem}{{\rm Rem}}
\newcommand{\Hom}{{\rm Hom}}

	\newcommand{\hs}{\hspace*}
\def\O{\mathscr{O}}
\def\K{\mathscr{K}}
\def\hO{{\hat{\O}}}
\def\hK{{\hat{\K}}}
\def\rd{{\rm{d}}}
\def\Sym{\mathfrak{S}}
\def\pr{{\rm pr}}
\def\tQ{\widetilde{Q}}
\let\<=\langle
\let\>=\rangle
\def\({\big(}
\def\){\big)}

\newcommand{\fourlinerightarrow}{%
	\mathrel{%
		\vcenter{%
			\offinterlineskip
			\halign{##\cr
				$\rule{0.6em}{0.3pt}\kern-0.2em\rule{0.3em}{0.3pt}$\cr
				\noalign{\kern0.9pt}
				$\rule{0.6em}{0.3pt}\kern-0.2em\rule{0.3em}{0.3pt}$\cr
				\noalign{\kern0.9pt}
				$\rule{0.6em}{0.3pt}\kern-0.2em\rule{0.3em}{0.3pt}$\cr
				\noalign{\kern0.9pt}
				$\rule{0.6em}{0.3pt}\kern-0.2em\rule{0.3em}{0.3pt}$\cr
			}%
		}%
		\kern-0.2em\raisebox{0.0ex}{$\succ$}%
	}%
}

\newcommand{\fourlineleftarrow}{%
	\mathrel{%
		\raisebox{-0.2ex}{$\prec$}
		\kern-0.2em
		\vbox{%
			\offinterlineskip
			\halign{##\cr
				$\rule{0.6em}{0.3pt}\kern-0.2em\rule{0.3em}{0.3pt}$\cr
				\noalign{\kern0.9pt}
				$\rule{0.6em}{0.3pt}\kern-0.2em\rule{0.3em}{0.3pt}$\cr
				\noalign{\kern0.9pt}
				$\rule{0.6em}{0.3pt}\kern-0.2em\rule{0.3em}{0.3pt}$\cr
				\noalign{\kern0.9pt}
				$\rule{0.6em}{0.3pt}\kern-0.2em\rule{0.3em}{0.3pt}$\cr
			}%
		}%
	}%
}

{\vskip-\lastskip\medskip
	\noindent
	{\em #1.}\enspace
}%
{\qed\par\medskip
}

\usepackage{enumitem}

\newlist{caselist}{enumerate}{1}
\setlist[caselist,1]{
	label=\textbf{Case \arabic{section}.\arabic{subsection}.\arabic*:},
	ref=\arabic{section}.\arabic{subsection}.\arabic*,
	before=\setcounter{caselisti}{0}
}

\newcounter{case}

\begin{document}

	\title[cyclotomic Hecke-Clifford superalgebras]{On the semisimplicity and Schur elements of (super)symmetric superalgebras}

    \author{Lei Shi}
    	\address{Max-Planck-Institut f\"ur Mathematik\\
	    Vivatsgasse 7, 53111 Bonn\\
	    Germany}
    \email{leishi202406@163.com}

	\begin{abstract}
In this paper, we use Schur elements to derive semisimplicity criteria for (super)symmetric superalgebras. We obtain a closed formula for the Schur elements of cyclotomic Hecke-Clifford superalgebras $\mathcal{H}^{f}_{\mathbb{K}}$. As applications, we prove that two trace functions $\gimel_n$ and $t_{1,n}$ on the Hecke-Clifford superalgebra, which are defined in different ways, are proportional. We give a semisimplicity criterion for $\mathcal{H}^{f}_{\mathbb{K}}$ when it is (super)symmetric. We also derive semisimplicity criteria for cyclotomic quiver Hecke superalgebras of types $A^{(1)}_{e}$, $C^{(1)}_{e}$, $A^{(2)}_{2e}$ and $D^{(2)}_{e+1}$.
	\end{abstract}
	\maketitle
	
	\setcounter{tocdepth}{1}
	\tableofcontents
	
	\section{Introduction}
	
	Let $\mathbb{K}$ be a field and let $A$ be a finite-dimensional symmetric algebra over $\mathbb{K}$. Let ${\rm Irr}(A)$ denote a complete set of pairwise non-isomorphic simple $A$-modules. To each $V\in {\rm Irr}(A)$, Geck and Pfeiffer \cite{GF} associated an element $S_V\in \mathbb{K}$, called the Schur element of $V$. They proved that 
	\begin{align}\label{semisimplicity-non-super}
	\text{ $A$ is semisimple if and only if $S_V\neq 0$ for every $V\in {\rm Irr}(A)$. }
	\end{align}
	They also gave a semisimplicity criterion in a modular reduction setting \cite[Proposition 7.3.9, Theorem 7.4.7]{GF}. These criteria are useful in the study of Hecke algebras \cite{GF} and their cyclotomic generalizations \cite{CJ}.
	
 {\bf In the rest of this paper, ${\rm R}$ is an integral domain of characteristic different from $2$ and $\mathbb{K}$ is a field of characteristic different from $2$.} The first goal of this paper is to establish super analogues of the above semisimplicity criteria. The super setting has an additional feature: simple modules may be of type \texttt{M} or type \texttt{Q}, and this distinction has to be reflected in any Schur-element criterion. We begin by recalling the notion of a (super)symmetric superalgebra.
	
%
%
	
	\begin{defn}\label{(super)symmetrizing form}\cite[Definition 2.1]{LS1}, \cite[Section 4.1, 5.1]{WW2}
		Let ${\mathcal{A}}={\mathcal{A}}_{\overline{0}}\oplus {\mathcal{A}}_{\overline{1}}$ be an ${\rm R}$-superalgebra
		which is free and of finite rank over ${\rm R},$ $|\cdot|: {\mathcal{A}} \rightarrow \Bbb Z_{2}$ be the parity map.
		
		(i) We call an ${\rm R}$-linear map $t:\mathcal{A} \rightarrow {\rm R}$ non-degenerate if there is a $\Z_2$-homogeneous basis $\mathcal{B}$ such that the determinant ${\rm det}\left(t(b_1b_2)\right)_{b_1,b_2\in \mathcal{B}}\in {\rm R}^\times.$
		
		(ii) The superalgebra $\mathcal{A}$ is called symmetric if there is an evenly non-degenerate ${\rm R}$-linear map $t:\mathcal{A} \rightarrow {\rm R}$
		such that $t(xy)=t(yx)$ for any $x, y\in \mathcal{A}$. In this case, we call $t$ a symmetrizing form on $\mathcal{A}$.
		
		(iii) The superalgebra $\mathcal{A}$ is called supersymmetric if there is an evenly non-degenerate ${\rm R}$-linear map $t:\mathcal{A} \rightarrow {\rm R}$
		such that  $t(xy)=(-1)^{|x||y|}t(yx)$ for any homogeneous $x, y\in \mathcal{A}$. In this case, we call $t$ a supersymmetrizing form on $\mathcal{A}.$
	\end{defn}
	
	Let $\mathcal{A}$ be a finite-dimensional $\mathbb{K}$-superalgebra. We denote by ${\rm Irr}(\mathcal{A})$ a complete set of pairwise non-isomorphic simple $\mathcal{A}$-supermodules. When $\mathcal{A}$ is supersymmetric or symmetric, the Schur element $s_V$ attached to $V\in{\rm Irr}(\mathcal{A})$ was introduced in \cite{LS3}, \cite{WW2}. 
		
	Now suppose that $\hO$ is a discrete valuation ring with residue field $\mathbb{K}$ and fraction field $\hK$, and that $\mathcal{A}$ is an $\hO$-algebra which is free of finite rank as an $\hO$-module. The following theorem is the first main result of this paper; we refer the reader to Section \ref{preli} for the unexplained notation.
	
		\begin{thm}\label{modular-semisimple 2}
			Suppose $\mathcal{A}$ is (super)symmetric with a (super)symmetrizing form $t$, and $\mathbb{K}\mathcal{A}$, $\mathscr{K}\mathcal{A}$ are split. Then the following hold.
			\begin{enumerate}
				\item If $t$ is symmetric, then $\mathbb{K}\mathcal{A}$ is semisimple if and only if ${\rm pr}(s_V)\neq 0$ for any $V\in { \rm Irr}(\mathscr{K}\mathcal{A})$.
				\item If $t$ is supersymmetric, then $\mathbb{K}\mathcal{A}$ is semisimple if and only if ${\rm pr}(s_V)\neq 0$ for any $V\in { \rm Irr}(\mathscr{K}\mathcal{A})$ and all simple modules of $\mathscr{K}\mathcal{A}$ are of {\em type }\texttt{M}.
			\end{enumerate}
		\end{thm}
		
	We also establish a super analogue of \eqref{semisimplicity-non-super} and a super version of Tits's deformation theorem (Theorem \ref{super tits's}), both of which are used in the proof of Theorem \ref{modular-semisimple 2}. 
		
	We then apply Theorem \ref{modular-semisimple 2} to cyclotomic Hecke-Clifford superalgebras $\mathcal{H}^{f}_{\mathbb{K}}$. These superalgebras were introduced by Brundan and Kleshchev \cite{BK}
		in their study of the modular representation theory of the spin symmetric group. They are related to crystals of twisted types $A^{(2)}$ \cite{BK} and $D^{(2)}$ \cite{T}. Kang, Kashiwara and Tsuchioka gave a non-trivial $\Z$-graded structure on $\mathcal{H}^{f}_{\mathbb{K}}$ \cite{KKT}. Moreover, Kang, Kashiwara and Oh used $\mathcal{H}^{f}_{\mathbb{K}}$ to categorify certain highest weight modules of (super)quantum groups of types $A^{(1)}$, $C^{(1)}$ and twisted types $A^{(2)}$ and $D^{(2)}$ \cite{KKO1,KKO2}. See also \cite{HW}, which initiated the study of categorification of quantum Kac-Moody superalgebras with no isotropic odd simple
		roots by cyclotomic quiver Hecke superalgebras.
		
	Let $n\in\N$ and let $\mathcal{H}^{f}_{\mathbb{K}}:=\mathcal{H}^{f}_{\mathbb{K}}(n)$ be the cyclotomic Hecke-Clifford superalgebra over an algebraically closed field $\mathbb{K}$, where $f=f^{(\bullet)}_{\undQ}$ is a polynomial determined by $\undQ=(Q_1,Q_2,\ldots,Q_m)\in({\rm R}^\times)^m$ and $\bullet\in\{\mathtt{0},\mathtt{s}\}$. We denote by $P^{(\bullet)}_{n}(q^2,\undQ)$ the Poincar\'e polynomial appearing in \cite[before Proposition 4.11]{SW}. It was proved in \cite[Theorem 1.1]{SW} that if $P^{(\bullet)}_{n}(q^2,\undQ)\neq 0$, then $\mathcal{H}^{f}_{\mathbb{K}}$ is semisimple over $\mathbb{K}$. In \cite[Conjecture 4.13]{SW}, we conjectured that this non-vanishing condition is also necessary for the semisimplicity of $\mathcal{H}^{f}_{\mathbb{K}}$. The following theorem gives a partial positive answer to this conjecture and is the second main result of this paper. Again, the unexplained notation can be found in Section \ref{Cyclotomic Hecke-Clifford superalgebra}.
		
%
		\begin{thm}\label{semisimplicity for cyclotomic Hecke-Clifford}
			Let $\bullet\in\{\mathtt{0},\mathtt{s}\}$ and $\undQ=(Q_1,Q_2,\ldots,Q_m)\in(\mathbb{K}^*)^m$. Assume $f=f^{(\bullet)}_{\undQ}$.
			Then we have the following.
			\begin{enumerate} 
				\item Suppose $\bullet=\mathtt{0}$. Then $\mathcal{H}^{f}_{\mathbb{K}}$ is (split) semisimple if and only if $P^{(\mathtt{0})}_{n}(q^2,\undQ)\neq 0$.
				\item Suppose $\bullet=\mathtt{s}$ and $\Gamma_n^{(\mathsf{s})}(q^2,\underline{Q})\neq 0$. Then $\mathcal{H}^{f}_{\mathbb{K}}$ is (split) semisimple if and only if $P^{(\mathtt{s})}_{n}(q^2,\undQ)\neq 0$.
			\end{enumerate}
		\end{thm}
		
	Theorem \ref{semisimplicity for cyclotomic Hecke-Clifford} can be viewed as a super version of the main result in \cite{Ar1}. The proof uses the (super)symmetrizing form $t_{r,n}$ from \cite[Theorem 1.2]{LS3} together with the Schur elements $s_{\undla}$ of simple modules \cite[Theorem 1.4]{LS3}. When $P^{(\bullet)}_{n}(q^2,\undQ)\neq 0$, the complete set of pairwise non-isomorphic simple modules is parametrized by $\mathscr{P}^{\bullet,m}_{n}$. For $\undla\in\mathscr{P}^{\bullet,m}_{n}$ with $\bullet\in\{\mathsf{0},\mathsf{s}\}$, a recursive formula for the Schur element $s_{\undla}$ of the simple $\mathcal{H}^{f}_{\mathbb{K}}$-module $\mathbb{D}(\undla)$ with respect to $t^{\mathbb{K}}_{r,n}$ was obtained in \cite[Theorem 1.4]{LS3}. In this paper, we derive the following closed formula for $s_{\undla}$, which may be viewed as an analogue of \cite[Corollary 6.3]{Ma2}.
		
			\begin{thm}\label{closed-schur}
			Suppose that $\bullet\in\{\mathsf{0},\mathsf{s}\}$ and  $P^{(\bullet)}_{n}(q^2,\undQ)\neq 0$ holds. Let $\undla\in\mathscr{P}^{\bullet,m}_{n},$ then the Schur element $s_{\undla}$ of simple $\mathcal{H}^{f}_{\mathbb{K}}$-module $\mathbb{D}(\undla)$ with respect to $t^{\mathbb{K}}_{r,n}$ is given by
			\begin{align*}
				s_{\undla}
				=\begin{cases}
					\left(\prod\limits_{(a,b,c)\in\undla}\left(-2\epsilon\sqrt{[b-a+1]_{c,q^2}[b-a]_{c,q^2}}\right)\right)
					\cdot \left(\prod\limits_{c=1}^m	\mathbb{Y}^{\undla}_{c}\right)\cdot \left(\prod\limits_{1\leq c_1<c_2\leq m}	\mathbb{X}^{\undla}_{c_1,c_2}\right),& \text{ if } \bullet=\mathsf{0},\\
					2^{-\lceil \sharp\mathcal{D}_{\undla}/2 \rceil}	\cdot \left(\prod\limits_{c=0}^m	\mathbb{Y}^{\undla}_{c}\right)\cdot \left(\prod\limits_{0\leq c_1<c_2\leq m}	\mathbb{X}^{\undla}_{c_1,c_2}\right), & \text{ if } \bullet=\mathsf{s}.
				\end{cases}
			\end{align*}
		\end{thm}
		
	It is natural to ask for a cancellation-free formula for $s_{\undla}$, as in \cite{CJ}. Because of the complexity of $s_{\undla}$ in the super setting, we obtain such formulas only for some special shapes; see Subsection \ref{Cancellation-free formula for some special shapes}. These special cases are nevertheless sufficient for the proof of Theorem \ref{semisimplicity for cyclotomic Hecke-Clifford}.
	
	As a byproduct, Theorem \ref{closed-schur} allows us to compare two trace forms on the Hecke-Clifford superalgebra. Let $\mathcal{HC}_{\rm R}(n)$ be the Hecke-Clifford superalgebra over ${\rm R}$. This algebra was introduced in the study of Olshanski duality \cite{O}. A trace function is a map $\phi:\mathcal{HC}_{\rm R}(n)\rightarrow {\rm R}$ such that $\phi(ab)=\phi(ba)$ for any $a,b\in{\mathcal{HC}_{\rm R}(n)}_{\bar{0}}$ and $\phi\left({\mathcal{HC}_{\rm R}(n)}_{\bar{1}}\right)=0$. In \cite{WW2}, Wan and Wang defined a trace form $\gimel^{\rm R}_n$ on $\mathcal{HC}_{\rm R}(n)$. Its definition is based on a specific basis of the cocenter, and it was shown to be non-degenerate after extending scalars to the fraction field of $\Z[\frac{1}{2}][q,q^{-1}]$. This trace function is closely related to spin fake degrees \cite{WW1}. Recently, another trace form $t^{\rm R}_{1,n}$ on $\mathcal{HC}_{\rm R}(n)$ was introduced in \cite[Proposition 1.3]{LS3}. The form $t^{\rm R}_{1,n}$ is obtained by modifying a natural Frobenius form by a certain element. It is non-degenerate in a more general setting, and its values on a certain basis of $\mathcal{HC}_{\rm R}(n)$ can be computed explicitly. The relationship between $\gimel^{\rm R}_n$ and $t^{\rm R}_{1,n}$ is not immediate from their definitions. We prove the following comparison.
	
		\begin{prop}\label{trace function}
		$$\gimel^{\rm R}_n=\frac{1}{2^n} t^{\rm R}_{1,n}.$$
	\end{prop}
	
	The proof reduces to the case ${\rm R}=\Z[\frac{1}{2}][q,q^{-1}]$ and then uses base change. In this special case, we apply semisimple representation theory and compare Schur elements to obtain the proposition.
	
	Finally, we use Kang-Kashiwara-Tsuchioka's isomorphism and Theorem \ref{semisimplicity for cyclotomic Hecke-Clifford} to study semisimplicity criteria for cyclotomic quiver Hecke superalgebras of affine types $A^{(1)}_{e}$, $C^{(1)}_{e}$, $A^{(2)}_{2e}$ and $D^{(2)}_{e+1}$; see Corollaries \ref{semisimplicty typeA' }, \ref{semisimplicty typeC' }, \ref{semisimplicty twised typeA' } and \ref{semisimplicty twised typeD' }. For affine types $A^{(1)}_{e}$ and $C^{(1)}_{e}$, we recover \cite[Corollary 1.6.11]{Ma1} and \cite[Theorem 1.1]{Sp}. For affine types $A^{(2)}_{2e}$ and $D^{(2)}_{e+1}$, the criteria appear to be new.
	
	In a forthcoming paper, we will prove the semisimplicity criterion for $\mathcal{H}^{f}_{\mathbb{K}}$, with $f=f^{(\bullet)}_{\undQ}$ and $\bullet\in\{\mathtt{0},\mathtt{s},\mathtt{ss}\}$, by a different method. We will also consider the degenerate case, namely the cyclotomic Sergeev superalgebra. 
	
The paper is organized as follows. In Section \ref{preli}, we recall some basic facts about superalgebras and prove the semisimplicity criterion in Subsection \ref{Semisimple Criterion}. After comparing several Schur elements and developing a super analogue of Tits's deformation theorem, we prove Theorem \ref{modular-semisimple 2} in Subsection \ref{Modular Reduction}. In Section \ref{Cyclotomic Hecke-Clifford superalgebra}, we recall basic properties of affine and cyclotomic Hecke-Clifford superalgebras, together with the related combinatorics; in particular, we review the separate condition and the semisimple representation theory. In Section \ref{Some formulae for Schur elements}, we study Schur elements for cyclotomic Hecke-Clifford superalgebras. Theorem \ref{closed-schur} is proved in Subsection \ref{Closed formula for Schur elements}, and cancellation-free formulas for some special shapes are obtained in Subsection \ref{Cancellation-free formula for some special shapes}. In Section \ref{On two trace functions of Hecke-Clifford superalgebra}, we prove Proposition \ref{trace function}. In Section \ref{Semisimplicity criterion for cyclotomic Hecke-Clifford superalgebra}, we prove Theorem \ref{semisimplicity for cyclotomic Hecke-Clifford}. In Section \ref{Cyclotomic quiver Hecke superalgebra and cyclotomic quiver Hecke-Clifford superalgebra}, we introduce the relevant Dynkin diagrams and the corresponding cyclotomic quiver Hecke superalgebras and cyclotomic quiver Hecke-Clifford superalgebras. In Section \ref{Semisimplicity criterion for cyclotomic quiver Hecke superalgebra}, we study semisimplicity criteria for cyclotomic quiver Hecke superalgebras of affine types $A^{(1)}_{e}$, $C^{(1)}_{e}$, $A^{(2)}_{2e}$ and $D^{(2)}_{e+1}$.

\bigskip
\centerline{\bf Acknowledgements}
\bigskip
The author thanks Shuo Li and Jinkui Wan for many valuable discussions.
\bigskip

	\section{Semisimplicity criterion for (super)symmetric superalgebras}\label{preli}

	\subsection{Some basics about superalgebras}\label{basics}
	We recall some basic notions of superalgebras, referring the
	reader to~\cite[\S 2-b]{BK}. Let us denote by
	$|v|\in\mathbb{Z}_2$ \label{pag:||} the parity of a homogeneous vector $v$ of an
	${\rm R}$-vector superspace. By a superalgebra, we mean a
	$\mathbb{Z}_2$-graded associative ${\rm R}$-algebra.  Let $\mathcal{A}$ be an
	${\rm R}$-superalgebra. By an $\mathcal{A}$-module, we mean a $\mathbb{Z}_2$-graded
	left $\mathcal{A}$-module. A homomorphism $f:V\rightarrow W$ of
	$\mathcal{A}$-modules $V$ and $W$ means a linear map such that $
	f(av)=(-1)^{|f||a|}af(v).$  Note that this and other such
	expressions only make sense for homogeneous $a, f$ and the meaning
	for arbitrary elements is to be obtained by extending linearly from
	the homogeneous case. 
	
{\bf In the rest of this subsection, we assume $\mathcal{A}$ is a finite-dimensional superalgebra over $\mathbb{K}$.} Let $V$ be a finite-dimensional
	$\mathcal{A}$-module. Let $\Pi
	V$ \label{pag:parity shift} be the same underlying vector space but with the opposite
	$\mathbb{Z}_2$-grading. The new action of $a\in\mathcal{A}$ on $v\in\Pi
	V$ is defined in terms of the old action by $a\cdot
	v:=(-1)^{|a|}av$. Note that the identity map on $V$ defines
	an isomorphism from $V$ to $\Pi V$. Let $\delta: \mathcal{A}\to \mathcal{A}, a\mapsto (-1)^{|a|}a$ be an automorphism of $\mathcal{A}$. Then we can associate another $\mathcal{A}$-module $V^{\delta}$, which is the same vector space $V$ with the new action $a\cdot v:=\delta(a)v$. By forgetting the grading we may consider any superalgebra $\mathcal{A}$ as the usual algebra which will be denoted by $|\mathcal{A}|$. Similarly, any $\mathcal{A}$-supermodule $V$ can be considered as a usual $|\mathcal{A}|$-module denoted by $|V|$. 

The category of finite-dimensional $\mathcal{A}$-supermodules is denoted $\mathcal{A}$-smod. If we forget the super structure, the category of finite-dimensional $|\mathcal{A}|$-modules is denoted $|\mathcal{A}|$-mod. For any module $V\in \mathcal{A}$-smod, we use $\rho_V$ to denote its matrix representation with respect to a $\Z_2$-graded basis of $V$. For $V'\in |\mathcal{A}|$-mod, we use $\rho_{V'}$ to denote its matrix representation with respect to any basis of $V$. Let ${\rm R}\to {\rm R'}$ be a ring homomorphism, we denote ${\rm R'}\mathcal{A}:={\rm R'}\otimes_{\rm R}\mathcal{A}$. Moreover, for any $V\in\mathcal{A}$-smod, we use the notation ${\rm R'}V:={\rm R'}\otimes_{\rm R}V \in{\rm R'}\mathcal{A}$-smod. The same notation applies to $V\in|\mathcal{A}|$-mod. 

Let  $V\in \mathcal{A}$-smod be a simple module, we call $V$ split simple if $\mathbb{K}'V\in \mathbb{K}'\mathcal{A}$-smod remains to be simple, for any field extension $\mathbb{K}\subset\mathbb{K}'$. We say $\mathcal{A}$ is split if any simple $\mathcal{A}$-module $V\in \mathcal{A}$-smod is split.  In particular, when $\mathbb{K}$ is an algebraically closed field, $\mathcal{A}$ is always split.

Suppose $\mathcal{A}$ is split. A superalgebra analogue of Schur's Lemma states that the endomorphism
	algebra of a finite-dimensional simple module over $\mathcal{A}$  is either one-dimensional or two-dimensional. In the
	former case, we call the module of {\em type }\texttt{M}, while in
	the latter case the module is called of {\em type }\texttt{Q}.
	
	\begin{lem}\cite[Lemma 12.2.1, Corollary 12.2.10]{K2}\label{lem:type MQ}
	Suppose $\mathcal{A}$ is split and $V$ is a simple $\mathcal{A}$-module. If $V$ is of type $\texttt{M}$, then by forgetting the grading, $|V|$ is a simple $|\mathcal{A}|$-module. If $V$ is of type $\texttt{Q}$, then by forgetting the grading, $|V|$  is isomorphic to a direct sum of two non-isomorphic simple $|\mathcal{A}|$-modules. That is, there exist two non-isomorphic simple $|\mathcal{A}|$-modules $V^+,V^-$ such that $|V|\cong V^+\oplus V^-$ as $|\mathcal{A}|$-modules. Moreover if $V_1,\cdots,V_m$ (resp. $V_{m+1},\cdots,V_n$) are pairwise non-isomorphic simple $\mathcal{A}$-modules of type \texttt{M} (resp. \texttt{Q}), then $$\{|V_1|, \cdots,|V_m|, V_{m+1}^\pm, \cdots, V_{n}^\pm\}$$ is a set of pairwise non-isomorphic $|\mathcal{A}|$-modules. Moreover, $(V_i^\pm)^\delta\cong V_i^\mp$.
	\end{lem} 
	
	\begin{example}\label{simple algebra}	
		(1) Let $V$ be a superspace with superdimension $(m,n)$ over a field $\mathbb{K}$. Then $\mathcal{M}_{m,n}:={\text{End}}_{\mathbb{K}}(V)$ is a simple superalgebra with simple module $V$ of {\em type }\texttt{M}.  \\
		(2) Let $V$ be a superspace with superdimension $(n,n)$ over a field $\mathbb{K}$. We define
		$\mathcal{Q}_n:=\left\{\begin{pmatrix}
			A  & B\\
			-B & A
		\end{pmatrix}\biggm| A,B\in  M_n\right\}\subset \mathcal{M}_{n,n},$ then $\mathcal{Q}_n$ is a simple superalgebra with simple module $V$ of {\em type }\texttt{Q}.
	\end{example}
	
Denote by $\mathcal{J}(\mathcal{A})$  the usual (non-super) Jacobson radical of $\mathcal{A}.$ We call $\mathcal{A}$ semisimple if $\mathcal{J}(\mathcal{A})=0$. Hence $\mathcal{A}$ is semisimple if and only if $|\mathcal{A}|$ is semisimple. By a super version of Wedderburn's theorem, any split semisimple superalgebra is a direct product of finitely many
	simple split superalgebras, and any split simple superalgebra is isomorphic to some $\mathcal{M}_{m,n}$ or $\mathcal{Q}_n$ (see \cite[Theorem 12.2.9]{K2}).

\subsection{Semisimplicity Criterion}\label{Semisimple Criterion}
Recall the definition of (super)symmetrizing form in Definition \ref{(super)symmetrizing form}. {\bf In this subsection, we assume $\mathcal{A}$ is free and of finite rank over ${\rm R}$ with a (super)symmetrizing form $t$.} Let $\mathcal{B}$ be a $\Z_2$-homogeneous basis of $\mathcal{A}.$ We denote by $\mathcal{B}^{\vee}=\{b^{\vee}\mid b\in \mathcal{B}\}$ the dual basis, which is also homogeneous and satisfies \begin{align}\label{dual}t(b^\vee b')=\delta_{b,b'}\qquad \text{ for any $b,b'\in\mathcal{B}.$ }
	\end{align}Suppose ${V,V'}$ are two $\mathcal{A}$-modules. For any homogeneous map $f\in{\rm Hom}_{{\rm R}}(V, {V'})$, we define $I(f)\in {\rm Hom}_{{\rm R}}(V, { V'})$ by letting $$
I(f)(v):=\begin{cases} \sum_{b\in \mathcal{B}}(-1)^{|f||b|}b^{\vee} f(bv), &\text{if $t$ is symmetric};\\
	\sum_{b\in \mathcal{B}}(-1)^{|f||v|+|b|}b^{\vee} f(bv), &\text{if $t$ is supersymmetric},
\end{cases}
$$ for $v\in V$.

As noted in \cite[\S 5.1]{WW2}, \cite[\S 2.1]{LS3}, when $t$ is (super)symmetric, $I(f)$ is independent of the choice of the homogeneous basis $\mathcal{B}$ and $I(f)\in{\rm Hom}_{\mathcal{A}}(V, {V'})$. 
By definition, for any $g\in{\rm End}_{\mathcal{A}}(V)_{\bar{0}},h\in{\rm End}_{\mathcal{A}}({V'})_{\bar{0}}$, we have $I(hfg)=hI(f)g$.

{\bf In the rest of this subsection, we require ${\rm R}=\mathbb{K}$, i.e. $\mathcal{A}$ is (super)symmetric over $\mathbb{K}$ with a (super)symmetrizing form $t$.}

\begin{defn}\cite[\S 1.1.2]{CW}
	Let $V$ be a superspace with superdimension $(m,n)$ over a field $\mathbb{K}$. For $f\in{\rm End}_{\mathbb{K}}(V)$, we write $f=f_{\bar{0}}+f_{\bar{1}}$, where $f_{a}$ is the $a$-th component of $f$. We define the supertrace of $f$ as follows$$
	{\rm suptr}(f):={\rm tr}(f_{\bar{0}}\downarrow_{V_{\bar{0}}})-{\rm tr}(f_{\bar{0}}\downarrow_{V_{\bar{1}}}).
	$$
\end{defn}


\begin{defn}	Suppose $\mathcal{A}$ is a superalgebra over $\mathbb{K}$ and $V$ is an $\mathcal{A}$-module, we define $$
	\chi_V(a):={\rm tr}\left(\rho_V(a)\right),\qquad \chi'_V(a):={\rm suptr}\left(\rho_V(a)\right),\qquad \forall a\in\mathcal{A}.
	$$
\end{defn}	

\begin{lem}\cite[Lemma 5.2]{WW2}\label{schur}
	Suppose $\mathcal{A}$ is symmetric with a symmetrizing form $t$.	Let $V$ be a split simple $\mathcal{A}$-module. Then there exists a unique element $s_V\in \mathbb{K}$ such that\begin{align*}I(f)=s_V{\rm tr}(f){\rm id}_V \qquad\text{for $f\in{{\rm End}_{\mathbb{K}}(V)}_{\bar{0}}$}.
		\end{align*}
	
 Furthermore, $s_V$ depends only on the isomorphism class of $V$.
\end{lem}

\begin{lem}\cite[Lemma 2.7]{LS3}\label{sschur}
	Suppose $\mathcal{A}$ is supersymmetric with a supersymmetrizing form $t$.	Let $V$ be a split irreducible $\mathcal{A}$-module of {\em type }\texttt{M}. Then there exists a unique element $s_V\in \mathbb{K}$ such that \begin{align*}
		I(f)=s_V{\rm suptr}(f){\rm id}_V,\qquad\text{for $f\in{{\rm End}_{\mathbb{K}}(V)}_{\bar{0}}$}.
	\end{align*}Furthermore, up to sign, $s_V$ depends only on the isomorphism class of $V$.
	\end{lem}
The element $s_V$ is called the {\bf Schur element} of $V$ with respect to $t$. \label{pag:Schur element of V}

We denote by ${\rm Irr}(\mathcal{A})$ the complete set of non-isomorphic simple $\mathcal{A}$-modules. For any $V\in{\rm Irr}(\mathcal{A}),$ we write
$$
\delta(V):=\begin{cases}
	0, &\text{if $V$ is of {\em type }\texttt{M}};\\
	1, &\text{if $V$ is of {\em type }\texttt{Q}}.
	\end{cases}
$$

\begin{prop}\label{schur formula 1}\cite[Proposition 5.4]{WW2}
Suppose $\mathcal{A}$ is a split semisimple superalgebra over $\mathbb{K}$ and $\mathcal{A}$ is symmetric with a symmetrizing form $t$. Then the Schur element $s_V$ for every simple $\mathcal{A}$-module $V$ is non-zero. Moreover,
\begin{align*}
t=\sum_{V\in{\rm Irr}(\mathcal{A})}\frac{1}{2^{\delta_V}s_V}\chi_V.
\end{align*}
\end{prop}

\begin{prop}\cite[Proposition 2.9]{LS3}\label{schur formula 2}
	Suppose $\mathcal{A}$ is a split semisimple superalgebra over $\mathbb{K}$ and $\mathcal{A}$ is supersymmetric with a supersymmetrizing form $t$. Then every simple $\mathcal{A}$-module $V$ is of {\em type }\texttt{M}, and the Schur element $s_V$ is non-zero.
Moreover, we have \begin{align*}
		t=\sum_{\substack{V\in{\rm Irr}(\mathcal{A})}}\frac{1}{s_V}\chi'_V.
	\end{align*}
\end{prop}

\begin{prop}\label{proj-criterion}
Suppose $\mathcal{A}$ is (super)symmetric with a (super)symmetrizing form $t$. Let $V\in\mathcal{A}$-smod. Then $V$ is projective if and only if there exists $\phi\in{\rm End}_{\mathbb{K}}(V)_{\bar{0}}$ such that $I(\phi)={\rm id}_V$.
\end{prop}

\begin{proof}
Suppose there exists $\phi\in{\rm End}_{\mathbb{K}}(V)_{\bar{0}}$ such that $I(\phi)={\rm id}_V$. Let $M\in \mathcal{A}$-smod and $f\in {\rm Hom}_{\mathcal{A}}(M,V)_{\bar{0}}$ be a surjection. We can choose $g\in {\rm Hom}_{\mathbb{K}}(V,M)_{\bar{0}}$, such that $f\circ g={\rm id}_V$. Hence $$f\circ I(  g\circ \phi)=I(f\circ g\circ \phi)=I(\phi)={\rm id}_V.$$ This proves $V$ is projective.

Conversely, suppose $V$ is projective. We consider the tensor product $\mathcal{M}:=\mathcal{A} \otimes_{\mathbb{K}} V$, where the action of $\mathcal{A}$ on $\mathcal{M}$ is given by left multiplication. We define $$
\phi'(b\otimes v):=
	t(b)\otimes v
$$ for any homogeneous $b\in \mathcal{A}, v\in V$ and then extend $\phi'$ to $\mathcal{M}$. It is clear that $\phi'\in {\rm End}_{\mathbb{K}}(\mathcal{M})_{\bar{0}}$. We claim \begin{align}\label{I=id}
	I(\phi')={\rm id}_{\mathcal{M}}.\end{align}
	We only prove \eqref{I=id} in the case that $t$ is supersymmetric, the proof for the case that $t$ is symmetric is similar. For $b'^{\vee}\in \mathcal{B}^{\vee}$, we have \begin{align*}
		I(\phi')(b'^{\vee} \otimes v)&=\sum_{b\in \mathcal{B}}(-1)^{|b|}b^{\vee} \phi'(bb'^{\vee} \otimes v)\\
		&=\sum_{b\in \mathcal{B}}(-1)^{|b'|}b^{\vee} (t(bb'^{\vee} )\otimes v)\\
		&=(-1)^{|b'|}b'^{\vee} (-1)^{|b'|} \otimes v\\
		&=b'^{\vee}\otimes v,
		\end{align*} where the last second equation follows from \eqref{dual}. Since $\mathcal{B}^{\vee}$ forms a basis of $\mathcal{A}$, we deduce that \eqref{I=id} holds.
		Now define $\pi: \mathcal{M}\to V, \,b\otimes v\mapsto bv$. Then $\pi\in {\rm Hom}_{\mathcal{A}}(\mathcal{M},V)_{\bar{0}}$ and $\pi$ is surjective. Since $V$ is projective, we have $\iota\in {\rm End}_{\mathcal{A}}(V,\mathcal{M})_{\bar{0}}$ such that $\pi\circ\iota={\rm id}_V$. Now we let $\phi:=\pi\circ\phi'\circ\iota\in {\rm End}_{\mathbb{K}}(V)_{\bar{0}}$. It follows $$I(\phi)=I(\pi\circ\phi'\circ\iota)=\pi\circ I(\phi')\circ\iota=\pi \circ\iota={\rm id}_V.
		$$
	\end{proof}

\begin{cor}\label{simple-projection}
Let $V\in\mathcal{A}$-smod be a split simple module. We have the following.
	\begin{enumerate}
		\item If $t$ is symmetric, then $V$ is projective if and only if $s_V\neq 0$.
		\item If $t$ is supersymmetric and $V$ is of {\em type }\texttt{M}, then $V$ is projective if and only if $s_V\neq 0$.
		\end{enumerate}
	\end{cor}
	
	\begin{proof}
		This follows from Lemma \ref{schur},\,\ref{sschur} and Proposition \ref{proj-criterion}.
	\end{proof}
	
	\begin{thm}\label{thm:semisimplicity criterion 1}
		Suppose that $\mathcal{A}$ is (super)symmetric with a (super)symmetrizing form $t$ and $\mathcal{A}$ is split. We have the following.
			\begin{enumerate}
			\item If $t$ is symmetric, then $\mathcal{A}$ is semisimple if and only if all Schur elements of simple modules are non-zero.
			\item If $t$ is supersymmetric, then $\mathcal{A}$ is semisimple if and only if all simple modules of $\mathcal{A}$ are of {\em type }\texttt{M}, and all Schur elements of simple modules are non-zero.
		\end{enumerate}
		\end{thm}
		
		\begin{proof}
			The only if part is given by Proposition \ref{schur formula 1} and \ref{schur formula 2}. The if part follows from Corollary \ref{simple-projection}.			
			\end{proof}

\subsection{Comparing Schur elements}{\bf In this subsection, we assume $\mathcal{A}$ is (super)symmetric over $\mathbb{K}$ with a (super)symmetrizing form $t$.} Let $\mathcal{B}$ be a $\Z_2$-homogeneous basis of $\mathcal{A}$ and $\mathcal{B}^{\vee}=\{b^{\vee}\mid b\in \mathcal{B}\}$ be the dual basis satisfying \eqref{dual}.

The following orthogonality relations will be used later.

\begin{lem}\cite[Remark 5.5]{WW2}
	Suppose $\mathcal{A}$ is symmetric with a symmetrizing form $t$. Let $V,V'\in \mathcal{A}$-smod be split two simple modules. Then \begin{equation}\label{orth:symm}
	\sum_{b\in\mathcal{B}}(\rho_V(b^\vee))_{ij}(\rho_{V'}(b))_{ks}=\begin{cases}\delta_{is}\delta_{jk} 2^{\delta_V}s_V,\qquad&{\text{if $V\cong V'$}}\\
		0; \qquad&{\text{otherwise}.}	
	\end{cases}
	 \end{equation}
\end{lem}

\begin{lem}\cite[proof of Lemma 2.7]{LS3}
	Suppose $\mathcal{A}$ is supersymmetric with a supersymmetrizing form $t$. Let $V,V'\in \mathcal{A}$-smod be two split simple modules of {\em type }\texttt{M}. Then  \begin{equation}\label{orth:supersymm}
	\sum_{b\in\mathcal{B}}(\rho_V(b^\vee))_{ij}(\rho_{V'}(b))_{ks}=\begin{cases}\delta_{is}\delta_{jk} \rm s_V,\qquad&{\text{if $V\cong V'$}}\\
		0; \qquad&{\text{otherwise}.}	
	\end{cases}
	 \end{equation}
\end{lem}

{\bf In the rest of this subsection, we assume $\mathcal{A}$ is symmetric with a symmetrizing form $t$.} As a usual symmetric $\mathbb{K}$-algebra $|\mathcal{A}|$, we recall the Schur elements defined by Geck and Pfeiffer.

Suppose ${V,V'}$ are two $|\mathcal{A}|$-modules. For any linear map $f\in{\rm Hom}_{{\rm R}}(V, {V'})$, we define $I'(f)\in {\rm Hom}_{{\rm R}}(V, { V'})$ by letting $$
I'(f)(v):=\sum_{b\in \mathcal{B}}b^{\vee} f(bv), 
$$ for $v\in V$.

\begin{lem}\cite[Lemma 7.1.10]{GF}\label{schur'}
	Suppose $\mathcal{A}$ is symmetric with a symmetrizing form $t$.	Let $V$ be a split simple $|\mathcal{A}|$-module. Then there exists a unique element $S_V\in \mathbb{K}$ such that for $f\in{{\rm End}_{\mathbb{K}}(V)}$,\begin{align}\label{schur2}I'(f)=S_V{\rm tr}(f){\rm id}_V,
	\end{align} Furthermore, $S_V$ depends only on the isomorphism class of $V$.
\end{lem}

Note that to distinguish these Schur elements, we use the notation $s_V$ for the Schur elements of simple module $V\in\mathcal{A}$-smod, while for simple module $V'\in|\mathcal{A}|$-mod, the Schur elements are denoted $S_{V'}$. We also have the orthogonality relations.

\begin{lem}\cite[Corollary 7.2.2]{GF}
	Suppose $\mathcal{A}$ is symmetric with a symmetrizing form $t$. Let $V,V'\in |\mathcal{A}|$-mod be two split simple modules. Then  \begin{equation}\label{orth:usual}
	\sum_{b\in\mathcal{B}}(\rho_V(b^\vee))_{ij}(\rho_{V'}(b))_{ks}=\begin{cases}\delta_{is}\delta_{jk} S_V,\qquad&{\text{if $V\cong V'$}.}\\
		0; \qquad&{\text{otherwise}}	
	\end{cases}
 \end{equation}
\end{lem}

\begin{cor}\label{the same schur}
	Let $V\in |\mathcal{A}|$-mod be a split simple $|\mathcal{A}|$-module. Then we have $S_{V}=S_{V^\delta}$.
	\end{cor}

\begin{proof}
	For any $f\in{\rm End}_{{\mathbb{K}}}(V)$, we use $I'(f)$ and ${I'}^\delta(f)$ to distinguish to maps $I'(f) \in{\rm End}_{|\mathcal{A}|}(V)$ with ${I'}^\delta(f) \in{\rm End}_{|\mathcal{A}|}(V^\delta)$. By definition, for any $f\in{\rm End}_{{\mathbb{K}}}(V)$, we can check that\begin{align*}
		{I'}^\delta(f)&=\sum_{b\in \mathcal{B}}(-1)^{\bar{b}}b^{\vee} f((-1)^{\bar{b}}bv)\\
		&=\sum_{b\in \mathcal{B}}b^{\vee} f(bv)={I'}(f)
		\end{align*} as linear transformations of $V$. Hence, by \eqref{schur2}, we have $S_{V}=S_{V^\delta}$.
	\end{proof}
	
Recall that any split simple module $V\in \mathcal{A}$-smod of {\em type }\texttt{Q} can be decomposed as $|V|\cong V^-\oplus V^+$, where $V^\pm\in |\mathcal{A}|$-mod are two non-isomorphic simple $|\mathcal{A}|$-modules.

\begin{prop}\label{Prop:comparing schur}
Let $V\in \mathcal{A}$-smod be a split simple module. We have the following.
\begin{enumerate}	
	\item If $V$ is of {\em type }\texttt{M}, then $s_V=S_{|V|}$.
	\item If $V$ is of {\em type }\texttt{Q}, then $2s_{V}=S_{V^-}=S_{V^+}$.
	\end{enumerate}
\end{prop}

\begin{proof}
	The first statement follows from the orthogonality relations \eqref{orth:symm} and \eqref{orth:usual}. We only consider the second statement. By Lemma \ref{lem:type MQ}, $V$ decomposes as $|V|\cong V^-\oplus V^+$, where $V^\pm\in |\mathcal{A}|$-mod are two non-isomorphic simple $|\mathcal{A}|$-modules and $(V^\pm)^\delta\cong V^\mp$. Now let $\rho_{V^\pm}$ be the matrix representation of $V^\pm$ with respect to any basis of $V^\pm$ and apply \eqref{orth:usual} and Corollary \ref{the same schur}, we have \begin{equation}\label{orth-matrix1}
	\sum_{b\in\mathcal{B}} \left (\left(\rho_{V^-}\oplus	\rho_{V^+}\right)(b^\vee)\cdot \left(\rho_{V^-}\oplus	\rho_{V^+}\right)(b)\right)=S_{V^-}\cdot {\rm id}_{2m},
	\end{equation} where $m=\dim V^-=\dim V^+$. Let $\rho_V$ be the matrix representation of $V$ with respect to any $\Z_2$-graded basis of $V$. Similarly, by \eqref{orth:symm}, we have \begin{equation}\label{orth-matrix2}
	\sum_{b\in\mathcal{B}} \left (\left(\rho_{V}\right)(b^\vee)\cdot \left(\rho_{V}\right)(b)\right)=2s_{V}\cdot {\rm id}_{2m}.
	\end{equation} Since the two representations $\rho_{V^-}\oplus	\rho_{V^+}$ and $\rho_V$ are equivalent, there exists an invertible $2m\times 2m$ matrix $P$, such that $\left(\rho_{V^-}\oplus	\rho_{V^+}\right)(a)=P^{-1}\rho_V(a)P$, for any $a\in \mathcal{A}$. This, combining with \eqref{orth-matrix1}, \eqref{orth-matrix2} and Corollary \ref{the same schur} implies $2s_{V}=S_{V^-}=S_{V^+}$.
	\end{proof}

\subsection{Modular Reduction}\label{Modular Reduction}
{\bf Throughout this subsection, $\mathscr{O}$ is a discrete valuation ring, $\mathscr{K}$ is the fraction field of $\mathscr{O}$ and $\mathbb{K}=\mathscr{O}/\mathfrak{m}$, where $\mathfrak{m}$ is the maximal ideal of $\mathscr{O}$. $\mathcal{A}$ is an $\hO$-algebra which is a free $\hO$-module of finite rank.} Recall the following notations $$
\mathscr{K}\mathcal{A}:=\mathscr{K}\otimes_{\mathscr{O}}\mathcal{A},\qquad\mathbb{K}\mathcal{A}:=\mathbb{K}\otimes_{\mathscr{O}}\mathcal{A}.
$$ 


 Let $V$ be a $|\mathscr{K}\mathcal{A}|$-module. As explained in \cite[7.3.7]{GF}, there is an $|\mathcal{A}|$-submodule $\tilde{V}\subset V$ such that 
 $\mathscr{K}\otimes_{\mathscr{O}}\tilde{V}\cong V$ and $\tilde{V}$ is free of finite rank as $\mathscr{O}$-module. In particular, if $V$ admits a $\Z_2$-grading, we can choose $\tilde{V}$ being $\Z_2$-graded too. The choice of $\tilde{V}$ might not be unique. {\bf From now on, we fix one choice of $\tilde{V}\in|\mathcal{A}|$-mod for each $V\in|\mathscr{K}\mathcal{A}|$-mod and we require $\tilde{V}$ being $\Z_2$-graded when $V$ admits a $\Z_2$-grading.}
Clearly, we have $
V= \mathscr{K}\tilde{V}$. We choose an $\mathscr{O}$-basis $\{v_1,\cdots,v_s\}$ of $\tilde{V}$, and denote by $\rho_V$ the corresponding matrix representation with respect to that basis. By our construction, for any $a\in\mathcal{A}$, $\rho_V(a)$ is a matrix whose entries belong to $\mathscr{O}$, i.e. \begin{align}\label{integral-matrix}\rho_V|_{\mathcal{A}}: \mathcal{A}\to \mathcal{M}_s(\mathscr{O}).\end{align} 
In conclusion, for any $a\in\mathcal{A}$, we have $$\chi_V(a)\in \mathscr{O}$$ and when $V$ is $\Z_2$-graded, we also have 
\begin{align}\label{integral-chi'}\chi'_V(a)\in \mathscr{O}.
	\end{align} 

 \begin{lem}\label{lem:super tits's}
 	Suppose $\mathbb{K}\mathcal{A}$ and $\mathscr{K}\mathcal{A}$ are split and $\mathbb{K}\mathcal{A}$ is semisimple. Then we have the following.
 	\begin{enumerate}
 		\item $\mathscr{K}\mathcal{A}$ is also semisimple and there is a bijection $${\rm Irr}(|\mathscr{K}\mathcal{A}|)\to{ \rm Irr}(|\mathbb{K}\mathcal{A}|),\, V\mapsto \mathbb{K}\tilde{V}.$$
 		\item  Suppose $V\in\mathscr{K}\mathcal{A}$-smod is a simple module of {\em type }\texttt{M}, then $\mathbb{K}\tilde{V}\in\mathbb{K}\mathcal{A}$-smod is also a simple module of {\em type }\texttt{M}.
 		\item Suppose $V\in\mathscr{K}\mathcal{A}$-smod is a simple module of {\em type }\texttt{Q}, then either $\mathbb{K}\tilde{V}$ is a simple  $\mathbb{K}\mathcal{A}$-module of {\em type }\texttt{Q} or both $\mathbb{K}\tilde{V^-}$ and $\mathbb{K}\tilde{V^+}$ have $\Z_2$-graded lift which are both simple of {\em type }\texttt{M}.
 		\end{enumerate}
 	
 	\end{lem}
 	
 	\begin{proof}
 	(1) follows from \cite[Theorem 7.4.6]{GF}. 
 	
 	(2). Suppose $V\in\mathscr{K}\mathcal{A}$-smod is a simple module of {\em type }\texttt{M}, then $V$ is already a $\Z_2$-graded simple module. Hence $\tilde{V}$ is also $\Z_2$-graded, so is $\mathbb{K}\tilde{V}$. Now apply (1), $|\mathbb{K}\tilde{V}|$ is already a simple $|\mathbb{K}\mathcal{A}|$-module. It follows that $\mathbb{K}\tilde{V}\in\mathbb{K}\mathcal{A}$-smod is also a simple module of {\em type }\texttt{M}. Hence we prove (2). 
 	
 	(3). Suppose $V\in\mathscr{K}\mathcal{A}$-smod is a simple module of {\em type }\texttt{Q}. Then we have \begin{align}\label{decomp1}
 	|\mathbb{K}\tilde{V}|=\mathbb{K}\tilde{V^-}\oplus \mathbb{K}\tilde{V^+}
 	\end{align}by construction. Moreover, $\mathbb{K}\tilde{V}\in\mathbb{K}\mathcal{A}$-smod and both $\mathbb{K}\tilde{V^-}$ and $\mathbb{K}\tilde{V^+}$ are two non-isomorphic simple $|\mathbb{K}\mathcal{A}|$-modules by (1). If $\mathbb{K}\tilde{V}$ is a simple $\mathbb{K}\mathcal{A}$-module, then it must be a simple module of {\em type }\texttt{Q} by \eqref{decomp1}. Otherwise, $\mathbb{K}\tilde{V}$ is a reducible $\mathbb{K}\mathcal{A}$-module since $\mathbb{K}\mathcal{A}$ is semisimple. Combining with \eqref{decomp1}, we conclude that both $\mathbb{K}\tilde{V^-}$ and $\mathbb{K}\tilde{V^+}$ have $\Z_2$-graded lifts. Clearly, these two graded lifts are simple of {\em type }\texttt{M}. This proves (3).

 		\end{proof}
 		
 		The following theorem is an analogue of Tits's deformation theorem in the super case.
 		
 		\begin{thm}\label{super tits's}Suppose $\mathbb{K}\mathcal{A}$ and $\mathscr{K}\mathcal{A}$ are split and $\mathbb{K}\mathcal{A}$ is semisimple. Then we have the following.
 			\begin{enumerate}
 				
 				\item Suppose all the simple $\mathscr{K}\mathcal{A}$-modules are of {\em type }\texttt{M}.  Then we have a bijection $${\rm Irr}(\mathscr{K}\mathcal{A})\to{ \rm Irr}(\mathbb{K}\mathcal{A}),\, V\mapsto \mathbb{K}\tilde{V}$$ and all the simple $\mathbb{K}\mathcal{A}$-modules are of {\em type }\texttt{M}.
 				\item Suppose $\mathscr{O}$ is complete. Then we have a bijection $${\rm Irr}(\mathscr{K}\mathcal{A})\to{ \rm Irr}(\mathbb{K}\mathcal{A}),\, V\mapsto \mathbb{K}\tilde{V}$$ which sends simple $\mathscr{K}\mathcal{A}$-module of {\em type }\texttt{M} to  simple $\mathbb{K}\mathcal{A}$-module of {\em type }\texttt{M} and simple $\mathscr{K}\mathcal{A}$-module of {\em type }\texttt{Q} to simple $\mathbb{K}\mathcal{A}$-module of {\em type }\texttt{Q}.
 				
 				\end{enumerate}
 			\end{thm}
 			
 			\begin{proof}
 			(1) follows from Lemma \ref{lem:super tits's} (1), (2).
 			
 			(2). Since $\mathscr{O}$ is complete. Applying \cite[Proposition 5.22, Theorem 6.7, \S6, Exercise 8]{CR}, we deduce that the number of primitive central idempotents in $\mathbb{K}\mathcal{A}$ is less than or equal to the number of primitive central idempotents in $\mathscr{K}\mathcal{A}$. Since both $\mathbb{K}\mathcal{A}$  and $\mathscr{K}\mathcal{A}$ are semisimple, we obtain $|{\rm Irr}(\mathscr{K}\mathcal{A})|\geq |{\rm Irr}(\mathbb{K}\mathcal{A})|$.  Suppose there is a simple  $\mathscr{K}\mathcal{A}$-module $V$ of {\em type }\texttt{Q} and $\mathbb{K}\tilde{V}$ is not simple. By Lemma \ref{lem:super tits's} (3), we have $|{\rm Irr}(\mathscr{K}\mathcal{A})|< |{\rm Irr}(\mathbb{K}\mathcal{A})|$. This leads to a contradiction and completes the proof of (2).
 				\end{proof}
 				
 			{\bf 	In the rest of this subsection, we assume that $\mathcal{A}$ is (super)symmetric with a (super)symmetrizing form $t$.} For simplicity, we denote $$
 		t^{\mathscr{K}}:=1_{\mathscr{K}}\otimes t: \mathscr{K}\mathcal{A}\to  \mathscr{K},\qquad	t^{\mathbb{K}}:=1_{\mathbb{K}}\otimes t: \mathbb{K}\mathcal{A}\to  \mathbb{K}.
 			$$  Then $\mathscr{K}\mathcal{A}$ (resp. $\mathbb{K}\mathcal{A}$) is (super)symmetric with a (super)symmetrizing form $t^{\mathscr{K}}$ (resp. $t^{\mathbb{K}}$). Let $\mathcal{B}$ be a $\Z_2$-homogeneous basis of $\mathcal{A}$ and $\mathcal{B}^{\vee}=\{b^{\vee}\mid b\in \mathcal{B}\}$ be the dual basis satisfying \eqref{dual}. 
 			
For any $V\in { \rm Irr}(\mathscr{K}\mathcal{A})$, we use $s_V$ to denote the Schur elements for $V$. 

The following Lemma gives integrality of the Schur elements.
 
 \begin{lem}
 	For any $V\in { \rm Irr}(\mathscr{K}\mathcal{A})$, we have $s_V\in \mathscr{O}$.
 	\end{lem}

 \begin{proof}
 This follows from \eqref{orth:symm}, \eqref{orth:supersymm} and \eqref{integral-matrix}.
 	\end{proof}
 
We denote ${\rm pr}:\mathscr{O}\to \mathbb{K}$ the natural projection. When $t$ is symmetric, we use $S_{V}$ to denote the Schur elements for any $V\in { \rm Irr}(|\mathscr{K}\mathcal{A}|)$. The following theorem was proved by Geck and Pfeiffer.
\begin{thm}\cite[Proposition 7.3.9, Theorem 7.4.7]{GF}\label{modular-semisimple 1}
	Suppose $\mathcal{A}$ is symmetric with a symmetrizing form $t$, and $|\mathbb{K}\mathcal{A}|$, $|\mathscr{K}\mathcal{A}|$ are split. Then for any $V\in { \rm Irr}(|\mathscr{K}\mathcal{A}|)$, we have $S_V\in \mathscr{O}$. Moreover, $|\mathbb{K}\mathcal{A}|$ is semisimple if and only if ${\rm pr}(S_V)\neq 0$ for any $V\in { \rm Irr}(|\mathscr{K}\mathcal{A}|)$.
	\end{thm}

Now we can prove Theorem \ref{modular-semisimple 2}, which can be viewed as the super analogue of Theorem \ref{modular-semisimple 1}.

	\medskip
{\bf Proof of Theorem  \ref{modular-semisimple 2}}: 
(1) We have the following. \begin{align*}
\mathbb{K}\mathcal{A} \text { is semisimple} &\Leftrightarrow |\mathbb{K}\mathcal{A}| \text{ is semisimple}  \\
&\xLeftrightarrow{\text{Theorem \ref{modular-semisimple 1}}}{\rm pr}(S_V)\neq 0 \text{ for any } V\in { \rm Irr}(|\mathscr{K}\mathcal{A}|)\,\\
&\xLeftrightarrow{\text{Lemma \ref{lem:type MQ} and Proposition \ref{Prop:comparing schur}}} {\rm pr}(s_V)\neq 0 \text{ for any } V\in { \rm Irr}(\mathscr{K}\mathcal{A}).
\end{align*} This proves (1).

(2). The proof is similar to Theorem \ref{modular-semisimple 1}. Suppose ${\rm pr}(s_V)\neq 0$ for any $V\in { \rm Irr}(\mathscr{K}\mathcal{A})$ and all simple modules of $\mathscr{K}\mathcal{A}$ are of {\em type }\texttt{M}. Then in particular, $s_V\neq 0$ for any $V\in { \rm Irr}(\mathscr{K}\mathcal{A})$. This implies $\mathscr{K}\mathcal{A}$ is semisimple by Theorem \ref{thm:semisimplicity criterion 1} (2). Applying Proposition \ref{schur formula 2}, we have  \begin{align*}
	t^{\mathscr{K}}=\sum_{\substack{V\in{\rm Irr}(\mathcal{A})}}\frac{1}{s_V}\chi'_V.
\end{align*} In particular, we have the following by restricting both sides on $\mathcal{A}$, \begin{align*}
t=\sum_{\substack{V\in{\rm Irr}(\mathcal{A})}}\frac{1}{s_V}(\chi'_V)|_{\mathcal{A}}.
\end{align*}By assumption, for any $V\in { \rm Irr}(\mathscr{K}\mathcal{A})$, we have $s_V\in \mathscr{O}^*$. Combining this with \eqref{integral-chi'}, we can reduce the above equation to $\mathbb{K}$:\begin{align}\label{modulo 1}
t^{\mathbb{K}}=\sum_{\substack{V\in{\rm Irr}(\mathcal{A})}}\frac{1}{{\rm pr}(s_V)}{\chi_V^{\mathbb{K}}}'
\end{align}
where ${\chi_V^{\mathbb{K}}}':=1_{\mathbb{K}}\otimes(\chi'_V)|_{\mathcal{A}}:\mathbb{K}\mathcal{A}\to \mathbb{K}.$ Clearly, ${\chi_V^{\mathbb{K}}}'=\chi'_{\mathbb{K}\tilde{V}}$. For any $x\in\mathcal{J}(\mathbb{K}\mathcal{A})_{\bar{0}}$ and $V\in{\rm Irr}(\mathcal{A})$, we have ${\chi_V^{\mathbb{K}}}'(x)=0$ since $x$ acts nilpotently on $\mathbb{K}\tilde{V}$. It follows from \eqref{modulo 1} and $t^{\mathbb{K}}(\mathbb{K}\mathcal{A}_{\bar{1}})=0$ that $t^{\mathbb{K}}(\mathcal{J}(\mathbb{K}\mathcal{A}))=0$. Note that $t^{\mathbb{K}}$ is non-degenerate and $\mathcal{J}(\mathbb{K}\mathcal{A})$ is a two-sided ideal of $\mathbb{K}\mathcal{A}$. We deduce $\mathcal{J}(\mathbb{K}\mathcal{A})=0$ and $\mathbb{K}\mathcal{A}$ is semisimple.

Conversely, assume $\mathbb{K}\mathcal{A}$ is semisimple. Applying Lemma \ref{lem:super tits's} (1) and Theorem \ref{thm:semisimplicity criterion 1} (2), we obtain $\mathscr{K}\mathcal{A}$ is also semisimple and all simple modules of $\mathscr{K}\mathcal{A}$ are of {\em type }\texttt{M}. Hence Theorem \ref{super tits's} gives a bijection $${\rm Irr}(\mathscr{K}\mathcal{A})\to{ \rm Irr}(\mathbb{K}\mathcal{A}),\, V\mapsto \mathbb{K}\tilde{V}.$$ 

Let $V,V'\in {\rm Irr}(\mathscr{K}\mathcal{A})$. By \eqref{integral-matrix}, we can similarly reduce the orthogonality relations \eqref{orth:supersymm} to $\mathbb{K}$: $$
	\sum_{b\in\mathcal{B}}(\rho_V^{\mathbb{K}}(b^\vee))_{ij}(\rho_{V'}^{\mathbb{K}}(b))_{ks}=\begin{cases}\delta_{is}\delta_{jk} \rm {\rm pr}(s_V),\qquad&{\text{if $V\cong V'$}}\\
		0; \qquad&{\text{otherwise},}	
	\end{cases} 
$$ where $\rho_V^{\mathbb{K}}:=1_{\mathbb{K}}\otimes (\rho_V)|_\mathcal{A}: \mathbb{K}\mathcal{A}\to \mathcal{M}_s(\mathbb{K}).$
It is clear that $\rho^{\mathbb{K}}_V$ is the matrix representation of $\mathbb{K}\tilde{V}$. Comparing with the orthogonality relations \eqref{orth:supersymm} between $\mathbb{K}\tilde{V}$ and $\mathbb{K}\tilde{V'}$, we deduce that for any $V\in{\rm Irr}(\mathscr{K}\mathcal{A})$, ${\rm pr}(s_V)=s_{\mathbb{K}\tilde{V}}$. Hence ${\rm pr}(s_V)\neq 0$ due to Theorem \ref{thm:semisimplicity criterion 1} (2) (apply to $\mathbb{K}\mathcal{A}$). This completes the proof of (2).

\qed
\medskip

\section{Cyclotomic Hecke-Clifford superalgebra}\label{Cyclotomic Hecke-Clifford superalgebra}
\subsection{Combinatorics}
\label{pag:The types of combinatorics}
In this subsection, we shall introduce combinatorics we will use in this paper. For $n\in \N$, let $\mathscr{P}_n$ be the set of partitions of $n$. For each $\lambda\in\mathscr{P}_n$, we denote by $\ell(\lambda)$ the number of nonzero parts in the partition $\lambda=(\lambda_1,\lambda_2,\cdots,\lambda_{\ell(\lambda)})$ and we set 
$$|\lambda|:=\sum_{i\geq 1}\lambda_i,\qquad n(\lambda):=\sum_{i\geq 1}(i-1)\lambda_i$$
and $$\lambda'_k:=\sharp \{i|i\geq 1 \text{ such that } \lambda_i\geq k\} \qquad \text{ for any $k\in \N$}.$$

Let $\mathscr{P}^m_n$ be the set of all $m$-multipartitions of $n$ for $m\geq 0$, where we use convention that $\mathscr{P}^0_n=\{\emptyset\}$. Let $\mathscr{P}^\mathsf{s}_n$ be the set of strict partitions of $n$. Then for $m\geq 0$, set
$$
\mathscr{P}^{\mathsf{s},m}_{n}:=
\cup_{a=0}^{n}( \mathscr{P}^{\mathsf{s}}_a\times \mathscr{P}^{m}_{n-a}),\qquad \mathscr{P}^{\mathsf{ss}, m}_{n}:=
\cup_{a+b+c=n}(\mathscr{P}^{\mathsf{s}}_a \times \mathscr{P}^{\mathsf{s}}_b\times \mathscr{P}^{m}_{c}).$$
We will formally write  $\mathscr{P}^{\mathsf{0},m}_{n}=\mathscr{P}^m_n$.  In convention, for any \label{pag:multipartition} $\undla\in  \mathscr{P}^{\mathsf{0},m}_{n}$, we write  $\undla=(\lambda^{(1)},\cdots,\lambda^{(m)}),$ while for any $\undla\in  \mathscr{P}^{\mathsf{s},m}_{n}$, we write  $\undla=(\lambda^{(0)},\lambda^{(1)},\cdots,\lambda^{(m)})$, i.e., we shall put the strict partition in the $0$-th component. Moreover, for any $\undla\in  \mathscr{P}^{\mathsf{ss},m}_{n}$, we write  $\undla=(\lambda^{(0_-)},\lambda^{(0_+)},\lambda^{(1)},\cdots,\lambda^{(m)})$, i.e., we shall put two strict partitions in the $0_-$-th component and the $0_+$-th component.

We will also identify the (strict) partition with the corresponding (shifted) young diagram.  For any  $\undla\in\mathscr{P}^{\bullet,m}_{n}$ with $\bullet\in\{\mathsf{0},\mathsf{s},\mathsf{ss}\}$ and $m\in \N$, the box in the $l$-th component with row $i$, column $j$ will be denoted by $(i,j,l)$  with $l\in\{1,2,\ldots,m\},$ or $l\in\{0,1,2,\ldots,m\}$ or $l\in\{0_-,0_+,1,2,\ldots,m\}$ in the case $\bullet=\mathsf{0},\mathsf{s},\mathsf{ss},$ respectively. We also use the notation $\alpha=(i,j,l)\in \undla$ if the diagram of $\undla$ has a box $\alpha$ on the $l$-th component of row $i$ and column $j$. We use $\Std(\undla)$ \label{pag: standard tableaux} to denote the set of standard tableaux of shape $\undla$. One can also regard each $\mathfrak{t}\in\Std(\undla)$ as a bijection $\mathfrak{t}:\undla\rightarrow \{1,2,\ldots, n\}$ satisfying $\mathfrak{t}((i,j,l))=k$ if the box occupied by $k$ is located in the $i$-th row, $j$-th column in the $l$-th component $\lambda^{(l)}$. We use $\mathfrak{t}^{\undla}$ (resp. $\mathfrak{t}_{\undla}$) to denote the standard tableaux obtained by inserting the symbols $1,2,\ldots,n$ consecutively by rows (resp. column) from the first (resp. last) component of $\undla$.

We use $\Add(\undla)$ and $\Rem(\undla)$ to denote the set of addable boxes of $\undla$ and the set of removable boxes of $\undla$, respectively. For $\mt\in\Std(\undla),$ we define $\Add(\mt):=\Add(\undla)$ and $\Rem(\mt):=\Rem(\undla).$ \label{pag:Add and Rem}

\label{pag:diag of undlam}
\begin{defn}(\cite[Definition 2.5]{SW})
	Let $\undla\in\mathscr{P}^{\bullet,m}_{n}$ with $\bullet\in\{\mathsf{0},\mathsf{s},\mathsf{ss}\}$.  We define
	$$
	\mathcal{D}_{\undla}:=\begin{cases} \emptyset,&\text{if $\undla\in\mathscr{P}^{\mathsf{0},m}_n$,}\\
		\{(a,a,0)|(a,a,0)\in \undla,\,a\in\N\}, &\text{if $\undla\in\mathscr{P}^{\mathsf{s},m}_{n}$,}\\
		\big\{(a,a,l)|(a,a,l)\in \undla,\,a\in\N, l\in\{0_-,0_+\}\big\}, &\text{if $\undla\in\mathscr{P}^{\mathsf{ss}, m}_{n}.$}
	\end{cases}
	$$
	For any $\mathfrak{t}\in\Std(\undla), $ we define
	\begin{align}\label{Dt}
		\mathcal{D}_{\mt}:=\{\mathfrak{t}(a,a,l)|(a,a,l)\in\mathcal{D}_{\undla}\}.
	\end{align}
\end{defn}

\begin{example} Let $\undla=(\lambda^{(0)},\lambda^{(1)})\in \mathscr{P}^{\mathsf{s},1}_{5}$, where via the identification with strict Young diagrams and Young diagrams:
	$$
	\lambda^{(0)}=\young(\,\, ,:\,), \lambda^{(1)}=\young(\,,\,).
	$$
	Then
	$$
	\mathfrak{t}^{\undla}=\Biggl(\young(12,:3), \young(4,5)\Biggr).
	$$
	and  an example of standard tableau is as follows:
	$$
	\mathfrak{t}=\Biggl(\young(13,:5), \young(2,4)\Biggr)\in \Std(\undla).
	$$ We have $$\mathcal{D}_{\undla}=\{(1,1,0),(2,2,0)\},\qquad \mathcal{D}_{\mathfrak{t}}=\{1,5\}.
	$$
\end{example}

	\subsection{Affine Hecke-Clifford superalgebra ${\mathcal{H}_{\rm R}(n)} $ }\label{basic-Non-dege}
	
{\bf In this subsection, we fix $q\in{\rm R^\times}\setminus\{\pm 1\}$ such that $q+q^{-1}\in{\rm R^\times}.$}
    We define $\epsilon:=q-q^{-1}\in{\rm R}\setminus\{0\}.$
    The non-degenerate affine Hecke-Clifford superalgebra $\mathcal{H}_{\rm R}(n)$ \label{pag:AHCA} is
	the superalgebra over ${\rm R}$ generated by even generators
	$T_1,\ldots,T_{n-1},$ $X_1^{\pm 1},\ldots,X_n^{\pm 1}$ and odd generators
	$C_1,\ldots,C_n$ subject to the following relations
	\begin{align}
		T_i^2=\epsilon T_i +1,\quad T_iT_j =T_jT_i, &\quad
		T_iT_{i+1}T_i=T_{i+1}T_iT_{i+1}, \quad|i-j|>1,\label{Braid}\\
		X_iX_j&=X_jX_i, X_iX^{-1}_i=X^{-1}_iX_i=1, \quad 1\leq i,j\leq n, \label{Poly}\\
		C_i^2=1,C_iC_j&=-C_jC_i, \quad 1\leq i\neq j\leq n, \label{Clifford}\\
		T_iX_i&=X_{i+1}T_i-\epsilon(X_{i+1}+C_iC_{i+1}X_i),\label{PX1}\\
		T_iX_{i+1}&=X_iT_i+\epsilon(1-C_iC_{i+1})X_{i+1},\label{PX2}\\
		T_iX_j&=X_jT_i, \quad j\neq i, i+1, \label{PX3}\\
		T_iC_i=C_{i+1}T_i, T_iC_{i+1}&=C_iT_i-\epsilon(C_i-C_{i+1}),T_iC_j=C_jT_i,\quad j\neq i, i+1, \label{PC}\\
		X_iC_i=C_iX^{-1}_i, X_iC_j&=C_jX_i,\quad 1\leq i\neq j\leq n.
		\label{XC}
	\end{align}
	
	For $\alpha=(\alpha_1,\ldots,\alpha_n)\in\mathbb{Z}^n$ and
	$\beta=(\beta_1,\ldots,\beta_n)\in\mathbb{Z}_2^n$, we set
	$X^{\alpha}=X_1^{\alpha_1}\cdots X_n^{\alpha},$
	$C^{\beta}=C_1^{\beta_1}\cdots C_n^{\beta_n}$ and define
	$\supp(\beta):=\{1 \leq k \leq n:\beta_{k}=\bar{1}\},$ $|\beta|:=\Sigma_{i=1}^{n}\beta_i \in \mathbb{Z}_2.$ \label{pag:suppot and sum}
	Then we have the
	following.
	\begin{lem}\cite[Theorem 2.2]{BK}\label{lem:PBWNon-dege}
		The set $\{X^{\alpha}C^{\beta}T_w~|~ \alpha\in\mathbb{Z}^n,
		\beta\in\mathbb{Z}_2^n, w\in \mathfrak{S}_n\}$ forms an ${\rm R}$-basis of $\mathcal{H}_{\rm R}(n).$
	\end{lem}
	
%

For any invertible element $\iota \in {\rm R}$, we define
\label{pag:q-function and b-function}
$$
	\mathtt{q}(\iota):=2\frac{q\iota+(q\iota)^{-1}}{q+q^{-1}}.
	$$

	\subsection{Cyclotomic Hecke-Clifford superalgebra $\mathcal{H}^f_{\rm R}(n)$}\label{basic-cyc-Non-dege}

	To define the cyclotomic Hecke-Clifford superalgebra $\mathcal{H}^f_{\rm R}(n)$ over ${\rm R},$ \label{pag:CHCA} we fix $m\geq 0,$ $\underline{Q}=(Q_1,Q_2,\ldots,Q_m)\in({\rm R^\times})^m$ \label{pag:Q-parameters} and take a $f=f(X_1)\in {\rm R}[X_1^\pm]$ satisfying \cite[(3.2)]{BK}.
It is noted in \cite{SW} that we only need to consider $f(X_1)\in {\rm R}[X_1^\pm]$ to be one of the following three forms:
 $$\begin{aligned}
 f=\begin{cases}
     f^{\mathsf{(0)}}_{\underline{Q}}=\prod_{i=1}^m \biggl(X_1+X^{-1}_1-\mathtt{q}(Q_i)\biggr), \\ 
      f^{\mathsf{(s)}}_{\underline{Q}}=(X_1-1)\prod_{i=1}^m \biggl(X_1+X^{-1}_1-\mathtt{q}(Q_i)\biggr), \\  
     f^{\mathsf{(ss)}}_{\underline{Q}} = (X_1-1)(X_1+1)\prod_{i=1}^m \biggl(X_1+X^{-1}_1-\mathtt{q}(Q_i)\biggr).
    \end{cases}
    \end{aligned}$$
In each case, the degree $r$ of the polynomial $f$ is $2m,\,2m+1,\,2m+2$ respectively. We will see that the different choices of $f\in \{f^{\mathsf{(0)}}_{\underline{Q}},\,f^{\mathsf{(s)}}_{\underline{Q}},\,f^{\mathsf{(ss)}}_{\underline{Q}}\}$ corresponds to different combinatorics $\mathscr{P}^{\mathsf{0},m}_{n},\mathscr{P}^{\mathsf{s},m}_{n},\mathscr{P}^{\mathsf{ss},m}_{n}$ respectively in the representation theory of $\mathcal{H}^f_{\rm R}(n)$.
	
	The non-degenerate cyclotomic Hecke-Clifford superalgebra $\mathcal{H}^f_{\rm R}(n)$ is defined as $$\mathcal{H}^f_{\rm R}(n):=\mathcal{H}_{\rm R}(n)/\mathcal{I}_f,
	$$ where $\mathcal{I}_f$ is the two sided ideal of $\mathcal{H}_{\rm R}(n)$ generated by $f(X_1)$. The degree $r$ \label{pag:nondege level} of $f$ is called the level of $\mathcal{H}^f_{\rm R}(n).$ We shall denote the images of $X^{\alpha}, C^{\beta}, T_w$ in the cyclotomic quotient $\mathcal{H}^f_{\rm R}(n)$ still by the same symbols. Then we have the following due to \cite{BK}.
	
	\begin{lem}\cite[Theorem 3.6]{BK}\label{basis}
		The set $\{X^{\alpha}C^{\beta}T_w~|~ \alpha\in\{0,1,\cdots,r-1\}^n,
		\beta\in\mathbb{Z}_2^n, w\in {\mathfrak{S}_n}\}$ forms an ${\rm R}$-basis of $\mathcal{H}^f_{\rm R}(n)$.
	\end{lem}

We set $Q_0=Q_{0_+}=1, Q_{0_-}=-1$. Now we can define the residues of boxes.
\begin{defn}
	Suppose  $\undla\in\mathscr{P}^{\bullet,m}_{n}$ with $\bullet\in\{\mathsf{0},\mathsf{s},\mathsf{ss}\}$ and $(i,j,l)\in \undla,$ we define the residue of box $(i,j,l)$ with respect to the parameter $\undQ=(Q_1,Q_2,\ldots,Q_m)\in({\rm R^\times})^m$ as follows:
	\label{pag:nondeg residue}
	\begin{equation}\label{eq:residue}
		\res(i,j,l):=Q_lq^{2(j-i)}\in {\rm R^\times}.
	\end{equation}
	If $\mathfrak{t}\in \Std(\undla)$ and $\mathfrak{t}(i,j,l)=k$, we set
	\begin{align}
		\res_\mathfrak{t}(k)&:=Q_lq^{2(j-i)}\in {\rm R^\times},\label{resNon-dege-1}\\
		\res(\mathfrak{t})&:=(\res_\mathfrak{t}(1),\cdots,\res_\mathfrak{t}(n))\in ({\rm R^\times})^m,\label{resNon-dege-2}
	\end{align}
	then the $\mathtt{q}$-sequence of $\mt$ is
	\begin{align}\label{resNon-dege-3}
		\mathtt{q}(\res(\mathfrak{t})):=(\mathtt{q}(\res_{\mathfrak{t}}(1)), \mathtt{q}(\res_{\mathfrak{t}}(2)),\ldots, \mathtt{q}(\res_{\mathfrak{t}}(n)))\in {\rm R}^m.
	\end{align}
\end{defn}

The following Frobenius form is due to \cite{BK}.

\begin{prop}\cite[Corollary 3.14]{BK}, \cite[Proposition 5.4]{LS3}\label{trace formula}
Let $\alpha=(\alpha_1,\ldots,\alpha_n)\in [0,r-1]^n,$ $\beta=(\beta_1,\ldots,\beta_n)\in\mathbb{Z}_2^n$ and $w\in \mathfrak{S}_n,$
then the map
\begin{align}\label{closed formula frob}
	\tau^{\rm R}_{r,n}( X^\alpha C^\beta T_w ):=\delta_{(\alpha,\beta,w),(0,0,1)}
\end{align}
is a Frobenius form on $\mathcal{H}^f_{\rm R}(n).$
\end{prop}

\begin{prop}\cite[Theorem 1.2]{LS3}\label{Nondengerate}
	
	(i) If $\bullet=\mathtt{0},$ then the cyclotomic Hecke-Clifford superalgebra $\mathcal{H}^f_{\rm R}(n)$ is supersymmetric with the supersymmetrizing form
	$$t^{\rm R}_{r,n}:=\tau^{\rm R}_{r,n}\Bigl(-\cdot (X_1X_2\cdots X_n)^{m} \Bigr);$$
	
   (ii) If $\bullet=\mathtt{s},$ and $(1+X_1)(1+X_2)\cdots (1+X_n)$ is invertible in $\mathcal{H}^f_{\rm R}(n),$ then the cyclotomic Hecke-Clifford superalgebra $\mathcal{H}^f_{\rm R}(n)$ is symmetric with the symmetrizing form
	$$t^{\rm R}_{r,n}:=\tau^{\rm R}_{r,n}\Bigl(-\cdot (X_1X_2\cdots X_n)^{m} (1+X_1)(1+X_2)\cdots (1+X_n)\Bigr).$$
\end{prop}


\begin{cor}\label{symmetric over ring}
	Let $\bullet=\mathtt{s}$ and $\mathscr{K}$ be the fraction field of ${\rm R}$. Suppose for any $k\in[n]$, the following hold.
	\begin{enumerate}
		\item The characteristic polynomial of $X_k$ on $\mathcal{H}^f_{\mathscr{K}}(n)$ splits.
		
		\item Each eigenvalue of $1+X_k$ on $\mathcal{H}^f_{\mathscr{K}}(n)$ is an invertible element in ${\rm R}$.
		\end{enumerate} Then the cyclotomic Hecke-Clifford superalgebra $\mathcal{H}^f_{\rm R}(n)$ is symmetric with the symmetrizing form $t^{\rm R}_{r,n}$.
		\end{cor}
\begin{proof}
Let $k\in[n]$. We deduce from the conditions above that the coefficients of the characteristic polynomial of $1+X_k$ on $\mathcal{H}^f_{\mathscr{K}}(n)$ belong to ${\rm R}$ and that the constant coefficient is invertible in ${\rm R}$. Hence $1+X_k$ is invertible in $\mathcal{H}^f_{\rm R}(n)$. The corollary follows from Proposition \ref{Nondengerate} (ii). 
	\end{proof}
	
\subsection{Separate condition and semisimple representation}
Recall $[n]=\{1,2,\ldots,n\}$ for $n\in\mathbb{N}.$ {\bf In this subsection, ${\rm R}=\mathbb{K}$ is the algebraically closed field of characteristic different from $2$.} For any $a\in \mathbb{K}$, we fix a solution of the equation $x^2=a$ and denote it by $\sqrt{a}$.

	For  $\iota\neq 0 \in \mathbb{K}$, we define
$$
\mathtt{b}_{\pm}(\iota):=\frac{\mathtt{q}(\iota)}{2}\pm \sqrt{\frac{\mathtt{q}(\iota)^2}{4}-1}.
$$
Clearly, $\mathtt{b}_{\pm}(\iota)$ is the solution of equation $x+x^{-1}=\mathtt{q}(\iota),$ thus $\mathtt{b}_{+}(\iota)\mathtt{b}_{-}(\iota)=1.$
We define the deformed quantum integers  as\begin{align}\label{deformed quantum number}
	\left[ m \right]_{l,q^2}:=\frac{Q_lq^{2m}-Q_l^{-1}q^{-2m}}{q^{2}-q^{-2}},\quad m\in \mathbb{Z}.\end{align}
Then
\begin{align}
	\mb_\pm(\res_\mathfrak{t}(k))=[j-i+1]_{l,q^2}-[j-i]_{l,q^2}\pm\epsilon\sqrt{[j-i+1]_{l,q^2}[j-i]_{l,q^2}}\label{eigenvalues}
\end{align}
for each $k=\mt(i,j,l).$

%

Recall that $\undQ=(Q_1,\ldots,Q_m)\in (\mathbb{K}^*)^n$ and  $\pm 1\neq q \in \mathbb{K}^*$. Then for any $n\in \N$, we define $P^{(\bullet)}_{n}(q^2,\undQ)$ as follows:

	$$\begin{aligned}
		P_n^{(\bullet)}(q^2,\undQ):=\begin{cases}
			\prod\limits_{t=1}^{n}\bigl(q^{2t}-1\bigr)\prod\limits_{i=1}^m\biggl(\prod\limits_{t=3-n}^{n-1}\bigl(Q^2_i-q^{-2t}\bigr)\prod\limits_{t=1-n}^{n}\bigl(Q^2_i-q^{-4t}\bigr)\biggr)&\\
			\cdot \prod\limits_{1\leq i<i'\leq m}\biggl(\prod\limits_{t=1-n}^{n-1}\bigl({Q_i}-Q_{i'}q^{-2t}\bigr)
			\bigl(Q_iQ_{i'}-q^{-2(t+1)}\bigr)\biggr), \quad \mbox{if $\bullet=\mathsf{0}$ };  \\
			&\\
			\prod\limits_{t=1}^{n}\biggl(\bigl(q^{2t}-1\bigr)\bigl(q^{2t}+1\bigr)\biggr)\prod\limits_{i=1}^m\biggl(\prod\limits_{t=3-n}^{n-1}\bigl(Q^2_i-q^{-2t}\bigr)\prod\limits_{t=1-n}^{n}\bigl(Q^2_i-q^{-4t}\bigr)\biggr)&\\
			\cdot \prod\limits_{1\leq i<i'\leq m}\biggl(\prod\limits_{t=1-n}^{n-1}\bigl({Q_i}-Q_{i'}q^{-2t}\bigr)
			\bigl(Q_iQ_{i'}-q^{-2(t+1)}\bigr)\biggr),  \quad \mbox{if $\bullet=\mathsf{s}$ or  $\mathsf{ss},$}\\
		\end{cases} 
	\end{aligned}$$ where for  $n=1,$  the product $\prod\limits_{t=3-n}^{n-1}\bigl(Q^2_i-q^{-2t}\bigr)$ is understood to be $1$.

%

	\begin{thm}\label{semisimple:non-dege}\cite[Theorem 4.10]{SW}
	Let $\undQ=(Q_1,Q_2,\ldots,Q_m)\in(\mathbb{K}^*)^m$.  Assume $f=f^{(\bullet)}_{\undQ}$ with $\bullet\in\{\mathtt{0},\mathtt{s},\mathtt{ss}\}$. We have the following. \begin{enumerate}
		\item  If $P^{(\bullet)}_{n}(q^2,\undQ)\neq 0$,  then $\mathcal{H}^f_{\mathbb{K}}(n)$ is a (split) semisimple superalgebra. 
\item 	If $P^{(\bullet)}_{n}(q^2,\undQ)\neq 0$,  then $$
	\{\mathbb{D}(\undla)\mid \undla\in\mathscr{P}^{\bullet,m}_{n}\}$$ forms a complete set of pairwise non-isomorphic irreducible $\mathcal{H}^f_{\mathbb{K}}(n)$-modules. 
Moreover, $\mathbb{D}(\undla)$ is of type  $\texttt{M}$ if and only if $\sharp \mathcal{D}_{\undla}$  is even and is of type  $\texttt{Q}$ if and only if $\sharp \mathcal{D}_{\undla}$  is odd.
	\end{enumerate}
\end{thm}
We shall prove the converse of Theorem \ref{semisimple:non-dege} (1) when $\mathcal{H}^f_{\mathbb{K}}(n)$ is (super)symmetric in Section \ref{Semisimplicity criterion for cyclotomic Hecke-Clifford superalgebra}.

\section{Some formulae for Schur elements}\label{Some formulae for Schur elements}
{\bf Throughout this section, ${\rm R}=\mathbb{K}$ is the algebraically closed field of characteristic different from $2$. We shall fix the parameter $\undQ=(Q_1,Q_2,\ldots,Q_m)\in(\mathbb{K}^*)^m$ and $f=f^{(\bullet)}_{\undQ}$ with $P^{(\bullet)}_{n}(q^2,\undQ)\neq 0$ for $\bullet\in\{\mathtt{0},\mathtt{s}\}.$} Accordingly, we define the residues of boxes in the young diagram $\undla$ via \eqref{eq:residue} as well as $\res(\mathfrak{t})$ for each $\mathfrak{t}\in\Std(\undla)$ with $\undla\in\mathscr{P}^{\bullet,m}_{n}$ with $m\geq 0.$

\subsection{Recursive formula for Schur elements}

\begin{defn}\label{diagonalnodes}
	For any $\bullet\in\{\mathsf{0},\mathsf{s}\},$ we denote the subset of boxes
	\begin{align}
		\mathcal{D}
		=\mathcal{D}^{(\bullet)}
		:=
		\begin{cases}
			\emptyset, \quad &\text{ if $\bullet=\mathsf{0},$}\\
			\{(i,i,0)\mid i\in\mathbb{Z}_{>0}\}, \quad &\text{ if $\bullet=\mathsf{s},$}
		\end{cases}
	\end{align}
\end{defn}
\begin{lem}\cite[Corollary 3.25]{LS3}
	Let $\undla\in\mathscr{P}^{\mathsf{\bullet},m}_{n}$ for $\bullet\in\{\mathsf{0},\mathsf{s}\}$ and $\mt\in\Std(\undla).$
	Then the following element
	\begin{align}\label{def qt}
		q(\undla)
		:=\prod\limits_{k=1}^{n}\frac{\prod\limits_{\beta\in\Rem(\mt\downarrow_{k-1})\setminus \mathcal{D}}\left(\mathtt{q}(\res_{\mt}(k))-\mathtt{q}(\res(\beta))\right)}{
			\prod\limits_{\alpha\in\Add(\mt\downarrow_{k-1})\setminus \{\mt^{-1}(k)\}}\left(\mathtt{q}(\res_{\mt}(k))-\mathtt{q}(\res(\alpha))\right)}\in\mathbb{K}^*
	\end{align}
	is independent to the choice of $\mt.$
\end{lem}

\begin{prop}\cite[Theorem 1.4]{LS3}\label{Schur-Non-dengerate}
	Suppose that $\bullet\in\{\mathsf{0},\mathsf{s}\}$ and  $P^{(\bullet)}_{n}(q^2,\undQ)\neq 0$ holds. Let $\undla\in\mathscr{P}^{\bullet,m}_{n},$ then the Schur element $s_{\undla}$ of simple $\mathcal{H}^f_{\mathbb{K}}(n)$-module $\mathbb{D}(\undla)$ with respect to $t^{\mathbb{K}}_{r,n}$ is given by
	\begin{align*}
		s_{\undla}
		=\begin{cases}
			\prod\limits_{\alpha\in\undla}\left(\mathtt{b}_{-}(\res(\alpha))-\mathtt{b}_{+}(\res(\alpha))\right)
			\cdot q(\undla)^{-1},& \text{ if } \bullet=\mathsf{0},\\
			2^{-\lceil \sharp\mathcal{D}_{\undla}/2 \rceil}	\cdot q(\undla)^{-1},& \text{ if } \bullet=\mathsf{s}.
		\end{cases}
	\end{align*}
\end{prop}

\subsection{Closed formula for Schur elements}\label{Closed formula for Schur elements}

In this subsection, we deduce a closed formula for Schur elements. To this end, by Proposition \ref{Schur-Non-dengerate}, we only need to give a closed formula for $q(\undla)^{-1}$.
We first recall the (generalized and shifted) hook lengths for partitions and strict partitions.

Let $\lambda\in\mathscr{P}_n$, $x=(i,j)\in \lambda$ and $\mu$ be another partition of $n'\in \N$. We define the generalized hook length
of $x$ with respect to $(\lambda,\mu)$ to be the integer: $$
h^{\lambda,\mu}_{i,j}:=\lambda_i-i+\mu'_j-j+1.
$$For $\lambda=\mu$, the above formula becomes the classical hook length formula and will be denoted $h^{\lambda}_{i,j}$.

Following \cite{WW2}, for any partition $\lambda\in\mathscr{P}_n$ whose main diagonal of the Young
diagram $\lambda$ contains $r$ cells. Let $\alpha_i=\lambda_i-i$ be the
number of cells in the $i$th row of $\lambda$ strictly to the right of
$(i,i)$, and let $\beta_i=\lambda_i'-i$ be the number of cells in the
$i$th column of $\lambda$ strictly below $(i,i)$, for $1\leq i\leq r$.
We have $\alpha_1>\alpha_2>\cdots>\alpha_r\geq0$ and
$\beta_1>\beta_2>\cdots>\beta_r\geq0$. Then the {\bf Frobenius notation}
for a partition is
$\lambda=(\alpha_1,\ldots,\alpha_r|\beta_1,\ldots,\beta_r)$. For
example, if $\lambda=(4,3,2,1)$ whose corresponding Young diagram is
$$
\mu =\young(\,\,\,\,,\,\,\,,\,\,,\,)
$$
then $\alpha =(3,1),
\beta=(3,1)$ and hence $\mu=(3,1|3,1)$ in Frobenius
notation.

Let $\mu$ be a strict partition of $n$. We have identified $\mu$ with its shifted young diagram. Denoting
$\ell(\mu)=\ell$, we define the {\bf double partition}
$\widetilde{\mu}$ to be $\widetilde{\mu}=(\mu_1,\ldots,\mu_\ell|
\mu_1-1,\mu_2-1,\ldots,\mu_\ell-1)$ in Frobenius notation. Then we can
identify the shifted diagram $\mu$ with the part of
$\widetilde{\mu}$ that lies strictly above the main diagonal. For each cell
$x=(i,j)\in \mu$, denote by $h^{\widetilde{\mu}}_{ij}$ the associated hook length in
the Young diagram $\widetilde{\mu}$. $h^{\widetilde{\mu}}_{ij}$ is called the shifted hook length for box $x$ in $\mu$.

\begin{example}
	Let $\mu= (3,2)$. The corresponding shifted diagram $\mu$ and
	double diagram $\widetilde\mu$ are
	$$
	\mu=\young(\,\,\,,:\,\,)
	\qquad \qquad
	\widetilde{\mu}=\young(\,\,\,\,,\,\,\,\,,\,\,)
	$$
	
	Then by definition,  the shifted hook lengths for each cell in $\mu$ can be computed as follows $$h^{\widetilde{\mu}}_{1,1}=5,\,h^{\widetilde{\mu}}_{1,2}=3,\,h^{\widetilde{\mu}}_{1,3}=2,\,h^{\widetilde{\mu}}_{2,2}=2,\,h^{\widetilde{\mu}}_{2,3}=1,$$
	which are  the
	usual hook lengths for the corresponding cell in the double diagram $\widetilde \mu.$
\end{example}

\begin{defn}
	Let $\undla\in\mathscr{P}^{\mathsf{\bullet},m}_{n}$, where $\bullet\in\{\mathtt{0},\mathtt{s}\}$. For $1\leq c\leq m$, we define \begin{align*}
	\mathbb{Y}^{\undla}_{c}:&=
		\frac{1}{q^{2n\left(\lambda^{(c)}\right)}}\cdot\left( \prod_{\substack{1\leq i<i'\leq \ell(\lambda^{(c)})}}\left(\frac{Q_c^2q^{2\left(\lambda^{(c)}_i-i+\lambda^{(c)}_{i'}-i'+2\right)}-1}{Q_c^2q^{2\left(2-i-i'\right)}-1}\right)\right)\\
		&\qquad\qquad\cdot \prod_{(a,b,c)\in\lambda^{(c)}}\left(\frac{Q_c^2q^{2\left(b-a-\ell(\lambda^{(c)})+1\right)}-1}{Q_c^2q^{4\left(b-a\right)}-1}
		\cdot \frac{q^{2h^{\lambda^{(c)}}_{a,b}}-1}{q^2-1}\right).
	\end{align*}
	When $\bullet=\mathtt{s}$, then we additionally define \begin{align*}
	\mathbb{Y}^{\undla}_{0}:=2^{\ell(\lambda^{(0)})}\cdot \frac{\left(\prod\limits_{1\leq i<i'\leq \ell(\lambda^{(0)})+1}\left(q^{2\left(\lambda^{(0)}_i+\lambda^{(0)}_{i'}\right)}-1\right)\right)\cdot \left(\prod\limits_{j=\ell(\lambda^{(0)})+1}^{\lambda^{(0)}_1}\prod\limits_{(i,j,0)\in\lambda^{(0)}}\left(q^{2h_{i,j}^{\widetilde{\lambda^{(0)}}}}-1\right)\right)}{q^{2n\left(\lambda^{(0)}\right)}\cdot (q^2-1)^n\cdot\prod\limits_{(i,j,0)\in\lambda^{(0)}}\left(q^{2(j-i)}+1\right) }.
\end{align*} 
\end{defn}

\begin{defn}
	Let $\undla\in\mathscr{P}^{\mathsf{\bullet},m}_{n}$, where $\bullet\in\{\mathtt{0},\mathtt{s}\}$. 
	For $1\leq c_1<{c_2}\leq m$, we define \begin{align*}
	 \mathbb{X}^{\undla}_{{c_1},{c_2}}:&=\prod_{\substack{(a,b,{c_1})\in\lambda^{({c_1})}}}\prod_{k=1}^{\lambda^{(c_2)}_1}\left(\frac{Q_{c_1}q^{2(b-a)}-Q_{{c_2}}q^{2(k-{\lambda^{(c_2)}_k}'-1)}}{Q_{c_1}q^{2(b-a)}-Q_{{c_2}}q^{2(k-{\lambda^{(c_2)}_k}')}} \cdot\frac{Q_{c_1}Q_{c_2}q^{2(b-a+k-{\lambda^{(c_2)}_k}')}-1}{Q_{c_1}Q_{c_2}q^{2(b-a+k-{\lambda^{(c_2)}_k}'+1)}-1}\right) \\
	 &\qquad\qquad\cdot\prod_{\substack{(a,b,c_1)\in\lambda^{(c_1)}}}\left(\frac{2\left(Q_{c_1}Q^{-1}_{c_2}q^{2(b-a-\lambda_1^{(c_2)})}-1\right)\left(Q_{c_1}Q_{{c_2}}q^{2(b-a+\lambda^{(c_2)}_1+1)}-1\right)}{\left(q^2+1\right)Q_{c_1}q^{2(b-a-\lambda_1^{(c_2)})}}\right)\\
	&\qquad\qquad\qquad  \cdot\prod_{\substack{(a,b,c_2)\in\lambda^{(c_2)}}}\left(\frac{2\left(Q_{c_2}Q^{-1}_{c_1}q^{2(b-a)}-1\right)\left(Q_{c_1}Q_{{c_2}}q^{2(b-a+1)}-1\right)}{\left(q^2+1\right)Q_{c_2}q^{2(b-a)}}\right).
	 \end{align*}
	When $\bullet=\mathtt{s}$, then for $1\leq {c}\leq m$, we additionally define
\begin{align*}
	 	\mathbb{X}^{\undla}_{{0},{c}}:&=\prod_{\substack{(a,b,{0})\in\lambda^{({0})}}}\prod_{k=1}^{\lambda^{(c)}_1}\left(\frac{q^{2(b-a)}-Q_{{c}}q^{2(k-{\lambda^{(c)}_k}'-1)}}{q^{2(b-a)}-Q_{{c}}q^{2(k-{\lambda^{(c)}_k}')}} \cdot\frac{Q_{c}q^{2(b-a+k-{\lambda^{(c)}_k}')}-1}{Q_{c}q^{2(b-a+k-{\lambda^{(c)}_k}'+1)}-1}\right) \\
	 	&\qquad\qquad\cdot\prod_{\substack{(a,b,0)\in\lambda^{(0)}}}\left(\frac{2\left(Q^{-1}_{c}q^{2(b-a-\lambda_1^{(c)})}-1\right)\left(Q_{{c}}q^{2(b-a+\lambda^{(c)}_1+1)}-1\right)}{\left(q^2+1\right)q^{2(b-a-\lambda_1^{(c)})}}\right)\\
	 	&\qquad\qquad\qquad  \cdot\prod_{\substack{(a,b,c)\in\lambda^{(c)}}}\left(\frac{2\left(Q_{c}q^{2(b-a)}-1\right)\left(Q_{{c}}q^{2(b-a+1)}-1\right)}{\left(q^2+1\right)Q_{c}q^{2(b-a)}}\right).
	 \end{align*}
	\end{defn}

	The following Lemma is a direct computation.
\begin{lem}Let $\iota_1,\iota_2\in\mathbb{K}^*$, we have
	\begin{align}\label{difference}
		\mathtt{q}(\iota_1)-\mathtt{q}(\iota_2)=2\frac{(\iota_1-\iota_2)(q^2\iota_1\iota_2-1)}{(q^2+1)\iota_1\iota_2}.
	\end{align}
\end{lem}

	By \eqref{difference}, for any two boxes $\alpha=(a,b,c),\beta=(i,j,c')$, we have 
\begin{align}
	\mathtt{q}(\res(\alpha))-\mathtt{q}(\res(\beta))&=2\frac{\left(Q_cq^{2(b-a)}-Q_{c'}q^{2(j-i)}\right)\left(Q_cQ_{c'}q^{2(b-a+j-i+1)}-1\right)}{(q^2+1)Q_cQ_{c'}q^{2(b-a+j-i)}}\nonumber\\
	&=2\frac{\left(Q_cQ^{-1}_{c'}q^{2(b-a-j+i)}-1\right)\left(Q_cQ_{c'}q^{2(b-a+j-i+1)}-1\right)}{(q^2+1)Q_cq^{2(b-a)}}.\label{difference-residue}
\end{align}

The following Proposition is key to this subsection.

\begin{prop}\label{closed q}
		Let $\undla\in\mathscr{P}^{\mathsf{\bullet},m}_{n}$, where $\bullet\in\{\mathtt{0},\mathtt{s}\}$. We have $$
		q(\undla)^{-1}=\begin{cases}
			\left(\prod\limits_{c=1}^m 	\mathbb{Y}^{\undla}_{c}\right)\cdot \left(\prod\limits_{1\leq c_1<c_2\leq m}	\mathbb{X}^{\undla}_{c_1,c_2}\right),&\qquad\text{if $\bullet=\mathtt{0}$};\\
				\left(\prod\limits_{c=0}^m 	\mathbb{Y}^{\undla}_{c}\right)\cdot \left(\prod\limits_{0\leq c_1<c_2\leq m}	\mathbb{X}^{\undla}_{c_1,c_2}\right),&\qquad\text{if $\bullet=\mathtt{s}$}.
			\end{cases}
		$$
	\end{prop}

	We start the proof of Proposition \ref{closed q} from two special cases.
	
	\begin{lem}\label{closed q-special} We have the following.
		\begin{enumerate}
		\item Suppose $m=1$ and $\bullet=\mathtt{0}$. Then for $\lambda\in\mathscr{P}_{n}$, we have $
		q(\lambda)^{-1}=\mathbb{Y}^{\lambda}_{1}.$	
		\item Suppose $m=0$ and $\bullet=\mathtt{s}$. Then for $\lambda\in\mathscr{P}^{\mathtt{s}}_{n}$, we have $
		q(\lambda)^{-1}=\mathbb{Y}^{\lambda}_{0}.$	
		\end{enumerate}
	\end{lem}
	
	\begin{proof}
		In both cases, we shall use the special tableaux $\mt_{\lambda}$ to compute $q(\lambda)^{-1}$ in the formula \eqref{def qt}, i.e. $$q(\lambda)^{-1}
		=\prod\limits_{k=1}^{n}\frac{
			\prod\limits_{\beta\in\Add(\mt_{\lambda}\downarrow_{k-1})\setminus \{\mt_{\lambda}^{-1}(k)\}}\left(\mathtt{q}(\res_{\mt_{\lambda}}(k))-\mathtt{q}(\res(\beta))\right)}{\prod\limits_{\beta\in\Rem(\mt_{\lambda}\downarrow_{k-1})\setminus \mathcal{D}}\left(\mathtt{q}(\res_{\mt_{\lambda}}(k))-\mathtt{q}(\res(\beta))\right)}$$ and we prove the Lemma by induction on $n$.
		We assume $\alpha=(a,b,c)=\mt^{-1}_{\lambda}(n)$ and set $\mu=\lambda\backslash \{\alpha\}$. It follows that$$
		q(\lambda)^{-1}=q(\mu)^{-1} \frac{
			\prod\limits_{\beta\in\Add(\mt_{\lambda}\downarrow_{n-1})\setminus \{\alpha\}}\left(\mathtt{q}(\res(\alpha))-\mathtt{q}(\res(\beta))\right)}{\prod\limits_{\beta\in\Rem(\mt_{\lambda}\downarrow_{n-1})\setminus \mathcal{D}}\left(\mathtt{q}(\res(\alpha))-\mathtt{q}(\res(\beta))\right)}.
		$$ We claim the following.
		 \begin{align}
		 	\text{ For $\lambda\in\mathscr{P}_{n}$, we have } &\mathbb{Y}^{\lambda}_{1}=\mathbb{Y}^{\mu}_{1}\cdot \frac{
				\prod\limits_{\beta=(a',b')\in\Add(\mt_{\lambda}\downarrow_{n-1})\setminus \{\alpha\}}\left(\mathtt{q}(\res(\alpha))-\mathtt{q}(\res(\beta))\right)}{\prod\limits_{\beta=(a',b')\in\Rem(\mt_{\lambda}\downarrow_{n-1})\setminus \mathcal{D}}\left(\mathtt{q}(\res(\alpha))-\mathtt{q}(\res(\beta))\right)}.\label{claim 1}\\
			\text{ For $\lambda\in\mathscr{P}^{\mathtt{s}}_{n}$, we have }&\mathbb{Y}^{\lambda}_{0}=\mathbb{Y}^{\mu}_{0}\cdot \frac{
					\prod\limits_{\beta=(a',b')\in\Add(\mt_{\lambda}\downarrow_{n-1})\setminus \{\alpha\}}\left(\mathtt{q}(\res(\alpha))-\mathtt{q}(\res(\beta))\right)}{\prod\limits_{\beta=(a',b')\in\Rem(\mt_{\lambda}\downarrow_{n-1})\setminus \mathcal{D}}\left(\mathtt{q}(\res(\alpha))-\mathtt{q}(\res(\beta))\right)}.\label{claim 2}
		\end{align}
		It is clear our Lemma follows from above claim \eqref{claim 1}, \eqref{claim 2} and induction hypothesis. Hence we only need to prove  \eqref{claim 1}, \eqref{claim 2}.
		
	Note that in our situation, the following holds \begin{align}\label{equal}|\Add(\mt_{\lambda}\downarrow_{n-1})\setminus \{\alpha\}|=|\Rem(\mt_{\lambda}\downarrow_{n-1})\setminus \mathcal{D}|.
		\end{align}
		
	(i). We first prove the claim \eqref{claim 1}. Applying \eqref{difference-residue} and \eqref{equal}, we have\begin{align}
			&\qquad\frac{
				\prod\limits_{\beta=(a',b')\in\Add(\mt_{\lambda}\downarrow_{n-1})\setminus \{\alpha\}}\left(\mathtt{q}(\res(\alpha))-\mathtt{q}(\res(\beta))\right)}{\prod\limits_{\beta=(a',b')\in\Rem(\mt_{\lambda}\downarrow_{n-1})\setminus \mathcal{D}}\left(\mathtt{q}(\res(\alpha))-\mathtt{q}(\res(\beta))\right)}\nonumber\\
			&=\frac{
				\prod\limits_{\beta=(a',b')\in\Add(\mt_{\lambda}\downarrow_{n-1})\setminus \{\alpha\}}\left(q^{2(b-a-b'+a')}-1\right)\left(Q_1^2q^{2(b-a+b'-a'+1)}-1\right)}{\prod\limits_{\beta=(a',b')\in\Rem(\mt_{\lambda}\downarrow_{n-1})\setminus \mathcal{D}}\left(q^{2(b-a-b'+a')}-1\right)\left(Q_1^2q^{2(b-a+b'-a'+1)}-1\right)}	\label{induction 1}			
		\end{align}
	
	Let $i>a$. Suppose $(i,d)\in \Rem(\mt_{\lambda}\downarrow_{n-1})\setminus \mathcal{D}$, then there exists $d'\leq d<b$, such that $(i+1,d')\in\Add(\mt_{\lambda}\downarrow_{n-1})\setminus \{\alpha\}$. This gives a bijection \begin{align*}\Rem(\mt_{\lambda}\downarrow_{n-1})\setminus\left(\mathcal{D}\cup \{(a-1,b)\}\right)&\to \Add(\mt_{\lambda}\downarrow_{n-1})\setminus \{\alpha,(1,b+1)\},\\
		(i,d)&\mapsto (i+1,d').
	\end{align*} We can compute \begin{align*}
		\frac{q^{2(b-a-d'+i+1)}-1}{q^{2(b-a-d+i)}-1}&=\prod_{\substack{j=d'}}^d \frac{q^{2(b-a-j+i+1)}-1}{q^{2(b-a-j+i)}-1}\\
		&=\prod_{\substack{j=d'}}^d \frac{q^{2(\lambda_a-a-j+\lambda'_j+1)}-1}{q^{2(\lambda_a-a-j+\lambda'_j)}-1}\\
		&=\prod_{\substack{j=d'}}^d\left(\frac{q^{2h^\lambda_{a,j}}-1}{q^{2h^\mu_{a,j}}-1}\right).
		\end{align*} 
		
		Therefore, if $a>1$, then \begin{align}
	&\qquad \frac{
	 		\prod\limits_{\beta=(a',b')\in\Add(\mt_{\lambda}\downarrow_{n-1})\setminus \{\alpha\}}\left(q^{2(b-a-b'+a')}-1\right)}{\prod\limits_{\beta=(a',b')\in\Rem(\mt_{\lambda}\downarrow_{n-1})\setminus \mathcal{D}}\left(q^{2(b-a-b'+a')}-1\right)}=\left(\prod_{k=1}^{b-1}\frac{q^{2h^\lambda_{a,k}}-1}{q^{2h^\mu_{a,k}}-1}\right) \cdot \frac{q^{-2a}-1}{q^{-2}-1}\nonumber\\
	 		&=\left(\prod_{k=1}^{b-1}\frac{q^{2h^\lambda_{a,k}}-1}{q^{2h^\mu_{a,k}}-1}\right) \cdot \frac{q^{2a}-1}{q^{2}-1}\cdot q^{-2(a-1)}\nonumber\\
	 		&=\left(\prod\limits_{k=1}^{b-1}\frac{q^{2h^\lambda_{a,k}}-1}{q^{2h^\mu_{a,k}}-1}\right) \cdot \frac{\prod\limits_{k=1}^a\left(q^{2h^\lambda_{k,b}}-1\right)}{(q^2-1)\cdot\prod\limits_{k=1}^{a-1}\left(q^{2h^\mu_{k,b}}-1\right)}\cdot q^{-2(a-1)}\label{induction 2}.
	 		\end{align} It is easy to check \eqref{induction 2} also holds when $a=1$.
	 		
	 		Next we deal with another part on the right-hand-side of \eqref{induction 1}, i.e. $$\frac{
	 			\prod\limits_{\beta=(a',b')\in\Add(\mt_{\lambda}\downarrow_{n-1})\setminus \{\alpha\}}\left(Q_1^2q^{2(b-a+b'-a'+1)}-1\right)}{\prod\limits_{\beta=(a',b')\in\Rem(\mt_{\lambda}\downarrow_{n-1})\setminus \mathcal{D}}\left(Q_1^2q^{2(b-a+b'-a'+1)}-1\right)}	.$$	
	 			{\bf Case 1:} $b>1$. Let $j<b-1$. Suppose $(e,j)\in \Rem(\mt_{\lambda}\downarrow_{n-1})\setminus \mathcal{D}$, then there exists $a<e'\leq e$, such that $(e',j+1)\in \Add(\mt_{\lambda}\downarrow_{n-1})\setminus \{\alpha\}$. This again gives rise to a bijection \begin{align*}\Rem(\mt_{\lambda}\downarrow_{n-1})\setminus\left(\mathcal{D}\cup \{(a-1,b),(\lambda'_{b-1},b-1)\}\right)&\to \Add(\mt_{\lambda}\downarrow_{n-1})\setminus \{\alpha,(1,b+1),(\ell(\lambda)+1,1)\},\\
	 				(e,j)&\mapsto (e',j+1).
	 			\end{align*} We can compute \begin{align*}
	 			\frac{Q_1^2q^{2(b-a+j-e'+2)}-1}{Q_1^2q^{2(b-a+j-e+1)}-1}&=\prod_{\substack{i=e'}}^e \frac{Q_1^2q^{2(b-a+j-i+2)}-1}{Q_1^2q^{2(b-a+j-i+1)}-1}\\
	 			&=\prod_{\substack{i=e'}}^e \frac{Q_1^2q^{2(\lambda_a-a+\lambda_i-i+2)}-1}{Q_1^2q^{2(\mu_a-a+\mu_i-i+2)}-1}.
	 			\end{align*} 
	 				Therefore, if $a>1$, then \begin{align}	
	 					&\quad\frac{
	 					\prod\limits_{\beta=(a',b')\in\Add(\mt_{\lambda}\downarrow_{n-1})\setminus \{\alpha\}}\left(Q_1^2q^{2(b-a+b'-a'+1)}-1\right)}{\prod\limits_{\beta=(a',b')\in\Rem(\mt_{\lambda}\downarrow_{n-1})\setminus \mathcal{D}}\left(Q_1^2q^{2(b-a+b'-a'+1)}-1\right)}\nonumber\\
	 						&=\left(\prod_{i=\lambda'_{b-1}+1}^{\ell(\lambda)}\frac{Q_1^2q^{2(\lambda_a-a+\lambda_i-i+2)}-1}{Q_1^2q^{2(\mu_a-a+\mu_i-i+2)}-1}\right)\cdot\frac{Q_1^2q^{2(b-a+b+1)}-1}{Q_1^2q^{2(b-a+1+b-a+1)}-1}\cdot\frac{Q_1^2q^{2(b-a+1-\ell(\lambda))}-1}{Q_1^2q^{2(b-a+b-\lambda'_{b-1})}-1}\nonumber\\
	 							&=\left(\prod_{i=\lambda'_{b-1}+1}^{\ell(\lambda)}\frac{Q_1^2q^{2(\lambda_a-a+\lambda_i-i+2)}-1}{Q_1^2q^{2(\mu_a-a+\mu_i-i+2)}-1}\right)\cdot\left(\prod_{i=1}^{a-1}\frac{Q_1^2q^{2(b-a+b-i+2)}-1}{Q_1^2q^{2(b-1-a+b-i+2)}-1}\right)\cdot\frac{Q_1^2q^{2(b-a+1-\ell(\lambda))}-1}{Q_1^2q^{2(b-a+b-\lambda'_{b-1})}-1}\nonumber\\
	 						&=\left(\prod_{i=\lambda'_{b-1}+1}^{\ell(\lambda)}\frac{Q_1^2q^{2(\lambda_a-a+\lambda_i-i+2)}-1}{Q_1^2q^{2(\mu_a-a+\mu_i-i+2)}-1}\right)\cdot\left(\prod_{i=1}^{a-1}\frac{Q_1^2q^{2(\lambda_a-a+\lambda_i-i+2)}-1}{Q_1^2q^{2(\mu_a-a+\mu_i-i+2)}-1}\right)\cdot\frac{Q_1^2q^{2(b-a+1-\ell(\lambda))}-1}{Q_1^2q^{2(b-a+b-\lambda'_{b-1})}-1}\nonumber\\
	 						&=\left(\prod_{i=\lambda'_{b-1}+1}^{\ell(\lambda)}\frac{Q_1^2q^{2(\lambda_a-a+\lambda_i-i+2)}-1}{Q_1^2q^{2(\mu_a-a+\mu_i-i+2)}-1}\right)\cdot\left(\prod_{i=1}^{a-1}\frac{Q_1^2q^{2(\lambda_a-a+\lambda_i-i+2)}-1}{Q_1^2q^{2(\mu_a-a+\mu_i-i+2)}-1}\right)\nonumber\\
	 						&\qquad\qquad\qquad\qquad\qquad\qquad\qquad\cdot\frac{Q_1^2q^{4(b-a)}-1}{Q_1^2q^{2(b-a+b-\lambda'_{b-1})}-1}\cdot \frac{Q_1^2q^{2(b-a+1-\ell(\lambda))}-1}{Q_1^2q^{4(b-a)}-1}\nonumber\\
	 						&=\left(\prod_{i=\lambda'_{b-1}+1}^{\ell(\lambda)}\frac{Q_1^2q^{2(\lambda_a-a+\lambda_i-i+2)}-1}{Q_1^2q^{2(\mu_a-a+\mu_i-i+2)}-1}\right)\cdot\left(\prod_{i=1}^{a-1}\frac{Q_1^2q^{2(\lambda_a-a+\lambda_i-i+2)}-1}{Q_1^2q^{2(\mu_a-a+\mu_i-i+2)}-1}\right)\nonumber\\
	 						&\qquad\qquad\qquad\qquad\qquad\qquad\qquad\cdot\left(\prod_{i=a+1}^{\lambda_{b-1}'}\frac{Q_1^2q^{2(b-a+b-1-i+2)}-1}{Q_1^2q^{2(b-1-a+b-1-i+2)}-1}\right)\cdot \frac{Q_1^2q^{2(b-a+1-\ell(\lambda))}-1}{Q_1^2q^{4(b-a)}-1}\nonumber\\
	 						&=\left(\prod_{i=\lambda'_{b-1}+1}^{\ell(\lambda)}\frac{Q_1^2q^{2(\lambda_a-a+\lambda_i-i+2)}-1}{Q_1^2q^{2(\mu_a-a+\mu_i-i+2)}-1}\right)\cdot\left(\prod_{i=1}^{a-1}\frac{Q_1^2q^{2(\lambda_a-a+\lambda_i-i+2)}-1}{Q_1^2q^{2(\mu_a-a+\mu_i-i+2)}-1}\right)\nonumber\\
	 						&\qquad\qquad\qquad\qquad\qquad\qquad\qquad\cdot\left(\prod_{i=a+1}^{\lambda_{b-1}'}\frac{Q_1^2q^{2(\lambda_a-a+\lambda_i-i+2)}-1}{Q_1^2q^{2(\mu_a-a+\mu_i-i+2)}-1}\right)\cdot \frac{Q_1^2q^{2(b-a+1-\ell(\lambda))}-1}{Q_1^2q^{4(b-a)}-1}\nonumber\\
	 						&=\left(\prod_{\substack{1\leq i\leq \ell(\lambda)\\i\neq a}}\frac{Q_1^2q^{2(\lambda_a-a+\lambda_i-i+2)}-1}{Q_1^2q^{2(\mu_a-a+\mu_i-i+2)}-1}\right)\cdot \frac{Q_1^2q^{2(b-a+1-\ell(\lambda))}-1}{Q_1^2q^{4(b-a)}-1}.\label{induction 3}
	 										\end{align}
	 					It is easy to check \eqref{induction 3} also holds when $a=1,b>1$.  Now \eqref{induction 2} and \eqref{induction 3} imply claim \eqref{claim 1} in this case.
	 					
	 		{\bf Case 2:} $b=1$. Note that in this case, we have $\lambda=(\underbrace{1,1,\cdots,1}_{b})$.  If $a>1$, we can compute \begin{align}
	 			\frac{
	 				\prod\limits_{\beta=(a',b')\in\Add(\mt_{\lambda}\downarrow_{n-1})\setminus \{\alpha\}}\left(Q_1^2q^{2(b-a+b'-a'+1)}-1\right)}{\prod\limits_{\beta=(a',b')\in\Rem(\mt_{\lambda}\downarrow_{n-1})\setminus \mathcal{D}}\left(Q_1^2q^{2(b-a+b'-a'+1)}-1\right)}&=\frac{Q_1^2q^{2(1-a+1+1)}-1}{Q_1^2q^{2(1-a+1-(a-1)+1)}-1}\nonumber\\
	 				&=\frac{Q_1^2q^{2(3-a)}-1}{Q_1^2q^{2(4-2a)}-1}\label{induction 4}.
	 			\end{align}  On the other hand, we have \begin{align}
	 		&\qquad	\prod\limits_{1\leq i<a}\left(\frac{Q_1^2q^{2\left(\lambda_i-i+\lambda_{a}-a+2\right)}}{Q_1^2q^{2\left(2-i-a\right)}-1}\right)=\prod_{\substack{1\leq i\leq a-1}}\frac{Q_1^2q^{2\left(4-i-a\right)}-1}{Q_1^2q^{2\left(2-i-a\right)}-1}\nonumber\\
	 		&=\frac{ \prod\limits_{\substack{-1\leq i\leq a-3}}\left(Q_1^2q^{2\left(2-i-a\right)}-1\right)}{\prod\limits_{\substack{1\leq i\leq a-1}}\left(Q_1^2q^{2\left(2-i-a\right)}-1\right)}\nonumber\\
	 		&=\frac{\left(Q_1^2q^{2(2-a)}-1\right)\left(Q_1^2q^{2(3-a)}-1\right)}{\left(Q_1^2q^{2(3-2a)}-1\right)\left(Q_1^2q^{2(4-2a)}-1\right)}\label{remaining 1}
	 			\end{align} and 
	 			\begin{align}\label{remaining 2}
	 	&\qquad	\frac{	 \prod\limits_{(i,1)\in\lambda}	\left(Q_1^2q^{2\left(1-i-\ell(\lambda)+1\right)}-1\right)}{\prod\limits_{(i,1)\in\mu}\left( Q_1^2q^{2\left(1-i-\ell(\mu)+1\right)}-1\right)}	\cdot \frac{1}{Q_1^2q^{4\left(b-a\right)}-1}\nonumber\\	 				 
	 			&=	 \frac{\prod\limits_{k=1}^{a}\left(Q_1^2q^{2\left(1-k-a+1\right)}-1\right)}{\prod\limits_{k=1}^{a-1}\left(Q_1^2q^{2\left(1-k-(a-1)+1\right)}-1\right)}
	 			\cdot \frac{1}{Q_1^2q^{4\left(1-a\right)}-1}\nonumber\\
	 			&=\frac{\prod\limits_{k=1}^{a}\left(Q_1^2q^{2\left(2-a-k\right)}-1\right)}{\prod\limits_{k=0}^{a-2}\left(Q_1^2q^{2\left(2-a-k\right)}-1\right)}
	 			\cdot \frac{1}{Q_1^2q^{4\left(1-a\right)}-1}\nonumber\\
	 			&=\frac{Q_1^2q^{2\left(3-2a\right)}-1}{Q_1^2q^{2\left(2-a\right)}-1}.
	 			\end{align}  It is easy to check \eqref{induction 4},\eqref{remaining 1}, \eqref{remaining 2} also hold when $a=b=1$. Combining \eqref{remaining 1}, \eqref{remaining 2} with the definition of $\mathbb{Y}$, we deduce that in this case, \begin{align*}
	 			\frac{\mathbb{Y}^{\lambda}_{1}}{\mathbb{Y}^{\mu}_{1}}=\prod_{k=1}^{b-1}\left(\frac{q^{2h^\lambda_{a,k}}-1}{q^{2h^\mu_{a,k}}-1}\right) \cdot \frac{\prod\limits_{k=1}^a\left(q^{2h^\lambda_{k,b}}-1\right)}{(q^2-1)\cdot\prod\limits_{k=1}^{a-1}\left(q^{2h^\mu_{k,b}}-1\right)}\cdot q^{-2(a-1)}\cdot\frac{Q_1^2q^{2(3-a)}-1}{Q_1^2q^{2(4-2a)}-1}.
	 			\end{align*}
	 			
	 		 Hence, we obtain our claim \eqref{claim 1} from \eqref{induction 2} and \eqref{induction 4} in this case. This completes the proof of claim \eqref{claim 1}.

	(ii). Then we prove the claim \eqref{claim 2}. Applying \eqref{difference-residue} and \eqref{equal}, we also have\begin{align}
	 	&\qquad\frac{
	 		\prod\limits_{\beta=(a',b')\in\Add(\mt_{\lambda}\downarrow_{n-1})\setminus \{\alpha\}}\left(\mathtt{q}(\res(\alpha))-\mathtt{q}(\res(\beta))\right)}{\prod\limits_{\beta=(a',b')\in\Rem(\mt_{\lambda}\downarrow_{n-1})\setminus \mathcal{D}}\left(\mathtt{q}(\res(\alpha))-\mathtt{q}(\res(\beta))\right)}\nonumber\\
	 	&=\frac{
	 		\prod\limits_{\beta=(a',b')\in\Add(\mt_{\lambda}\downarrow_{n-1})\setminus \{\alpha\}}\left(q^{2(b-a-b'+a')}-1\right)\left(q^{2(b-a+b'-a'+1)}-1\right)}{\prod\limits_{\beta=(a',b')\in\Rem(\mt_{\lambda}\downarrow_{n-1})\setminus \mathcal{D}}\left(q^{2(b-a-b'+a')}-1\right)\left(q^{2(b-a+b'-a'+1)}-1\right)}	\nonumber\\
	 		&=\frac{
	 			\prod\limits_{\beta=(a',b')\in\Add(\mt_{\lambda}\downarrow_{n-1})\setminus \{\alpha\}}\left(q^{2(\lambda_a-\lambda_{a'}-1)}-1\right)\left(q^{2(\lambda_a+\lambda_{a'})}-1\right)}{\prod\limits_{\beta=(a',b')\in\Rem(\mt_{\lambda}\downarrow_{n-1})\setminus \mathcal{D}}\left(q^{2(\lambda_a-\lambda_{a'})}-1\right)\left(q^{2(\lambda_a+\lambda_{a'}-1)}-1\right)}.
	 		\label{induction 1'}			
	 \end{align}
	 
	 {\bf Case 1:} $a<b.$ We can use completely similar computations as \eqref{induction 2} and \eqref{induction 3} to obtain that\begin{align}
	 	\frac{
	 		\prod\limits_{\beta=(a',b')\in\Add(\mt_{\lambda}\downarrow_{n-1})\setminus \{\alpha\}}\left(q^{2(\lambda_a-\lambda_{a'}-1)}-1\right)}{\prod\limits_{\beta=(a',b')\in\Rem(\mt_{\lambda}\downarrow_{n-1})\setminus \mathcal{D}}\left(q^{2(\lambda_a-\lambda_{a'})}-1\right)}=\left(\prod_{k=\ell(\lambda)+1}^{b-1}\frac{q^{2h^{\tilde{\lambda}}_{a,k}}-1}{q^{2h^{\tilde{\mu}}_{a,k}}-1}\right)\cdot\frac{\prod\limits_{k=1}^a\left( q^{2h^{\tilde{\lambda}}_{k,b}}-1\right)}{(q^2-1)\cdot\prod\limits_{k=1}^{a-1}\left(q^{2h^{\tilde{\mu}}_{k,b}}-1\right)}\cdot q^{-2(a-1)}\label{induction 2'}
	 \end{align} and \begin{align}
	 \frac{
	 	\prod\limits_{\beta=(a',b')\in\Add(\mt_{\lambda}\downarrow_{n-1})\setminus \{\alpha\}}\left(q^{2(\lambda_a+\lambda_{a'})}-1\right)}{\prod\limits_{\beta=(a',b')\in\Rem(\mt_{\lambda}\downarrow_{n-1})\setminus \mathcal{D}}\left(q^{2(\lambda_a+\lambda_{a'}-1)}-1\right)}&=\left(\prod_{\substack{1\le i\leq \ell(\lambda)\\ i\neq a}}\frac{q^{2(\lambda_a+\lambda_i)}-1}{q^{2(\mu_a+\mu_i)}-1}\right) \cdot \frac{q^{2(\lambda_a-1)-1}}{q^{4(\lambda_a-1)-1}}\nonumber\\
	 	&=\left(\prod_{\substack{1\le i\leq \ell(\lambda)\\ i\neq a}}\frac{q^{2(\lambda_a+\lambda_i)}-1}{q^{2(\mu_a+\mu_i)}-1}\right) \cdot \frac{1}{q^{2(\lambda_a-1)+1}}.	\label{induction 3'}
	 	\end{align} 
	 Then our claim \eqref{claim 2} follows from \eqref{induction 2'} and \eqref{induction 3'} in this case.
	 
	 {\bf Case 2:} $a=b.$ Note that in this case, $\lambda=(a,a-1,\cdots,1)$. If $a>1$, by \eqref{induction 1'}, we have \begin{align}
	 	\qquad\frac{
	 		\prod\limits_{\beta=(a',b')\in\Add(\mt_{\lambda}\downarrow_{n-1})\setminus \{\alpha\}}\left(\mathtt{q}(\res(\alpha))-\mathtt{q}(\res(\beta))\right)}{\prod\limits_{\beta=(a',b')\in\Rem(\mt_{\lambda}\downarrow_{n-1})\setminus \mathcal{D}}\left(\mathtt{q}(\res(\alpha))-\mathtt{q}(\res(\beta))\right)}&=\frac{\left(q^{-2a}-1\right)\left(q^{2(a+1)}-1\right)}{\left(q^{-2}-1\right)\left(q^{4}-1\right)}\nonumber\\
	 		&=q^{-2(a-1)}\cdot\frac{\left(q^{2a}-1\right)\left(q^{2(a+1)}-1\right)}{\left(q^{2}-1\right)\left(q^{4}-1\right)}\label{induction 4'}.
	 \end{align}  
	 On the other hand, we have \begin{align}\label{remaining 1'}
	 \frac{\prod\limits_{1\leq i<i'\leq \ell(\lambda^{(0)})+1}\frac{q^{2\left(\lambda^{(0)}_i+\lambda^{(0)}_{i'}\right)}-1}{q^2-1}}{\prod\limits_{1\leq i<i'\leq \ell(\mu^{(0)})+1}\frac{q^{2\left(\mu^{(0)}_i+\mu^{(0)}_{i'}\right)}-1}{q^2-1}}= \prod_{k=3}^{a+1}\frac{q^{2k}-1}{q^2-1}
	 \end{align} and  \begin{align}\label{remaining 2'}
	 \prod\limits_{i=1}^{a-1}\left(\frac{q^2-1}{q^{2h_{i,a}^{\widetilde{\mu^{(0)}}}}-1}\right)=\prod\limits_{k=1}^a\left(\frac{q^2-1}{ q^{2k}-1}\right)
	 	 \end{align} It is easy to check \eqref{induction 4'},\eqref{remaining 1'}, \eqref{remaining 2'} also hold when $a=b=1$. Combining \eqref{remaining 1'}, \eqref{remaining 2'} with the definition of $\mathbb{Y}$, we deduce that in this case, \begin{align*}
	 	 \frac{\mathbb{Y}^{\lambda}_{0}}{\mathbb{Y}^{\mu}_{0}}=q^{-2(a-1)}\cdot\frac{\left(q^{2a}-1\right)\left(q^{2(a+1)}-1\right)}{\left(q^{2}-1\right)\left(q^{4}-1\right)}.	 	 
	 	 \end{align*}
		 Hence, we obtain our claim \eqref{claim 2} from \eqref{induction 4'} in this case. This completes the proof of \eqref{claim 2} and the whole Lemma.
		\end{proof}

		\medskip
		{\bf Proof of Proposition \ref{closed q}}:
			We shall again use the special tableaux $\mt_{\undla}$ to compute $q(\undla)^{-1}$ in the formula \eqref{def qt}, i.e. $$q(\undla)^{-1}
			=\prod\limits_{k=1}^{n}\frac{
					\prod\limits_{\beta\in\Add(\mt_{\undla}\downarrow_{k-1})\setminus \{\mt_{\undla}^{-1}(k)\}}\left(\mathtt{q}(\res_{\mt_{\undla}}(k))-\mathtt{q}(\res(\beta))\right)}{\prod\limits_{\beta\in\Rem(\mt_{\undla}\downarrow_{k-1})\setminus \mathcal{D}}\left(\mathtt{q}(\res_{\mt_{\undla}}(k))-\mathtt{q}(\res(\beta))\right)}$$ and we prove the proposition by induction on $n$.
				We assume $\alpha=(a,b,c)=\mt^{-1}_{\undla}(n)$ and $\undmu=\undla\backslash \{\alpha\}$. It follows that$$
				q(\undla)^{-1}=q(\undmu)^{-1} \frac{
						\prod\limits_{\beta\in\Add(\mt_{\undla}\downarrow_{n-1})\setminus \{\alpha\}}\left(\mathtt{q}(\res(\alpha))-\mathtt{q}(\res(\beta))\right)}{\prod\limits_{\beta\in\Rem(\mt_{\undla}\downarrow_{n-1})\setminus \mathcal{D}}\left(\mathtt{q}(\res(\alpha))-\mathtt{q}(\res(\beta))\right)}.
				$$
				
				For $0 \leq c'\leq m$, we denote 	\begin{align*}
				\mathscr{R}_{c'}:=	\{(a',b',c')\in\Rem(\mt_{{\undla}}\downarrow_{n-1})\setminus\mathcal{D}\},\qquad \mathscr{A}_{c'}:=\left\{(a',b',c')\in\Add(\mt_{{\undla}}\downarrow_{n-1})\setminus \{\alpha\}\right\}.
					\end{align*} Then we have \begin{align*}
					\Rem(\mt_{{\undla}}\downarrow_{n-1})\setminus\mathcal{D}=\begin{cases}\bigsqcup\limits_{1\leq c'\leq m} \mathscr{R}_{c'}\qquad&\text{ if $\bullet=\mathtt{0},$}\\
						\bigsqcup\limits_{0\leq c'\leq m} \mathscr{R}_{c'}\qquad&\text{ if $\bullet=\mathtt{s},$}
						\end{cases}
						\end{align*} and  \begin{align*}
						\Add(\mt_{\undla}\downarrow_{n-1})\setminus \{\alpha\}=\begin{cases}\bigsqcup\limits_{1\leq c'\leq m} \mathscr{A}_{c'}\qquad&\text{ if $\bullet=\mathtt{0},$}\\
							\bigsqcup\limits_{0\leq c'\leq m} \mathscr{A}_{c'}\qquad&\text{ if $\bullet=\mathtt{s}.$}
						\end{cases}
						\end{align*} 
			We claim the following.
			\begin{enumerate}
		\item\label{claim 1'}	\begin{align*}
				\mathbb{Y}^{\undla}_{c}&=\mathbb{Y}^{\undmu}_{c}\cdot \frac{
								\prod\limits_{\beta=(a',b',c)\in\mathscr{A}_{c}}\left(\mathtt{q}(\res(\alpha))-\mathtt{q}(\res(\beta))\right)}{\prod\limits_{\beta=(a',b',c)\in\mathscr{R}_{c}}\left(\mathtt{q}(\res(\alpha))-\mathtt{q}(\res(\beta))\right)};
		\end{align*}
			\item\label{claim 2'} 			For any $1\leq c_1<c<c_2\leq m$,	\begin{align*}
				\mathbb{X}^{\undla}_{{c_1},{c}}=\mathbb{X}^{\undmu}_{{c_1},{c}}\cdot \frac{
					\prod\limits_{\beta=(a',b',c_1)\in\mathscr{A}_{c_1}}\left(\mathtt{q}(\res(\alpha))-\mathtt{q}(\res(\beta))\right)}{\prod\limits_{\beta=(a',b',c_1)\in\mathscr{R}_{c_1}}\left(\mathtt{q}(\res(\alpha))-\mathtt{q}(\res(\beta))\right)},	
			\end{align*} and \begin{align*}
							\mathbb{X}^{\undla}_{{c},{c_2}}=\mathbb{X}^{\undmu}_{{c},{c_2}}\cdot \frac{
								\prod\limits_{\beta=(a',b',c_2)\in\mathscr{A}_{c_2}}\left(\mathtt{q}(\res(\alpha))-\mathtt{q}(\res(\beta))\right)}{\prod\limits_{\beta=(a',b',c_2)\in\mathscr{R}_{c_2}}\left(\mathtt{q}(\res(\alpha))-\mathtt{q}(\res(\beta))\right)},
								\end{align*}

			\item \label{claim 3'} 
					When $\bullet=\mathtt{s}$. If $c>0$, then\begin{align*}
						\mathbb{X}^{\undla}_{{0},{c}}&=\mathbb{X}^{\undmu}_{{0},{c}}\cdot \frac{
						\prod\limits_{\beta=(a',b',0)\in\mathscr{A}_{0}}\left(\mathtt{q}(\res(\alpha))-\mathtt{q}(\res(\beta))\right)}{\prod\limits_{\beta=(a',b',0)\in\mathscr{R}_{0}}\left(\mathtt{q}(\res(\alpha))-\mathtt{q}(\res(\beta))\right)};
					\end{align*}
			 if $c=0$, then for any $1\leq c'\leq m$\begin{align*}
				\mathbb{X}^{\undla}_{{0},{c'}}&=\mathbb{X}^{\undmu}_{{0},{c'}}\cdot \frac{
					\prod\limits_{\beta=(a',b',c')\in\mathscr{A}_{c'}}\left(\mathtt{q}(\res(\alpha))-\mathtt{q}(\res(\beta))\right)}{\prod\limits_{\beta=(a',b',c')\in\mathscr{R}_{c'}}\left(\mathtt{q}(\res(\alpha))-\mathtt{q}(\res(\beta))\right)}.
			\end{align*}			
		\end{enumerate}
				It is obvious that the proposition follows from \eqref{claim 1'},\eqref{claim 2'},\eqref{claim 3'} and induction hypothesis. Hence we only need to prove \eqref{claim 1'},\eqref{claim 2'},\eqref{claim 3'}.
					
		\eqref{claim 1'} follows from Lemma \ref{closed q-special}. We prove \eqref{claim 2'}.

					Let $c<c_2\leq m$. 
					
					Suppose $(i,d,c_2)\in \mathscr{R}_{c_2}$, then there exists $d'\leq d<b$, such that $(i+1,d',c_2)\in\mathscr{A}_{c_2}$. This gives a bijection \begin{align*}\mathscr{R}_{c_2}&\to \mathscr{A}_{c_2}\setminus \{(1,\lambda_{1}^{(c_2)}+1,c_2)\},\\
						(i,d,c_2)&\mapsto (i+1,d',c_2).
					\end{align*} 
In particular, the following holds \begin{align}\label{equal'}|\mathscr{A}_{c_2}|=|\mathscr{R}_{c_2}|+1.
	\end{align} 
	
	Applying \eqref{difference-residue} and \eqref{equal'}, we have\begin{align}
		&\qquad	\frac{
			\prod\limits_{\beta=(a',b',c_2)\in\mathscr{A}_{c_2}}\left(\mathtt{q}(\res(\alpha))-\mathtt{q}(\res(\beta))\right)}{\prod\limits_{\beta=(a',b',c_2)\in\mathscr{R}_{c_2}}\left(\mathtt{q}(\res(\alpha))-\mathtt{q}(\res(\beta))\right)}\nonumber\\
		&=	 2\frac{\left(Q_cQ^{-1}_{c_2}q^{2(b-a-\lambda_{1}^{(c_2)})}-1\right)\left(Q_cQ_{c_2}q^{2(b-a+\lambda_{1}^{(c_2)}+1)}-1\right)}{(q^2+1)Q_cq^{2(b-a)}}\nonumber\\
		&\qquad\qquad\cdot\frac{
			\prod\limits_{\beta=(a',b',c_2)\in\mathscr{A}_{c_2}\setminus \{(1,\lambda_{1}^{(c_2)}+1,c_2)\}}\left(\mathtt{q}(\res(\alpha))-\mathtt{q}(\res(\beta))\right)}{\prod\limits_{\beta=(a',b',c_2)\in \mathscr{R}_{c_2}}\left(\mathtt{q}(\res(\alpha))-\mathtt{q}(\res(\beta))\right)}.\label{induction 11}			
	\end{align}				
We can compute by \eqref{difference-residue} \begin{align}
&\qquad\qquad	\frac{\mathtt{q}(\res(\alpha))-\mathtt{q}(\res(i+1,d',c_2))}{\mathtt{q}(\res(\alpha))-\mathtt{q}(\res(i,d,c_2))}\nonumber\\
	&=\frac{Q_cQ_{c_2}q^{2(b-a+d-i)}}{Q_cQ_{c_2}q^{2(b-a+d'-i-1)}}\cdot\frac{\left(Q_cq^{2(b-a)}-Q_{c_2}q^{2(d'-i-1)}\right)\left(Q_cQ_{c_2}q^{2(b-a+d'-i)}-1\right)}{\left(Q_cq^{2(b-a)}-Q_{c_2}q^{2(d-i)}\right)\left(Q_cQ_{c_2}q^{2(b-a+d-i+1)}-1\right)}\nonumber\\
	&=q^{2(d-d'+1)}
	\cdot\frac{\left(Q_cq^{2(b-a)}-Q_{c_2}q^{2(d'-i-1)}\right)\left(Q_cQ_{c_2}q^{2(b-a+d'-i)}-1\right)}{\left(Q_cq^{2(b-a)}-Q_{c_2}q^{2(d-i)}\right)\left(Q_cQ_{c_2}q^{2(b-a+d-i+1)}-1\right)}
	\nonumber\\
	&=q^{2(d-d'+1)}
	\cdot\prod_{k=d'}^d\frac{\left(Q_cq^{2(b-a)}-Q_{c_2}q^{2(k-i-1)}\right)\left(Q_cQ_{c_2}q^{2(b-a+k-i)}-1\right)}{\left(Q_cq^{2(b-a)}-Q_{c_2}q^{2(k-i)}\right)\left(Q_cQ_{c_2}q^{2(b-a+k-i+1)}-1\right)}
	\nonumber\\
	&=q^{2(d-d'+1)}
	\cdot\prod_{k=d'}^d\frac{\left(Q_cq^{2(b-a)}-Q_{c_2}q^{2(k-{\lambda_k^{(c_2)}}'-1)}\right)\left(Q_cQ_{c_2}q^{2(b-a+k-{\lambda_k^{(c_2)}}')}-1\right)}{\left(Q_cq^{2(b-a)}-Q_{c_2}q^{2(k-{\lambda_k^{(c_2)}}')}\right)\left(Q_cQ_{c_2}q^{2(b-a+k-{\lambda_k^{(c_2)}}'+1)}-1\right)}.\label{induction 22}
	\end{align}
		
Combining \eqref{induction 11} with \eqref{induction 22}, we conclude \begin{align*}
		&\qquad\qquad\frac{
		\prod\limits_{\beta=(a',b',c_2)\in\mathscr{A}_{c_2}}\left(\mathtt{q}(\res(\alpha))-\mathtt{q}(\res(\beta))\right)}{\prod\limits_{\beta=(a',b',c_2)\in\mathscr{R}_{c_2}}\left(\mathtt{q}(\res(\alpha))-\mathtt{q}(\res(\beta))\right)}\\
		&=2\frac{\left(Q_cQ^{-1}_{c_2}q^{2(b-a-\lambda_{1}^{(c_2)})}-1\right)\left(Q_cQ_{c_2}q^{2(b-a+\lambda_{1}^{(c_2)}+1)}-1\right)}{(q^2+1)Q_cq^{2(b-a-\lambda_1^{(c_2)})}}\\
	&\qquad
	\cdot\prod_{k=1}^{\lambda^{(c_2)}_1}\frac{\left(Q_cq^{2(b-a)}-Q_{c_2}q^{2(k-{\lambda_k^{(c_2)}}'-1)}\right)\left(Q_cQ_{c_2}q^{2(b-a+k-{\lambda_k^{(c_2)}}')}-1\right)}{\left(Q_cq^{2(b-a)}-Q_{c_2}q^{2(k-{\lambda_k^{(c_2)}}')}\right)\left(Q_cQ_{c_2}q^{2(b-a+k-{\lambda_k^{(c_2)}}'+1)}-1\right)}=\frac	{\mathbb{X}^{\undla}_{{c},{c_2}}}{\mathbb{X}^{\undmu}_{{c},{c_2}}}.
	\end{align*} The argument for $1\leq c_1<c$ is quite similar and easier and we shall omit it. Hence, \eqref{claim 2'} holds. Similarly, we can prove \eqref{claim 3'}. As we have explained before, once claims \eqref{claim 1'},\eqref{claim 2'},\eqref{claim 3'} are true, the proof of the proposition is completed. 
		\qed
		\medskip
		
		We are now in the position to prove Theorem \ref{closed-schur}.
		
	\medskip
	{\bf Proof of Theorem \ref{closed-schur}}:

			This follows from \eqref{eigenvalues}, Proposition \ref{Schur-Non-dengerate} and Proposition \ref{closed q}.
	\qed
	\medskip

\subsection{Cancellation-free formula for some special shapes}\label{Cancellation-free formula for some special shapes}

In this subsection, we shall further simplify the closed formula obtained in Theorem \ref{closed-schur} for some special choices of $\undla\in\mathscr{P}^{\bullet,m}_{n}$.

\begin{defn}
	Let $1\leq u\leq n$, we define $\Gamma_{u,n-u}=(u,\underbrace{1,1,\cdots,1}_{n-u})$.
	\end{defn}

\begin{cor}\label{special shape 1}
	Suppose $\undla\in\mathscr{P}^{\bullet,m}_{n}$, where $\bullet\in\{\mathtt{0},\mathtt{s}\}$ such that $\lambda^{(c)}=\Gamma_{u,n-u}$ for some $1\leq c\leq m$ and $1\leq u \leq n$.
	Then we have the following.
	\begin{enumerate}	
		\item\label{special shape-0} If $\bullet=\mathsf{0}$, \begin{align*}
		s_{\undla}&=\left(\prod\limits_{k=u-n}^{u-1}\prod_{\substack{1\leq c'\leq m\\c\neq c'}}\frac{2\left(Q_{c}Q^{-1}_{c'}q^{2k}-1\right)\left(Q_{c}Q_{{c'}}q^{2(k+1)}-1\right)}{\left(q^2+1\right)Q_{c}q^{2k}}\right)\cdot \left( \frac{q^{2n}-1}{q^2-1}\cdot\prod\limits_{k=1}^{u-1} \frac{q^{2k}-1}{q^2-1}\cdot \prod\limits_{k=1}^{n-u} \frac{q^{2k}-1}{q^2-1}\right)\\
		&\qquad \qquad\cdot\frac{(-2\epsilon)^n\cdot\sqrt{\left(Q_c^2q^{4(u-n)}-1\right)\left(Q_c^2q^{4u}-1\right)}}{q^{2n\left(\lambda^{(c)}\right)}\cdot (q^2-q^{-2})^n\cdot \left(Q_cq^{2u-n}\right)\cdot\prod\limits_{k=u-n+1}^{u-1}\left(Q_cq^{2k}\right)}\cdot \prod_{\substack{u-n+1\leq k\leq u\\k\neq 2u-n}}\left(Q_c^2q^{2k}-1\right).
			\end{align*}
			
		\item\label{special shape-s1}  If $\bullet=\mathsf{s}$, \begin{align*}
			s_{\undla}=&\left(\prod\limits_{k=u-n}^{u-1}\prod_{\substack{1\leq c'\leq m\\c\neq c'}}\frac{2\left(Q_{c}Q^{-1}_{c'}q^{2k}-1\right)\left(Q_{c}Q_{{c'}}q^{2(k+1)}-1\right)}{\left(q^2+1\right)Q_{c}q^{2k}}\right)\cdot \left( \frac{q^{2n}-1}{q^2-1}\cdot\prod\limits_{k=1}^{u-1} \frac{q^{2k}-1}{q^2-1}\cdot \prod\limits_{k=1}^{n-u} \frac{q^{2k}-1}{q^2-1}\right)\\
			&\qquad\cdot\frac{ \left(Q_{c}q^{2u}-1\right)}{q^{2n\left(\lambda^{(c)}\right)}\cdot\prod\limits_{k=u-n+1}^{u-1}\left(Q_c^2q^{2k}+1\right)}\cdot \left(\prod_{\substack{u-n+1\leq k\leq u\\k\neq 2u-n}}\left(Q_c^2q^{2k}-1\right)\right) \cdot\prod\limits_{k=u-n}^{u-1}\left(\frac{2\left(Q_{c}q^{2k}-1\right)}{\left(q^2+1\right)Q_{c}q^{2k}}\right).
	\end{align*}
	\end{enumerate}
	\end{cor}

	\begin{proof}
		We can compute		\begin{align*}
			&\qquad\prod_{\substack{1\leq i<i'\leq \ell(\lambda^{(c)})}}\left(\frac{Q_c^2q^{2\left(\lambda^{(c)}_i-i+\lambda^{(c)}_{i'}-i'+2\right)}-1}{Q_c^2q^{2\left(2-i-i'\right)}-1}\right)\\
			&=\frac{\prod\limits_{k=u-(n-u-1)}^u\left(Q_c^2q^{2k}-1\right)\cdot\prod\limits_{2\leq i<i'\leq n-u+1}\left(Q_c^2q^{2\left(4-i-i'\right)}-1\right)}{\prod\limits_{1\leq i<i'\leq n-u+1}\left(Q_c^2q^{2\left(2-i-i'\right)}-1\right)}\\
			&=\frac{\prod\limits_{k=2u-n+1}^u\left(Q_c^2q^{2k}-1\right)\cdot\prod\limits_{1\leq i<i'\leq n-u}\left(Q_c^2q^{2\left(2-i-i'\right)}-1\right)}{\prod\limits_{1\leq i<i'\leq n-u+1}\left(Q_c^2q^{2\left(2-i-i'\right)}-1\right)}\\
			&=\frac{\prod\limits_{k=2u-n+1}^u\left(Q_c^2q^{2k}-1\right)}{\prod\limits_{k=2u-2n+1}^{u-n}\left(Q_c^2q^{2k}-1\right)}
			\end{align*}
			and
			\begin{align*}
				\prod_{(a,b,c)\in\lambda^{(c)}}\left(\frac{Q_c^2q^{2\left(b-a-\ell(\lambda^{(c)})+1\right)}-1}{Q_c^2q^{4\left(b-a\right)}-1}\right)&=\frac{\prod\limits_{k=2u-2n}^{2u-n-1}\left(Q_c^2q^{2k}-1\right)}{\prod\limits_{k=u-n}^{u-1}\left(Q_c^2q^{4k}-1\right)}.
				\end{align*}
		Therefore, by definition, we have
		\begin{align}
			\mathbb{Y}^{\undla}_{c}&=
			\frac{1}{q^{2n\left(\lambda^{(c)}\right)}}\cdot\left( \prod_{\substack{1\leq i<i'\leq \ell(\lambda^{(c)})}}\left(\frac{Q_c^2q^{2\left(\lambda^{(c)}_i-i+\lambda^{(c)}_{i'}-i'+2\right)}-1}{Q_c^2q^{2\left(2-i-i'\right)}-1}\right)\right)\nonumber\\
			&\qquad\qquad\cdot \prod_{(a,b,c)\in\lambda^{(c)}}\left(\frac{Q_c^2q^{2\left(b-a-\ell(\lambda^{(c)})+1\right)}-1}{Q_c^2q^{4\left(b-a\right)}-1}
			\cdot \frac{q^{2h^{\lambda^{(c)}}_{a,b}}-1}{q^2-1}\right)\nonumber\\
			&=\frac{1}{q^{2n\left(\lambda^{(c)}\right)}}\cdot\left(\prod_{(a,b,c)\in\lambda^{(c)}} \frac{q^{2h^{\lambda^{(c)}}_{a,b}}-1}{q^2-1}\right)\cdot \frac{\prod\limits_{k=2u-n+1}^u\left(Q_c^2q^{2k}-1\right)}{\prod\limits_{k=2u-2n+1}^{u-n}\left(Q_c^2q^{2k}-1\right)}\cdot \frac{\prod\limits_{k=2u-2n}^{2u-n-1}\left(Q_c^2q^{2k}-1\right)}{\prod\limits_{k=u-n}^{u-1}\left(Q_c^2q^{4k}-1\right)}\nonumber\\
			&=\frac{1}{q^{2n\left(\lambda^{(c)}\right)}}\cdot\left(\prod_{(a,b,c)\in\lambda^{(c)}} \frac{q^{2h^{\lambda^{(c)}}_{a,b}}-1}{q^2-1}\right)\cdot \frac{ \prod_{\substack{2u-2n\leq k\leq u\\k\neq 2u-n}}\left(Q_c^2q^{2k}-1\right)}{\prod\limits_{k=2u-2n+1}^{u-n}\left(Q_c^2q^{2k}-1\right)\cdot\prod\limits_{k=u-n}^{u-1}\left(Q_c^2q^{4k}-1\right)}\nonumber\\
			&=\frac{1}{q^{2n\left(\lambda^{(c)}\right)}}\cdot\left(\prod_{(a,b,c)\in\lambda^{(c)}} \frac{q^{2h^{\lambda^{(c)}}_{a,b}}-1}{q^2-1}\right)\cdot \frac{ \prod_{\substack{u-n+1\leq k\leq u\\k\neq 2u-n}}\left(Q_c^2q^{2k}-1\right)}{\prod\limits_{k=u-n+1}^{u-1}\left(Q_c^2q^{4k}-1\right)}\nonumber\\
		&=\frac{1}{q^{2n\left(\lambda^{(c)}\right)}}\cdot\left( \frac{q^{2n}-1}{q^2-1}\cdot\prod\limits_{k=1}^{u-1} \frac{q^{2k}-1}{q^2-1}\cdot \prod\limits_{k=1}^{n-u} \frac{q^{2k}-1}{q^2-1}\right)\cdot \frac{ \prod_{\substack{u-n+1\leq k\leq u\\k\neq 2u-n}}\left(Q_c^2q^{2k}-1\right)}{\prod\limits_{k=u-n+1}^{u-1}\left(Q_c^2q^{4k}-1\right)}	\label{Y-reduced}.
		\end{align}
		For $1\leq c'\neq c\leq m$, we have $\lambda^{(c')}=\emptyset$ by condition. It follows that for $1\leq c_1<c$, we have 
		\begin{align}\label{X-reduced1}
			\mathbb{X}^{\undla}_{{c_1},{c}}&=\prod_{\substack{(a,b,c)\in\lambda^{(c)}}}\left(\frac{2\left(Q_{c}Q^{-1}_{c_1}q^{2(b-a)}-1\right)\left(Q_{c}Q_{{c_1}}q^{2(b-a+1)}-1\right)}{\left(q^2+1\right)Q_{c}q^{2(b-a)}}\right)\nonumber\\
			&=\prod\limits_{k=u-n}^{u-1}\left(\frac{2\left(Q_{c}Q^{-1}_{c_1}q^{2k}-1\right)\left(Q_{c}Q_{{c_1}}q^{2(k+1)}-1\right)}{\left(q^2+1\right)Q_{c}q^{2k}}\right).
			\end{align}
	Similarly, for $c<c_2\leq m$, we have
				\begin{align}\label{X-reduced2}
				\mathbb{X}^{\undla}_{{c},{c_2}}=\prod\limits_{k=u-n}^{u-1}\left(\frac{2\left(Q_{c}Q^{-1}_{c_2}q^{2k}-1\right)\left(Q_{c}Q_{{c_2}}q^{2(k+1)}-1\right)}{\left(q^2+1\right)Q_{c}q^{2k}}\right).
			\end{align}
			
	We first prove \eqref{special shape-0}. In this case, we can compute \begin{align}\prod\limits_{(a,b,c)\in\undla}\left(\sqrt{[b-a+1]_{c,q^2}[b-a]_{c,q^2}}\right)
		&=\prod\limits_{k=u-n}^{u-1}\left(\sqrt{[k+1]_{c,q^2}[k]_{c,q^2}}\right)\nonumber\\
		&=\sqrt{[u-n]_{c,q^2}}\sqrt{[u]_{c,q^2}}\left(\prod\limits_{k=u-n+1}^{u-1}[k]_{c,q^2}\right)\nonumber\\
	&=	\frac{\sqrt{\left(Q_c^2q^{4(u-n)}-1\right)\left(Q_c^2q^{4u}-1\right)}}{(q^2-q^{-2})Q_cq^{2u-n}}\cdot \left(\prod\limits_{k=u-n+1}^{u-1}\frac{Q_c^2q^{4k}-1}{(q^2-q^{-2})Q_cq^{2k}}\right)\label{reduced difference b}.
		\end{align} Now \eqref{special shape-0} follows from Theorem \ref{closed-schur}, \eqref{Y-reduced}, \eqref{X-reduced1}, \eqref{X-reduced2} and \eqref{reduced difference b}.
	
	Suppose $\bullet=\mathsf{s}$. Then $\lambda^{(0)}=\emptyset$ by condition and we can compute \begin{align}
			\mathbb{X}^{\undla}_{{0},{c}}&=\prod_{\substack{(a,b,c)\in\lambda^{(c)}}}\left(\frac{2\left(Q_{c}q^{2(b-a)}-1\right)\left(Q_{{c}}q^{2(b-a+1)}-1\right)}{\left(q^2+1\right)Q_{c}q^{2(b-a)}}\right)\nonumber\\
			&=\prod\limits_{k=u-n}^{u-1}\left(\frac{2\left(Q_{c}q^{2k}-1\right)\left(Q_{{c}}q^{2(k+1)}-1\right)}{\left(q^2+1\right)Q_{c}q^{2k}}\right)\nonumber\\
			&=\prod\limits_{k=u-n}^{u-1}\left(\frac{2\left(Q_{c}q^{2k}-1\right)}{\left(q^2+1\right)Q_{c}q^{2k}}\right)\cdot \prod\limits_{k=u-n+1}^{u}\left(Q_{c}q^{2k}-1\right)\label{X-reduced3}.
		\end{align}
	Hence, \eqref{special shape-s1} follows from Theorem \ref{closed-schur}, \eqref{Y-reduced}, \eqref{X-reduced1}, \eqref{X-reduced2} and \eqref{X-reduced3}.
		\end{proof}

\begin{cor}\label{special shape 2}
	Suppose $\undla\in\mathscr{P}^{\mathtt{s},m}_{n}$ satisfying $\lambda^{(c)}=\emptyset$ for any $1\leq c\leq m$. Then we have
	\begin{align*}s_{\undla}=&2^{\lfloor \sharp\mathcal{D}_{\undla}/2 \rfloor}\cdot \frac{\left(\prod\limits_{1\leq i<i'\leq \ell(\lambda^{(0)})+1}\left(q^{2\left(\lambda^{(0)}_i+\lambda^{(0)}_{i'}\right)}-1\right)\right)\cdot \left(\prod\limits_{j=\ell(\lambda^{(0)})+1}^{\lambda^{(0)}_1}\prod\limits_{(i,j,0)\in\lambda^{(0)}}\left(q^{2h_{i,j}^{\widetilde{\lambda^{(0)}}}}-1\right)\right)}{q^{2n\left(\lambda^{(0)}\right)}\cdot (q^2-1)^n\cdot\prod\limits_{(i,j,0)\in\lambda^{(0)}}\left(q^{2(j-i)}+1\right) }\\
	&\qquad\qquad\cdot\prod\limits_{c=1}^m\prod_{\substack{(a,b,0)\in\lambda^{(0)}}}\left(\frac{2\left(Q^{-1}_{c}q^{2(b-a)}-1\right)\left(Q_{{c}}q^{2(b-a+1)}-1\right)}{\left(q^2+1\right)q^{2(b-a)}}\right).
	\end{align*} In particular, if $\lambda^{(0)}=(n)$,	\begin{align*}s_{\undla}=& \frac{\prod\limits_{k=1}^n\left(q^{2k}-1\right)}{q^{2n\left(\lambda^{(0)}\right)}\cdot (q^2-1)^n\cdot\prod\limits_{k=0}^{n-1}\left(q^{2k}+1\right)}
	\cdot\prod\limits_{c=1}^m\prod\limits_{k=0}^{n-1}\left(\frac{2\left(Q^{-1}_{c}q^{2k}-1\right)\left(Q_{{c}}q^{2(k+1)}-1\right)}{\left(q^2+1\right)q^{2k}}\right).
	\end{align*} 
	\end{cor}
\begin{proof}
This directly follows from Theorem \ref{closed-schur}.
\end{proof}

\section{On two trace functions of Hecke-Clifford superalgebra}\label{On two trace functions of Hecke-Clifford superalgebra}

 The Hecke-Clifford superalgebra $\mathcal{HC}_{\rm R}(n)$ over ${\rm R}$ is defined as $\mathcal{H}^f_{\rm R}(n)$,  where $f=f^{(\mathtt{s})}_{\emptyset}=X_1-1$. Following \cite{JN}, $\mathcal{HC}_{\rm R}(n)$ can be viewed as a subalgebra of $\mathcal{H}_{\rm R}(n)$ generated by $T_1,\cdots, T_{n-1}, C_1,\cdots,C_n$. 

Let $\mathcal{OP}_{n}$ be the set of odd partition of $n$, i.e., each non-zero component is odd. For each $\nu\in\mathcal{OP}_{n}$, we can associate an element $T_{w_\nu}\in \mathcal{HC}_{\rm R}(n)$\cite[(3.2)]{WW2}. 

\begin{thm}\cite[Theorem 4.8]{WW2}\label{cocenter}
Every trace function $\phi:\mathcal{HC}_{\rm R}(n)\rightarrow {\rm R}$ is uniquely
determined by its values on the elements $T_{w_{\nu}}$ for
$\nu\in\mathcal{OP}_{n}$.
\end{thm}

By Theorem \ref{cocenter}, Wan and Wang defined a trace function $\gimel^{\rm R}_n: \mathcal{HC}_{\rm R}(n)\rightarrow {\rm R}$ which is
characterized by the conditions
\begin{align*}
	\gimel^{\rm R}_n(T_{w_{\nu}})&= \Big(\frac{q^2-1}{2}\Big)^{n-\ell(\nu)},
	\qquad \text{ for }\nu\in\mathcal{OP}_{n}, \\
	\gimel^{\rm R}_n(z)&=0, \qquad \text{ for }z\in\mathcal{HC}_{\rm R}(n)_{\bar{1}}.
\end{align*}

Let ${\bf A}:=\Z[\frac{1}{2}][q,q^{-1}],$ where $q$ is an indeterminate over $\Z$ and $\mathcal{HC}_{\bf A}(n)$ be the generic Hecke-Clifford superalgebra defined over ${\bf A}$. We denote $\mathbb{F}$ the algebraic closure of $\C(q)$.

\begin{thm} \cite[Theorem 5.8]{WW2}\label{th:spinSchur}
	$\mathbb{F}\otimes_{\bf A}\gimel_n$ is a symmetrizing trace form on $\mathcal{HC}_{\mathbb{F}}(n)$. For
	$\lambda\in\mathscr{P}^{\mathtt{s}}_{n}$, the Schur element $c^{\lambda}$ of the simple
	$\mathcal{HC}_{\mathbb{F}}(n)$-module  $\mathbb{D}_{\lambda}$ with respect to $\mathbb{F}\otimes_{\bf A}\gimel_n$ is given by
	\begin{equation}\label{eq:elemSchur}
		c^{\lambda}=2^{n+\lfloor \sharp\mathcal{D}_{\undla}/2 \rfloor} \frac{\prod\limits_{(i,j)\in
				\lambda}(q^{2h^{\tilde{\lambda}}_{i,j}}-1)}{q^{2n(\lambda)}(q^2-1)^n\prod\limits_{(i,j)\in
				\lambda}(1+q^{2(j-i)})}.
	\end{equation}
\end{thm}

\begin{prop}\cite[Proposition 1.3]{LS3}\label{Hecke-Cliffod sym}
	Suppose that $q$ is not a primitive $4l$-th root of unity for $l\in [n],$ then $\mathcal{HC}_{\mathbb{K}}(n)$ is a symmetric superalgebra with the symmetrizing form $t^{\mathbb{K}}_{1,n}$.
\end{prop}

\begin{cor}
	Suppose that $q$ is not a primitive $4l$-th root of unity for $l\in [n].$ Then  $\{C^{\beta}T_w(1+X_1)^{-1}(1+X_2)^{-1}\cdots (1+X_n)^{-1}~|~ 
	\beta\in\mathbb{Z}_2^n, w\in {\mathfrak{S}_n}\}$ forms an $\mathbb{K}$-basis of $\mathcal{HC}_{\mathbb{K}}(n)$. Moreover, we have
	\begin{align*}
		t^{\mathbb{K}}_{1,n}\left( C^{\beta}T_w(1+X_1)^{-1}(1+X_2)^{-1}\cdots (1+X_n)^{-1} \right)=\delta_{(\beta,w),(0,1)}.
	\end{align*}
\end{cor}
\begin{proof}
It follows from the proof of \cite[Proposition 7.9]{LS3} and \cite[Proposition 1.3]{LS3} that $(1+X_1)(1+X_2)\cdots (1+X_n)$ is invertible in $\mathcal{HC}_{\mathbb{K}}(n)$. Combining with Lemma \ref{basis}, we prove the first statement. The second part follows from Proposition \ref{trace formula} and Proposition \ref{Nondengerate} (2).
	\end{proof}
Recall that $s_{\lambda}$ is the  Schur element of simple $\mathcal{HC}_{\mathbb{F}}(n)$-module $\mathbb{D}(\lambda)$ with respect to $t^{\mathbb{F}}_{1,n}$, for $\lambda\in\mathscr{P}^{\mathtt{s}}_{n}$. 

\begin{lem}\label{schur=schur}
	Let $\lambda\in\mathscr{P}^{\mathtt{s}}_{n}$. Then $c^{\lambda}=2^n\cdot s_\lambda$.
	\end{lem}
	
	\begin{proof}
		For $1\leq i\leq j\leq \ell(\lambda)$, we can check by definition that $$
		h^{\tilde{\lambda}}_{i,j}=\lambda_i+\lambda_{j+1}.
		$$ It follows that $$\prod\limits_{1\leq i<i'\leq \ell(\lambda^{(0)})+1}\left(q^{2\left(\lambda^{(0)}_i+\lambda^{(0)}_{i'}\right)}-1\right)=\prod\limits_{1\leq i\leq j\leq \ell(\lambda)}(q^{2h^{\tilde{\lambda}}_{i,j}}-1).$$ 
		Hence,$$
		\left(\prod\limits_{1\leq i<i'\leq \ell(\lambda^{(0)})+1}\left(q^{2\left(\lambda^{(0)}_i+\lambda^{(0)}_{i'}\right)}-1\right)\right)\cdot \left(\prod\limits_{j=\ell(\lambda^{(0)})+1}^{\lambda^{(0)}_1}\prod\limits_{(i,j,0)\in\lambda^{(0)}}\left(q^{2h_{i,j}^{\widetilde{\lambda^{(0)}}}}-1\right)\right)=\prod\limits_{(i,j)\in
			\lambda}(q^{2h^{\tilde{\lambda}}_{i,j}}-1).$$
		Combining with Corollary \ref{special shape 2} and \eqref{eq:elemSchur}, we complete the proof of the Lemma.
		\end{proof}
		
			\medskip
		{\bf Proof of Proposition \ref{trace function}}:	By base change, we have $t_{1,n}^{\rm R}=1_{\rm R}\otimes_{\bf A}t_{1,n}^{\bf A}$ and $\gimel_n^{\rm R}=1_{\rm R}\otimes_{\bf A}\gimel_n^{\bf A}$. Hence, we only need to prove $\gimel_n^{\bf A}=\frac{1}{2^n}t_{1,n}^{\bf A}.$ To this end, we claim $\gimel_n^{\mathbb{F}}=\frac{1}{2^n}t_{1,n}^{\mathbb{F}}.$ Since $\mathbb{F}$ is the algebraic closure of $\C(q)$, $\mathcal{HC}_{\mathbb{F}}(n)$ is semisimple over $\mathbb{F}$. Therefore, our claim follows from Proposition \ref{schur formula 1} and Lemma \ref{schur=schur}. This completes the proof of the proposition.
\qed
\medskip
				
\section{Semisimplicity criterion for cyclotomic Hecke-Clifford superalgebra}\label{Semisimplicity criterion for cyclotomic Hecke-Clifford superalgebra}

{\bf Throughout this subsection, ${\rm R}=\mathbb{K}$ is the algebraically closed field of characteristic different from $2$. We fix $n\in\N, q^2\neq \pm 1$, $\undQ=(Q_1,\cdots,Q_m)\in(\mathbb{K}^*)^m$ and $f=f^{(\bullet)}_{\underline{Q}}$ with $\bullet\in\{\mathtt{0},\mathtt{s}\}$.  We shortly denote $\mathcal{H}^f_{\mathbb{K}}:=\mathcal{H}^f_{\mathbb{K}}(n)$. Let $x$ be an indeterminate, we set $\hO:=\mathbb{K}[[x]]=\{a_0+a_1x+a_2x^2+\cdots|~a_i\in \mathbb{K}\}$ and $\hK$ be the fraction field of $\hO$. We modify the parameters as follows: $q':=x^4+q,$ $Q'_i:=x^{8ni}+q^{-1}Q_i, 1\leq i\leq m$. Then we can define $\mathcal{H}^{f'}_{\hO}:=\mathcal{H}^{f'}_{\hO}(n)$, where $$f'=\begin{cases}
		f^{\mathsf{(0)}}_{\underline{Q'}}=\prod_{i=1}^m \biggl(X_1+X^{-1}_1-\mathtt{q}(Q_i')\biggr), &\qquad\text{ if $\bullet=\mathtt{0},$}\\
		f^{\mathsf{(s)}}_{\underline{Q'}}=(X_1-1)\prod_{i=1}^m \biggl(X_1+X^{-1}_1-\mathtt{q}(Q_i')\biggr),&\qquad\text{ if $\bullet=\mathtt{s}.$}
	\end{cases}$$  Similarly, we can define $\mathcal{H}^{f'}_{\hK}:=\mathcal{H}^{f'}_{\hK}(n)$.} Then we have $$
\mathcal{H}^{f'}_{\hK}\cong \hK\otimes_{\hO} \mathcal{H}^{f'}_{\hO},\qquad\mathcal{H}^{f}_{\mathbb{K}}\cong \mathbb{K}\otimes_{\hO}\mathcal{H}^{f'}_{\hO}.
$$ One can check $P_n^{(\bullet)}({q'}^2,\undQ')\neq 0$, hence $\mathcal{H}^{f'}_{\hK}$ is semisimple over $\hK$ by Theorem \ref{semisimple:non-dege}.   {\bf Accordingly, we define the residues of boxes in the young diagram $\undla$ via \eqref{eq:residue} as well as $\res(\mathfrak{t})$ for each $\mathfrak{t}\in\Std(\undla)$ with $\undla\in\mathscr{P}^{\bullet,m}_{n}$ with $m\geq 0$ with respect to parameters $(q',Q'_1,\cdots,Q'_m)$. }

One can check that all the eigenvalues $\mathtt{b}_{\pm}(\res_{\mt}(k))$ of $X_k$ belong to $\mathbb{K}[[x^2]]\subset\hO$. Furthermore, 
\begin{align}\label{integral}\sqrt{1\pm {q'}^2{Q'_c}^{2t}}\in \hO, \text{ for any $t\in \Z$ and $1\leq c\leq m$. }	
	\end{align} We use $t^{\hO}_{r,n}$ to denote the natural map $\mathcal{H}^{f'}_{\hO}\to\hO$. For $a\in \hO$, we use $a|_{x=0}\in \mathbb{K}$ to denote the image of $a$ in residue field $\mathbb{K}\cong \hO/(x)$.

		\medskip
	{\bf Proof of Theorem \ref{semisimplicity for cyclotomic Hecke-Clifford} (1)}: If $P^{(\mathtt{0})}_{n}(q^2,\undQ)\neq 0$, then $\mathcal{H}^{f}_{\mathbb{K}}$ is (split) semisimple by Theorem \ref{semisimple:non-dege}. Assume that $\mathcal{H}^{f}_{\mathbb{K}}$ is (split) semisimple. By Proposition \ref{Nondengerate} (1), $\mathcal{H}^{f'}_{\hO}$ is supersymmetric. Therefore, we can apply Theorem \ref{modular-semisimple 2} (1), and deduce that for all $\undla\in\mathscr{P}^{\mathtt{0},m}_{n}$, we have $s_{\undla}|_{x=0}\neq 0$. We shall prove $P^{(\mathtt{0})}_{n}(q^2,\undQ)\neq 0$ by choosing some special $\undla\in\mathscr{P}^{\mathtt{0},m}_{n}$.
	
	\begin{enumerate}
		\item Let $\undla=(\Gamma_{1,n-1},\underbrace{\emptyset,\cdots,\emptyset}_{m-1})$, then by Corollary \ref{special shape 1} (1) and \eqref{integral}, we deduce $\prod\limits_{t=1}^{n}\bigl(q^{2t}-1\bigr)\neq 0.$
		
		\item We fix $1\leq i\leq m$. Let $\undla\in\mathscr{P}^{\mathtt{0},m}_{n}$ such that $\lambda^{(i)}=\Gamma_{1,n-1}$. Then by Corollary \ref{special shape 1} (1) and \eqref{integral}, we deduce $\prod\limits_{t=3-n}^{1}\bigl(Q^2_i-q^{-2t}\bigr)\neq 0$. Similarly, let $\undmu\in\mathscr{P}^{\mathtt{0},m}_{n}$ such that $\mu^{(i)}=\Gamma_{n,0}$. Then by Corollary \ref{special shape 1} (1) and \eqref{integral}, we deduce $\prod\limits_{t=1}^{n-1}\bigl(Q^2_i-q^{-2t}\bigr)\neq 0$. Hence, $$\prod\limits_{t=3-n}^{n-1}\bigl(Q^2_i-q^{-2t}\bigr)\neq 0.$$
		
		\item We fix $1\leq i\leq m$. Let $\undla\in\mathscr{P}^{\mathtt{0},m}_{n}$ such that $\lambda^{(i)}=\Gamma_{u,n-u}$ for some $1\leq u\leq n$. Then by Corollary \ref{special shape 1} (1) and \eqref{integral}, we deduce 
		$\bigl(Q^2_i-q^{-4u}\bigr)\bigl(Q^2_i-q^{-4(u-n)}\bigr)\neq 0.$ Hence,  $$\prod\limits_{u=1}^n\left(\bigl(Q^2_i-q^{-4u}\bigr)\bigl(Q^2_i-q^{-4(u-n)}\bigr)\right)=\prod\limits_{t=1-n}^{n}\bigl(Q^2_i-q^{-4t}\bigr)\neq 0.$$
		
		\item We fix $1\leq i<i'\leq m$. Let $\undla\in\mathscr{P}^{\mathtt{0},m}_{n}$ such that $\lambda^{(i)}=\Gamma_{1,n-1}$. Then by Corollary \ref{special shape 1} (1) and \eqref{integral}, we deduce
		$\prod\limits_{t=1-n}^{0}\bigl({Q_i}-Q_{i'}q^{-2t}\bigr)
		\bigl(Q_iQ_{i'}-q^{-2(t+1)}\bigr)\neq 0.$ Similarly, let $\undmu\in\mathscr{P}^{\mathtt{0},m}_{n}$ such that $\mu^{(i)}=\Gamma_{n,0}$. Then by Corollary \ref{special shape 1} (1) and \eqref{integral}, we deduce $\prod\limits_{t=0}^{n-1}\bigl({Q_i}-Q_{i'}q^{-2t}\bigr)
		\bigl(Q_iQ_{i'}-q^{-2(t+1)}\bigr)\neq 0.$  Hence $$\prod\limits_{t=1-n}^{n-1}\bigl({Q_i}-Q_{i'}q^{-2t}\bigr)
		\bigl(Q_iQ_{i'}-q^{-2(t+1)}\bigr)\neq 0.$$
		\end{enumerate}
		
		To sum up, we complete the proof of $P^{(\mathtt{0})}_{n}(q^2,\undQ)\neq 0$. 
	\qed
	\medskip
	
In the rest of this subsection, we assume $\bullet=\mathtt{s}$.

\begin{lem}\cite[Corollary 6.28]{LS4}\label{eigenvalue}
	Let $i\in[n]$. Then $\tilde{a}$ is an eigenvalue of $X_i$ on $\mathcal{H}^{f}_{\mathbb{K}}$ if and only if $\tilde{a}=a|_{x=0}$ and $a$ is an eigenvalue of $X_i$ on $\mathcal{H}^{f'}_{\mathscr{K}}$.
	\end{lem}
	
	Recall $Q_0:=1.$ We define
	\begin{align}\label{Gamma poly}
		\Gamma_n^{(\mathsf{s})}(q^2,\underline{Q})
		:=\prod_{i=0}^{m}\prod_{k=1-n}^{n}
		\left(q^{2k}Q_i+1\right)\in\mathbb{K}.
	\end{align}
	
	\begin{prop}\label{Gamma condition}
		Suppose $\Gamma_n^{(\mathtt{s})}(q^2,\underline{Q})\neq 0$ in $\mathbb{K}.$ Then all of the eigenvalues of $1+X_j$ on $\mathcal{H}^{f}_{\mathbb{K}}$ ($1\leq j\leq n$) are non-zero and $\mathcal{H}^{f}_{\mathbb{K}}$ is a symmetric algebra with symmetrizing form $t^{\mathbb{K}}_{r,n}.$
	\end{prop}
	
	\begin{proof}
	The first statement is contained in the proof of \cite[Proposition 7.9]{LS3}, and the second statement is \cite[Proposition 7.9]{LS3}.
		\end{proof}
	\begin{cor}\label{symmetric-O}
		Suppose $\Gamma_n^{(\mathsf{s})}(q^2,\underline{Q})\neq 0$ in $\mathbb{K}.$ Then $\mathcal{H}^{f'}_{\hO}$ is symmetric with the symmetrizing form $t^{\hO}_{r,n}$.
		\end{cor}
		
		\begin{proof}
		Let $k\in[n]$. By construction, all of the eigenvalues of $X_k$ on $\mathcal{H}^{f'}_{\hK}$ belong to $\hO$. By Lemma \ref{eigenvalue} and Proposition \ref{Gamma condition}, all of the eigenvalues of $1+X_k$ on $\mathcal{H}^{f'}_{\hK}$ ($1\leq j\leq n$) are invertible in $\hO$. Hence, the corollary follows from Corollary \ref{symmetric over ring}.
					\end{proof}
	
	\medskip
	{\bf Proof of Theorem \ref{semisimplicity for cyclotomic Hecke-Clifford} (2)}: If $P^{(\mathtt{s})}_{n}(q^2,\undQ)\neq 0$, then $\mathcal{H}^{f}_{\mathbb{K}}$ is (split) semisimple by Theorem \ref{semisimple:non-dege}. Assume that $\mathcal{H}^{f}_{\mathbb{K}}$ is (split) semisimple. By Corollary \ref{symmetric-O}, $\mathcal{H}^{f'}_{\hO}$ is symmetric. Therefore, we can apply Theorem \ref{modular-semisimple 2} (2), and deduce that for all $\undla\in\mathscr{P}^{\mathtt{s},m}_{n}$, we have $s_{\undla}|_{x=0}\neq 0$. We shall prove $P^{(\mathtt{s})}_{n}(q^2,\undQ)\neq 0$ by choosing some special $\undla\in\mathscr{P}^{\mathtt{s},m}_{n}$. 
	
	\begin{enumerate}
		\item Let $\undla\in\mathscr{P}^{\mathtt{s},m}_{n}$ such that $\lambda^{(0)}=(n)$. Then by Corollary \ref{special shape 2}, we deduce $\prod\limits_{t=1}^{n}\bigl(q^{2t}-1\bigr)\neq 0$. Combining with $\Gamma_n^{(\mathsf{s})}(q^2,\underline{Q})\neq 0$, we have $$
		\prod\limits_{t=1}^{n}\biggl(\bigl(q^{2t}-1\bigr)\bigl(q^{2t}+1\bigr)\biggr)\neq 0.
		$$
		\item We fix $1\leq i\leq m$. The proof of $$\prod\limits_{t=3-n}^{n-1}\bigl(Q^2_i-q^{-2t}\bigr)\neq 0$$ is exactly the same as the case $\bullet=\mathtt{0}$.
		
		\item We fix $1\leq i\leq m$. Let $\undla$ such that $\lambda^{(i)}=\Gamma_{1,n-1}$. Then by Corollary \ref{special shape 1} (2) and $\Gamma_n^{(\mathsf{s})}(q^2,\underline{Q})\neq 0$, we deduce 
		$\prod\limits_{t=1-n}^{1}\bigl(Q_i-q^{-2t}\bigr)\neq 0$.  Similarly, let $\undmu$ such that $\mu^{(i)}=\Gamma_{n,0}$. Then by Corollary \ref{special shape 1} (2) and $\Gamma_n^{(\mathsf{s})}(q^2,\underline{Q})\neq 0$, we deduce
		$\prod\limits_{t=0}^{n}\bigl(Q_i-q^{-2t}\bigr)\neq 0$.  Combining with $\Gamma_n^{(\mathsf{s})}(q^2,\underline{Q})\neq 0$, we have
		 $$\prod\limits_{t=1-n}^n\left(\bigl(Q_i-q^{-2t}\bigr)\bigl(Q_i+q^{-2t}\bigr)\right)=\prod\limits_{t=1-n}^{n}\bigl(Q^2_i-q^{-4t}\bigr)\neq 0.$$
		
		\item We fix $1\leq i<i'\leq m$. The proof of $$\prod\limits_{t=1-n}^{n-1}\bigl({Q_i}-Q_{i'}q^{-2t}\bigr)
		\bigl(Q_iQ_{i'}-q^{-2(t+1)}\bigr)\neq 0$$ is exactly the same as the case $\bullet=\mathtt{0}$.

	\end{enumerate}
	
	To sum up, we complete the proof of $P^{(\mathtt{s})}_{n}(q^2,\undQ)\neq 0$. 
	\qed
	\medskip
	
\section{Cyclotomic quiver Hecke superalgebra and cyclotomic quiver Hecke-Clifford superalgebra}\label{Cyclotomic quiver Hecke superalgebra and cyclotomic quiver Hecke-Clifford superalgebra}
\subsection{Dynkin quivers}\label{quivers}
Let $I$ be a finite set. An integral matrix $(a_{ij})_{i,j\in I}$ is called a Cartan matrix if it satisfies: i) $a_{ii}=2$, ii) $a_{ij}\leq 0$ for $i\neq j$, iii) $a_{ij}=0$ if and only if $a_{ji}=0$. We say ${\rm{A}}$ is symmetrizable if there is a diagonal matrix ${\rm{D}}={\rm{diag}}(\rd_i\in\Z_{>0}|i\in I)$ such that ${\rm{DA}}$ is symmetric.

Let $\bigl({\rm{A}}=(a_{ij})_{i,j\in I},P,\Pi,\Pi^\vee\bigr)$ be a Cartan superdatum in the sense of \cite[\S4.1]{KKO2}. That means, \begin{enumerate}
	\item[CS1)] ${\rm{A}}$ is a symmetrizable Cartan matrix;
	\item[CS2)] $P$ is a free abelian group, which is called the weight lattice;
	\item[CS3)] $\Pi=\{\alpha_i\in P|i\in I\}$, called the set of simple roots, is $\Z$-linearly independent;
	\item[CS4)] $\Pi^\vee=\{h_i\in P|i\in I\}\subset P^\vee=\Hom_\Z(P,\Z)$, called the set of simple coroots, satisfies that $\<h_i,\alpha_j\>=a_{ij}$ for all $i,j\in I$;
	\item[CS5)] there is a decomposition $I=I_{\rm{even}}\sqcup I_{\rm{odd}}$ such that \begin{equation}\label{evenodd}
		a_{ij}\in 2\Z, \quad \text{for all $i\in I_{\rm{odd}}$ and $j\in I$.}
	\end{equation}
\end{enumerate}

The diagonal matrix ${\rm{D}}$ gives rise to a symmetric bilinear form $(-|-)$ on $P$ which satisfies: $$
(\nu_i|\lambda)=\rd_i\<h_i,\lambda\>\quad \text{for all $\lambda\in P$.}
$$
In particular, we have $(\nu_i|\nu_j)=\rd_i a_{ij}$ and hence $\rd_i=(\nu_i|\nu_i)/2$ for each $i\in I$.

We define the root lattice $Q:=\oplus_{i\in I}\Z\nu_i$
and the positive root lattice $Q^+:=\oplus_{i\in I}\Z_{\geq 0}\nu_i$.
For any $\nu=\sum_{i\in I}k_i\nu_i\in Q^+$, we define ${\rm ht}(\nu):=\sum_{i\in I}k_i$.
For any $n\in\Z_{\geq 0}$, we define $Q_n^+:=\{\nu\in Q^+\mid {\rm ht}(\nu)=n\}$.
Let $P^+:=\{\Lambda\in P|\text{$\<h_i,\Lambda\>\in\Z_{\geq 0}$ for all $i\in I$}\}$.
Any element $\Lambda\in P^+$ is called a dominant integral weight.

For a Cartan superdatum $({\rm{A}},P,\Pi,\Pi^\vee)$,
we define the parity function $|\cdot|: I\rightarrow\{\overline{0},\overline{1}\}$ by
\begin{equation}\label{pi1}
	\overline{i}:=\begin{cases} \overline{1}, &\text{if $i\in I_{\rm{odd}}$,}\\
		\overline{0}, &\text{if $i\in I_{\rm{even}}$.}
	\end{cases}
\end{equation}

{\bf In the rest of this paper, we shall consider one of the following quivers with the corresponding index set $I=\{0,1,\cdots,e\}$.}

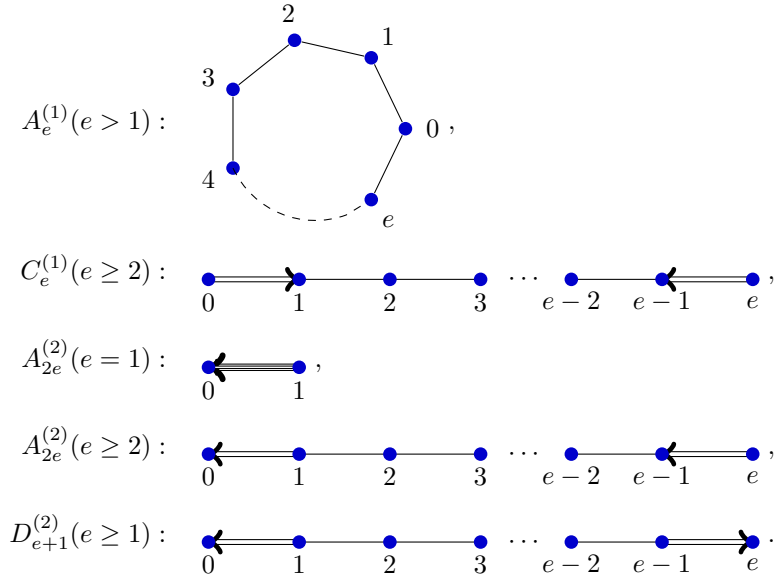
\begin{align*}A^{(1)}_{e}(e>1):\quad&
	\begin{tikzpicture}[scale=1.2,baseline]
		\foreach \r [remember=\r as \rr] in {0,...,4} {
			\node[circle,inner sep=1.8pt,fill=blue!80!black] (\r) at (360/7*\r:1){};
			\node at (360/7*\r:1.3){$\r$};
			\ifnum\r>0\draw(\rr)--(\r);\fi
		}
		\node[circle,inner sep=1.8pt,fill=blue!80!black] (6) at (360/7*6:1){};
		\node at (360/7*6:1.3){$e$};
		\draw[dashed](4) arc [start angle=205, end angle=290, radius=1] -- (6);
		\draw(6)--(0);
	\end{tikzpicture},
	\\[1mm]
	C^{(1)}_{e}(e\geq 2):\quad&
	\begin{tikzpicture}[scale=1.2,baseline]
		\draw[->, double, double distance=0.5mm] (0,0) -- (1,0);
		\draw[->,double,double distance=0.5mm](6,0)--(5,0);
		\draw(1,0)--(3,0);
		\draw(4,0)--(5,0);
		\node at (3.5,0){$\cdots$};
		\foreach \x in {0,...,3} {
			\node[circle,inner sep=1.8pt,fill=blue!80!black] (\x) at (\x,0){};
			\node at (\x,-0.25){$\x$};
		}
		\foreach \x [evaluate=\x as \c using {int(6-\x)}] in {4,5} {
			\node[circle,inner sep=1.8pt,fill=blue!80!black] (\x) at (\x,0){};
			\node at (\x,-0.25){$e-\c$};
		}
		\node[circle,inner sep=1.8pt,fill=blue!80!black] (6) at (6,0){};
		\node at (6,-0.25){$e$};
	\end{tikzpicture},
	\\[1mm]
	A^{(2)}_{2 e}(e=1):\quad&
	\begin{tikzpicture}[scale=1.2,baseline]
		\draw[->,double, double distance=2pt](1,0)--(0,0);
		\draw[->,double, double distance=0.5pt](1,0)--(0,0);
		\foreach \x in {0,1} {
			\node[circle,inner sep=1.8pt,fill=blue!80!black] (\x) at (\x,0){};
			\node at (\x,-0.25){$\x$};
		}
	\end{tikzpicture},
	\\[1mm]
	A^{(2)}_{2 e}(e\geq 2):\quad&
	\begin{tikzpicture}[scale=1.2,baseline]
		\draw[->,double,double distance=0.5mm](1,0)--(0,0);
		\draw[->,double,double distance=0.5mm](6,0)--(5,0);
		\draw(1,0)--(3,0);
		\draw(4,0)--(5,0);
		\node at (3.5,0){$\cdots$};
		\foreach \x in {0,...,3} {
			\node[circle,inner sep=1.8pt,fill=blue!80!black] (\x) at (\x,0){};
			\node at (\x,-0.25){$\x$};
		}
		\foreach \x [evaluate=\x as \c using {int(6-\x)}] in {4,5} {
			\node[circle,inner sep=1.8pt,fill=blue!80!black] (\x) at (\x,0){};
			\node at (\x,-0.25){$e-\c$};
		}
		\node[circle,inner sep=1.8pt,fill=blue!80!black] (6) at (6,0){};
		\node at (6,-0.25){$e$};
	\end{tikzpicture},
	\\[1mm]
	D^{(2)}_{e+1}(e\geq 1):\quad&
	\begin{tikzpicture}[scale=1.2,baseline]
		\draw[->,double,double distance=0.5mm](1,0)--(0,0);
		\draw[->,double,double distance=0.5mm](5,0)--(6,0);
		\draw(1,0)--(3,0);
		\draw(4,0)--(5,0);
		\node at (3.5,0){$\cdots$};
		\foreach \x in {0,...,3} {
			\node[circle,inner sep=1.8pt,fill=blue!80!black] (\x) at (\x,0){};
			\node at (\x,-0.25){$\x$};
		}
		\foreach \x [evaluate=\x as \c using {int(6-\x)}] in {4,5} {
			\node[circle,inner sep=1.8pt,fill=blue!80!black] (\x) at (\x,0){};
			\node at (\x,-0.25){$e-\c$};
		}
		\node[circle,inner sep=1.8pt,fill=blue!80!black] (6) at (6,0){};
		\node at (6,-0.25){$e$};
	\end{tikzpicture}.
\end{align*}

We orient each single edge arbitrarily. Following \cite[\S 4.7]{Ka-infdim-lie}, for each quiver above with the index set $I$, we define the generalized Cartan matrix as  \begin{align*}
	a_{ij}=
	\begin{cases}
		-1,& \text{ if } i\rightarrow j, i \leftarrow j, i\Rightarrow j \text{ or } i \fourlinerightarrow j, \\
		-2,& \text{ if } i\Leftarrow j \text{ or } i \Leftrightarrow j, \\
		-4,& \text{ if } i \fourlineleftarrow j, \\
		2, & \text{ if } i =j, \\
		0, & \text{ otherwise.}
	\end{cases}
\end{align*}
Moreover, we define $I_{\text{odd}}:=\{i\in I\mid a_{ij}\in 2\Z, \,\text{for all $j\in I$}\}$. Then we obtain a Cartan superdatum  $({\rm{A}},P,\Pi,\Pi^\vee)$,
with the parity function $|\cdot|: I\rightarrow\{\overline{0},\overline{1}\}$.

\subsection{Quiver Hecke superalgebra}

Recall that ${\rm R}$ is an integral domain of characteristic different from $2$. Let $\bigl({\rm{A}},P,\Pi,\Pi^\vee\bigr)$ be a Cartan superdatum in Subsection \ref{quivers}.

Let $u$ and $v$ be indeterminates over $\mathbb{K}$. For any $i,j\in I$, we define\begin{align*}
	Q_{i,j}(u,v)=
	\begin{cases}
		u-v, & \text{ if } i\rightarrow j, \\
		v-u, & \text{ if } i\leftarrow j, \\
		u-v^{2}, & \text{ if } i\Rightarrow j, \\
		v-u^{2}, & \text{ if } i\Leftarrow j, \\
		(u-v)(v-u), & \text{ if } i\Leftrightarrow j, \\
		u-v^{4}, & \text{ if } i\fourlinerightarrow  j, \\
		v-u^{4}, & \text{ if } i\fourlineleftarrow  j, \\
		0, & \text{ if } i= j. \\
		1, & \text{ otherwise.}
	\end{cases}
\end{align*}

\begin{defn}\cite[Definition 3.1]{KKT}
	Let $({\rm{A}},P,\Pi,\Pi^\vee)$ be a Cartan superdatum in Subsection \ref{quivers}.
	$\{Q_{i,j}|i,j\in I\}$ be chosen as above, and $n\in \N.$
	The quiver Hecke superalgebras $R_n$ is the superalgebra over ${\rm R},$ which is defined by the generators $$
	e({\bf i})\, ({\bf i}\in I^n), x_k\, (1\leq k\leq n),\, \tau_a (1\leq a\leq n-1),
	$$
	the parity $$
	\overline{e({\bf i})}=\bar{0},\quad \overline{x_ke({\bf i})}=\overline{\nu_k},\quad \overline{\tau_ae({\bf i})}=\overline{\nu_a} \cdot \overline{\nu_{a+1}},
	$$
	and the following relations:
	\begin{align}
		& e({\bf j})e({\bf i})=\delta_{{\bf j},{\bf i}}e({\bf i}),\,\,\text{for ${\bf j},{\bf i}\in I^\nu$}, \,\,e({\bf i})=\sum_{{\bf i}\in I^\nu}e({\bf i})\nonumber\\
		& x_px_qe({\bf i})=(-1)^{\overline{\bf i}_{p} \overline{\bf i}_{q}}x_qx_pe({\bf i}),\,\,\text{if $p\neq q$,}\nonumber\\
		& x_pe({\bf i})=e({\bf i})x_p,\,\,\,\tau_ae({\bf i})=e(s_a{\bf i})\tau_a,\,\,\text{where $s_a=(a,a+1)$,}\nonumber\\
		& \tau_ax_pe({\bf i})=(-1)^{\overline{\bf i}_{p} \overline{\bf i}_a\overline{\bf i}_{a+1}}x_p\tau_ae({\bf i}),\,\,\text{if $p\neq a,a+1$,}\nonumber\\
		& \bigl(\tau_ax_{a+1}-(-1)^{\overline{\bf i}_a\overline{\bf i}_{a+1}}x_a\tau_a\bigr)e({\bf i})\nonumber\\
		&\qquad =\bigl(x_{a+1}\tau_a-(-1)^{\overline{\bf i}_a\overline{\bf i}_{a+1}}\tau_ax_a\bigr)=\delta_{{\bf i}_a,{\bf i}_{a+1}}e({\bf i}),\nonumber\\
		&\tau_a^2e({\bf i})=Q_{{\bf i}_a,{\bf i}_{a+1}}(x_a,x_{a+1})e({\bf i}),\nonumber\\
		&\tau_a\tau_be({\bf i})=(-1)^{\overline{\bf i}_a\overline{\bf i}_{a+1}\overline{\bf i}_{b} \overline{\bf i}_{b+1}}\tau_b\tau_ae({\bf i}),\,\,\text{if $|a-b|>1$},\nonumber\\
		& (\tau_{a+1}\tau_a\tau_{a+1}-\tau_a\tau_{a+1}\tau_a)e({\bf i})\nonumber\\
		&=\begin{cases}
			\frac{Q_{{\bf i}_a,{\bf i}_{a+1}}(x_{a+2},x_{a+1})-Q_{{\bf i}_a,{\bf i}_{a+1}}(x_{a},x_{a+1})}{x_{a+2}-x_a}e({\bf i}), &\text{if ${\bf i}_a={\bf i}_{a+2}\in I_{\rm{even}}$;}\\
			(-1)^{\overline{\bf i}_{a+1}}(x_{a+2}-x_a)\frac{Q_{{\bf i}_a,{\bf i}_{a+1}}(x_{a+2},x_{a+1})-Q_{{\bf i}_a,{\bf i}_{a+1}}(x_{a},x_{a+1})}{x_{a+2}^2-x_a^2}e({\bf i}), &\text{if ${\bf i}_a={\bf i}_{a+2}\in I_{\rm{odd}}$;}\nonumber\\
			0, &\text{otherwise.}
		\end{cases}
	\end{align}
\end{defn}

$R_n$ is $\Z$-graded by setting
$$
\deg(e({\bf i}))=0,\quad \deg(x_ke({\bf i}))=(\nu_{{\bf i}_k}|\nu_{{\bf i}_k}),\quad \deg(\tau_ae({\bf i}))=-(\nu_{{\bf i}_a}|\nu_{{\bf i}_{a+1}}).
$$

\begin{prop}\cite[Corollary 3.15]{KKT}
	For each $w\in\Sym_n$,
	we fix a reduced expression $w=s_{i_1}\cdots s_{i_l}$, and define $\tau_w:=\tau_{i_1}\cdots\tau_{i_l}$,
	then the set of elements
	$$\{x^a \tau_w e({\bf i})\mid a\in (\Z_{\geq 0})^n,\,w\in \Sym_n,\,{\bf i}\in I^n\}$$
	forms a basis of the free ${\rm R}$-module $R_n,$ where
	$x^a=x_1^{a_1}\cdots x_n^{a_n}$ for
	$a=(a_1,\ldots,a_n)\in (\Z_{\geq 0})^n.$
\end{prop}

If $\Lambda\in P^+,\,i\in I$ and $u$ is an indeterminate over $\Z$, then we define
$$
a_i^\Lambda(x_1)=x_1^{\<h_i,\Lambda\>},\quad
a^\Lambda(x_1):=\sum_{{\bf i}\in I^n}x_1^{\<h_{{\bf i}_1},\Lambda\>}e({\bf i})\in R_n.
$$
\begin{defn}\cite[Section 3.7]{KKT}
	Let $\Lambda\in P^+$.
	The cyclotomic quiver Hecke superalgebra $R^\Lambda_n$ is defined to be the quotient algebra:
	$$R^\Lambda_n:=R_n/\<a^\Lambda(x_1)\>.$$
\end{defn}

$R^\Lambda_n$ inherits $\Z\times\Z_2$-grading from $R^\Lambda_n.$
That says, $R^\Lambda_n$ is a $\Z$-graded superalgebra too.
By a slight abuse of notation, we shall use the same symbols to denote the generators of both $R_n$ and $R^\Lambda_n.$
For any $\nu\in Q_n^+$, we define
$$ I^\nu:=\Bigl\{{\bf i}=({\bf i}_1,\cdots,{\bf i}_n)\in I^n\Bigm|\sum_{s=1}^{n}\nu_{{\bf i}_s}=\nu\Bigr\}. $$
Let $e_\nu:=\sum_{{\bf i}\in I^{\nu}}e({\bf i})$ be the central idempotent, then we define
$$R_\nu:=e_\nu R_n, \qquad R^\Lambda_\nu:=e_\nu R^\Lambda_n.$$
\subsection{Quiver Hecke-Clifford superalgebra}
Let $\bigl({\rm{A}}=(a_{ij})_{i,j\in I},P,\Pi,\Pi^\vee\bigr)$ be a Cartan superdatum in Subsection \ref{quivers}.  
Then we can define the quiver Hecke-Clifford
${\rm R}$-superalgebra $RC_n=RC_n(I).$

Let the set $J:=(I_{\rm odd}\times\{0\}) \sqcup (I_{\rm even} \times\{\pm \}).$ There is an involution $c: J\to J$ which fixes $I_{\rm odd}\times\{0\}$ and sends
$({\bf i}, \pm)$ to $({\bf i}, \mp)$ for each ${\bf i}\in I_{\rm even}$.
We also denote by $J^c:=I_{\rm odd}\times\{0\}$ the set of fixed points $\{j\in J \mid c(j)=j \}$ and
$\pr$ the canonical projection $J\to I$. The symmetric
group $\Sym_n$ acts on $J^n$ in a natural way.
For $p\in[n],$ we define
$c_p: J^n\to J^n$ by
$$c_{p}{\bf i}=(c^{\delta_{p\ell}}{\bf i}_{\ell})_{1\leq \ell\leq n}
\quad\text{for ${\bf i}=({\bf i}_1,\ldots,{\bf i}_n)\in J^n$.}$$

Following \cite[Remark 3.14]{KKT}, we define $\widetilde{Q}=(\widetilde{Q}_{j,j'}(u,v))_{j,j'\in J}\subseteq {\rm R}[u,v]$ be the family of polynomials via the following way: for any $(i,\varepsilon),(i',\varepsilon') \in J,$ where $i,i'\in I,$ $\varepsilon,\varepsilon'\in\{0,\pm\},$ we set
\begin{align}\label{extend Q-polys}
	\widetilde{Q}_{(i,\varepsilon),(i',\varepsilon')}(u,v):=Q_{i,i'}((-1)^{\varepsilon}u,(-1)^{\varepsilon'}v).
\end{align}

Note that $Q_{j,j'}(u,v)=Q_{j,j'}(-u,v)$ for $j\in J^c, j'\in J.$

\begin{defn}\cite[Definition 3.5]{KKT}
	Let $\bigl({\rm{A}}=(a_{ij})_{i,j\in I},P,\Pi,\Pi^\vee\bigr)$ be a Cartan superdatum,
	$Q=(Q_{j,j'}(u,v))_{j,j'\in J}$ be chosen as above, and $n\in \N$.
	The quiver Hecke-Clifford superalgebra
	$RC_n=RC_n(I)$ is the ${\rm R}$-superalgebra generated by the even generators
	$\{y_p\}_{1\le p\le n}$, $\{\sigma_a\}_{1\le a<n}$,
	$\{e({\bf i})\}_{{\bf i}\in J^n}$ and the odd generators $\{c_p\}_{1\leq
		p\leq n}$ with the following defining relations: for ${\bf j},{\bf i}\in
	J^n$, $1\leq p,\,q\leq n$, $1\leq a\leq n-1$, we have
	\begin{enumerate}
		\item $e({\bf j})e({\bf i})=\delta_{{\bf j},{\bf i}}e({\bf i})$, $1=\sum_{{\bf i}\in J^n}e({\bf i})$,
		$y_pe({\bf i})=e({\bf i})y_p$,  $c_pe({\bf i})=e(c_p{\bf i})c_p$,
		\item $y_py_q=y_qy_p$, $c_pc_q+c_qc_p=2\delta_{pq}$,
		\item $c_py_q=(-1)^{\delta_{p,q}}y_qc_p$,
		\item $\sigma_ae({\bf i}) = e(s_a{\bf i})\sigma_a, \sigma_a c_p = c_{s_a(p)}\sigma_a$,
		\item $\sigma_a y_pe({\bf i}) = y_{p}\sigma_ae({\bf i})$ if $p\not=a,a+1$,
		\item 
		$$\sigma_ay_{a+1}-y_a\sigma_a=
		\sum_{{\bf i}_a={\bf i}_{a+1}}e({\bf i})-\sum_{{\bf i}_a=c{\bf i}_{a+1}}c_ac_{a+1}e({\bf i}),$$
		\item$$y_{a+1}\sigma_a-\sigma_ay_a=
		\sum_{{\bf i}_a={\bf i}_{a+1}}e({\bf i})+\sum_{{\bf i}_a=c{\bf i}_{a+1}}c_ac_{a+1}e({\bf i}),$$
		\item $\sigma_a^2e({\bf i}) = \tQ_{{\bf i}_a,{\bf i}_{a+1}}(y_a,y_{a+1})e({\bf i})$,
		\item $\sigma_a\sigma_{b}=\sigma_{b}\sigma_{a}$ if $|a-b|>1$,
		\item
		\begin{align*}
			\sigma_{a+1}\sigma_a\sigma_{a+1}-\sigma_a\sigma_{a+1}\sigma_a
			&=\sum_{{\bf i}_a={\bf i}_{a+2}}
			\dfrac{\tQ_{{\bf i}_a,{\bf i}_{a+1}}(y_{a+2},y_{a+1})-\tQ_{{\bf i}_a,{\bf i}_{a+1}}(y_{a},y_{a+1})}
			{y_{a+2}-y_{a}}e({\bf i}) \\
			&\hs{-1ex}+\kern-2ex\sum_{{\bf i}_a=c{\bf i}_{a+2}}
			\dfrac{\tQ_{{\bf i}_a,{\bf i}_{a+1}}(y_{a+2},y_{a+1})-\tQ_{{\bf i}_a,{\bf i}_{a+1}}(-y_{a},y_{a+1})}{y_{a+2}+y_{a}}
			c_ac_{a+2}e({\bf i}).
		\end{align*}
	\end{enumerate}
\end{defn}

$RC_n$ is also $\Z$-graded by setting
$$
\deg(e({\bf i}))=0,\quad \deg(y_pe({\bf i}))=(\nu_{\pr({\bf i}_k)}|\nu_{\pr({\bf i}_k)}),\quad \deg(\sigma_ae({\bf i}))=-(\nu_{\pr({\bf i}_a)}|\nu_{\pr({\bf i}_{a+1})}).
$$

\begin{prop}\cite[Corollary 3.9]{KKT}
	For each $w\in \Sym_n$, we choose
	a reduced expression $s_{i_1}\cdots s_{i_{\ell}}$
	of $w$, and set $\sigma_w=\sigma_{i_1}\cdots
	\sigma_{i_{\ell}}$.
	Then the set of elements
	\begin{align*}
		\{
		y^a c^{\eta}\sigma_{w}e({\bf i})\mid
		a\in(\Z_{\geq 0})^n,\, \eta\in\Z_2^n,\,w\in\Sym_n,\,{\bf i}\in J^n \}
	\end{align*}
	forms an ${\rm R}$-basis of $RC_n,$ where
	$y^a=y_1^{a_1}\cdots y_n^{a_n}$ for
	$a=(a_1,\ldots,a_n)\in (\Z_{\geq 0})^n$ and
	$c^{\eta}=c_1^{\eta_1}\cdots c_n^{\eta_n}$
	for $\eta=(\eta_1,\cdots,\eta_n)\in\Z_2^n$.
\end{prop}

If $\Lambda\in P^+,\,j\in J$ and $u$ is an indeterminate over $\Z$, then we define
$$
a_j^\Lambda(u)=u^{\<h_{\pr(j)},\Lambda\>},\quad
a^\Lambda(y_1):=\sum_{{\bf i}\in J^n}y_1^{\<h_{\pr({\bf i}_1)},\Lambda\>}e({\bf i})\in RC_n.
$$
\begin{defn}\cite[Section 3.7]{KKT}
	Let $\Lambda\in P^+$.
	The cyclotomic quiver Hecke superalgebra $RC^\Lambda_n$ is defined to be the quotient algebra:
	$$RC^\Lambda_n:=RC_n/\<a^\Lambda(y_1)\>.$$
\end{defn}
Similarly, $RC^\Lambda_n$ inherits $\Z\times\Z_2$-grading from $RC^\Lambda_n.$
By a slight abuse of notation, we shall use the same symbols to denote the generators of both $RC_n$ and $RC^\Lambda_n.$


For any $\nu\in Q_n^+$, we define
$$ J^\nu:=\Bigl\{{\bf i}=({\bf i}_1,\cdots,{\bf i}_n)\in J^n\Bigm|\sum_{s=1}^{n}\nu_{\pr({\bf i}_s)}=\nu\Bigr\}. $$
Let $e_\nu^J=\sum_{{\bf i}\in J^{\nu}}e({\bf i})$ be the central idempotent, then we define
$$RC_\nu:=e_\nu^J RC_n, \qquad RC^\Lambda_\nu:=e_\nu^J RC^\Lambda_n.$$
Recall the canonical projection $\pr: J\to I$. We choose $J^\dag \subset J$ such that the projection $\pr$ induces a bijection $J^\dag\to I$. Let $e^\dag:=\sum_{{\bf i}\in {J^{\dag n}}}e({\bf i}).$
Kang, Kashiwara and Tsuchioka \cite{KKT} proved that the (cyclotomic) quiver Hecke superalgebra and the (cyclotomic) quiver Hecke-Clifford superalgebra are weakly Morita superequivalent to each other.

\begin{thm}\cite[Below Definition 3.10, Theorem 3.13]{KKT}\label{supermorita}
	Let $\Lambda\in P^+$ and $\nu=\sum_{i\in I}m_i\nu_i\in Q_+$. 
	
	\begin{enumerate}
		
		\item We have a
		$\Z\times \Z_2$-graded ${\rm R}$-algebra isomorphism
		\begin{align*}
			R_{\nu}^\Lambda\otimes \mathcal{C}_{m(\nu)}\cong e^\dag RC_{\nu}^\Lambda e^\dag.
		\end{align*}
		
		\item Suppose ${\rm R}=\mathbb{K}$ is a field, then we have the following morita superequivalent $$RC_{\nu}^\Lambda \overset{\text{sMor}}{\sim} e^\dag RC_{\nu}^\Lambda e^\dag.$$
	\end{enumerate}
\end{thm}

\begin{cor}\label{semisimple-equivalence}
	$R_n^\Lambda$ is semisimple if and only if $RC_n^\Lambda$ is semisimple.
	\end{cor}
	
	\begin{proof}
		We have the following. \begin{align*}
			\text{$R_n^\Lambda$ is semisimple } &\Leftrightarrow \begin{matrix}	&\text{For any $\nu\in I^n$, $R_\nu^\Lambda$ is either semisimple or zero.}\\ 
				&\text{ There exists some $\nu\in I^n$, such that $R_\nu^\Lambda\neq 0$. }
				\end{matrix}\\
	&\xLeftrightarrow{ \text{ Threorem \ref{supermorita}}} \begin{matrix}	&\text{For any $\nu\in I^n$, $RC_\nu^\Lambda$ is either semisimple or zero.}\\ 
				&\text{ There exists some $\nu\in I^n$, such that $RC_\nu^\Lambda\neq 0$. }
			\end{matrix}\\
			&\Leftrightarrow 	\text{ $RC_n^\Lambda$ is semisimple }
			\end{align*}
		\end{proof}

\section{Semisimplicity criterion for cyclotomic quiver Hecke superalgebra}\label{Semisimplicity criterion for cyclotomic quiver Hecke superalgebra}
\subsection{Semisimplicity criterion for type $A^{(1)}_{e}$}
In this subsection, we consider the Dynkin quiver of type $A^{(1)}_{e}(e> 1)$. As explained in Subsection \ref{quivers}, in this case, $I_{\text{odd}}=\emptyset$ and $R_n$ is the usual quiver Hecke algebra of type $A^{(1)}_{e}$. We shall explain how to use Theorem \ref{semisimplicity for cyclotomic Hecke-Clifford} to derive semisimplicity criterion for $R^\Lambda_n$. 

We assume $p\nmid (e+1)$. Therefore, we can choose $q^2$ to be an $(e+1)$-th primitive root of unity. We fix $\xi\notin\{q^k \mid k\in\Z\}$. 

\begin{prop}\label{semisimplicty typeA }
	Let $f=f^{(\mathtt{0})}_{\undQ}$, where $\undQ=(\xi q^{2i_1},\cdots,\xi q^{2i_m})$ and $i_s\in I=\{0,1,\cdots,e\}$, for $1\leq s\leq m$. Then $P^{(\mathtt{0})}_{n}(q^2,\undQ)\neq 0$ if and only if \begin{align}\label{criterion of type A}
e\geq n,	\text { and for any $1\leq s_1\neq s_2\leq m,$ we have $n\leq |i_{s_1}-i_{s_2}|\leq e+1-n.$}
	 \end{align}
	\end{prop}

\begin{proof}Note that in this case, $$\prod\limits_{s=1}^m\biggl(\prod\limits_{t=3-n}^{n-1}\bigl(Q^2_s-q^{-2t}\bigr)\prod\limits_{t=1-n}^{n}\bigl(Q^2_s-q^{-4t}\bigr)\biggr)
	\cdot \prod\limits_{1\leq s<s'\leq m}
	\bigl(Q_sQ_{s'}-q^{-2(t+1)}\bigr)\neq 0.$$ Hence, $P^{(\mathtt{0})}_{n}(q^2,\undQ)\neq 0$ if and only if \begin{align}\label{reduce in type A}\prod\limits_{t=1}^{n}\bigl(q^{2t}-1\bigr)\left(\prod\limits_{1\leq s<s'\leq m}\prod\limits_{t=1-n}^{n-1}\bigl({Q_s}-Q_{s'}q^{-2t}\bigr)\right)\neq 0.
		\end{align}Clearly, $\prod\limits_{t=1}^{n}\bigl(q^{2t}-1\bigr)\neq 0$ if and only if $ n\leq e.$ It is an easy exercise to compute that under the condition $0\leq k_s\leq e$ for any $1\leq s\leq m$ and  $1\leq n\leq e$, $\prod\limits_{t=1-n}^{n-1}\bigl({Q_s}-Q_{s'}q^{-2t}\bigr)\neq 0$ holds if and only if 
		 \begin{align*}
			\text { for any $1\leq s_1\neq s_2\leq m,$ we have $n\leq |i_{s_1}-i_{s_2}|\leq e+1-n.$}
		\end{align*} This completes the proof of the proposition.
	\end{proof}

Now we consider the cyclotomic polynomial being the following form : $f=f^{(\mathtt{0})}_{\undQ}$, where $Q_s=\xi q^{2i_s}$ and $i_s\in I=\{0,1,\cdots,e\}$, for $1\leq s\leq m$. We can associate $f$ with a dominant integral weight $\Lambda(f):=\sum\limits_{s=1}^m \Lambda_{i_s}\in P^+$. Then we have the following.

\begin{thm}\cite[Corollary 4.8]{KKT}\label{type A}
	$\mathcal{H}^{f}_{\mathbb{K}}(n)\cong RC^{\Lambda(f)}_n(I)$.
\end{thm}

It is clear that for any $\Lambda\in P^+$, we can find a unique $f= f^{(\mathtt{0})}_{\undQ}$, where $\undQ=(\xi q^{2i_1},\cdots,\xi q^{2i_m})$ and $k_s\in I=\{0,1,\cdots,e\}$, for $1\leq s\leq m$, such that $\Lambda=\Lambda(f)$.

\begin{cor}\label{semisimplicty typeA' }
	 Assume $p\nmid (e+1)$, then  $R^{\Lambda}_n(I)$ is semisimple if and only if \begin{align*}
	e\geq n,	\text { and for any $1\leq s_1\neq s_2\leq m,$ we have $n\leq |i_{s_1}-i_{s_2}|\leq e+1-n.$}
	\end{align*}
\end{cor}

\begin{proof}
	By Corollary \ref{semisimple-equivalence} and Theorem \ref{type A}, $R^{\Lambda}_n(I)$ is semisimple if and only if the corresponding cyclotomic Hecke-Clifford superalgebra $\mathcal{H}^{f}_{\mathbb{K}}(n)$ is semisimple. Hence the corollary follows from Theorem \ref{semisimplicity for cyclotomic Hecke-Clifford} and Proposition \ref{semisimplicty typeA }.
	\end{proof}
	
	\begin{rem}
	The condition in Corollary \ref{semisimplicty typeA' } appears weaker than that in \cite[Corollary 1.6.11]{Ma1}. However, it is easy to check that they are equivalent.
		\end{rem}
\subsection{Semisimplicity criterion for type $C^{(1)}_{e}$}
In this subsection, we consider the Dynkin quiver of type $C^{(1)}_{e}\, (e\geq 2)$ . As explained in Subsection \ref{quivers}, in this case, $I_{\text{odd}}=\emptyset$ and $R_n$ is the usual quiver Hecke algebra of type $C^{(1)}_{e}$. We shall explain how to use Theorem \ref{semisimplicity for cyclotomic Hecke-Clifford} to derive semisimplicity criterion for $R^\Lambda_n$. 

We assume $p\nmid e$. Therefore, we can choose $q^2$ to be a $2e$-th primitive root of unity.
\begin{prop}\label{semisimplicty typeC }
	Let $f= f^{(\mathtt{0})}_{\undQ}$, where $\undQ=(q^{2i_1-1},\cdots,q^{2i_m-1})$ and $i_s\in I=\{0,1,\cdots,e\}$, for $1\leq s\leq m$. Then $P^{(\mathtt{0})}_{n}(q^2,\undQ)\neq 0$ if and only if 
	\begin{enumerate}
		\item\label{criterion of type C1} For any $1\leq s\leq m,$ we have $\frac{n-1}{2} \leq i_s\leq e-\frac{n-1}{2}$.
		\item \label{criterion of type C2} For any $1\leq s_1\neq s_2\leq m,$ we have $n\leq |i_{s_1}-i_{s_2}|.$
	\end{enumerate}
	
\end{prop}
\begin{proof}Since $q^2$ is a $2e$-th primitive root of unity, we have $\prod\limits_{s=1}^m\prod\limits_{t=1-n}^{n}\bigl(Q^2_s-q^{-4t}\bigr)\neq 0.$ Hence, $P^{(\mathtt{0})}_{n}(q^2,\undQ)\neq 0$ if and only if \begin{align}\label{reduce in type C}\prod\limits_{t=1}^{n}\bigl(q^{2t}-1\bigr)\left(\prod\limits_{s=1}^m\prod\limits_{t=3-n}^{n-1}\bigl(Q^2_s-q^{-2t}\bigr)\right)
		\left(\prod\limits_{1\leq s<s'\leq m}\prod\limits_{t=1-n}^{n-1}\bigl({Q_s}-Q_{s'}q^{-2t}\bigr)
		\bigl(Q_sQ_{s'}-q^{-2(t+1)}\bigr)\right)\neq 0.
	\end{align}Clearly, $\prod\limits_{t=1}^{n}\bigl(q^{2t}-1\bigr)\neq 0$ if and only if $n\leq 2e-1.$ It is an easy exercise to compute that under the condition $0\leq i_s\leq e$ for any $1\leq s\leq m$ and  $1\leq n\leq 2e-1$, $\prod\limits_{s=1}^m\prod\limits_{t=3-n}^{n-1}\bigl(Q^2_s-q^{-2t}\bigr)\neq 0$ holds if and only if \eqref{criterion of type C1} holds. On the other hand, if \eqref{criterion of type C1} holds, then we have $\frac{n-1}{2} \leq e-\frac{n-1}{2}$. Combining with $e\geq 2$, this implies $n\leq 2e-1$. Therefore, $\prod\limits_{t=1}^{n}\bigl(q^{2t}-1\bigr)\left(\prod\limits_{s=1}^m\prod\limits_{t=3-n}^{n-1}\bigl(Q^2_s-q^{-2t}\bigr)\right)\neq 0$ if and only if \eqref{criterion of type C1} holds. It is an easy exercise to compute that under the condition \eqref{criterion of type C1}, $\prod\limits_{1\leq s<s'\leq m}\prod\limits_{t=1-n}^{n-1}\bigl({Q_s}-Q_{s'}q^{-2t}\bigr)\neq 0$ holds if and only if \eqref{criterion of type C2} holds. Moreover, if both \eqref{criterion of type C1},\eqref{criterion of type C2} hold, then $\prod\limits_{1\leq s<s'\leq m}\prod\limits_{t=1-n}^{n-1}\bigl(Q_sQ_{s'}-q^{-2(t+1)}\bigr)\neq 0.$ To sum up, we have proved that \eqref{reduce in type C} holds if and only if \eqref{criterion of type C1},\eqref{criterion of type C2} hold.
	\end{proof}

Now we consider the cyclotomic polynomial being the following form : $f= f^{(\mathtt{0})}_{\undQ}$, where $\undQ=(q^{2i_1-1},\cdots,q^{2i_m-1})$ and $i_s\in I=\{0,1,\cdots,e\}$, for $1\leq s\leq m$ and associate $f$ with a dominant integral weight $\Lambda(f):=\sum\limits_{s=1}^m \Lambda_{k_s}\in P^+$. Then we have the following.

\begin{thm}\cite[Corollary 4.8]{KKT}\label{type C}
	$	\mathcal{H}^{f}_{\mathbb{K}}(n)\cong RC^{\Lambda(f)}_n(I)$.
\end{thm}

It is clear that for any $\Lambda\in P^+$, we can find some $f= f^{(\mathtt{0})}_{\undQ}$, where $Q_s=q^{2i_s-1}$ and $i_s\in I=\{0,1,\cdots,e\}$, for $1\leq s\leq m$, such that $\Lambda=\Lambda(f)$.

\begin{cor}\label{semisimplicty typeC' }
	Assume $p\nmid e$, then  $R^{\Lambda}_n(I)$ is semisimple if and only if 
	\begin{enumerate}
		\item For any $1\leq s\leq m,$ we have $\frac{n-1}{2} \leq i_s\leq e-\frac{n-1}{2}$.
		\item  For any $1\leq s_1\neq s_2\leq m,$ we have $n\leq |i_{s_1}-i_{s_2}|.$
	\end{enumerate}
	
\end{cor}

\begin{proof}
	By Corollary \ref{semisimple-equivalence} and Theorem \ref{type C}, $R^{\Lambda}_n(I)$ is semisimple if and only if the corresponding cyclotomic Hecke-Clifford superalgebra $\mathcal{H}^{f}_{\mathbb{K}}(n)$ is semisimple. Hence the corollary follows from Theorem \ref{semisimplicity for cyclotomic Hecke-Clifford} and Proposition \ref{semisimplicty typeC }.
	\end{proof}
	
	\begin{rem}
		The conditions in Corollary \ref{semisimplicty typeC' } appear weaker than those in \cite[Theorem 1.1]{Sp}. However, it is easy to check that they are equivalent.
	\end{rem}
	
\subsection{Semisimplicity criterion for type $A^{(2)}_{2e}$}
In this subsection, we consider the Dynkin quiver of type $A^{(2)}_{2e}\, (e\geq 1)$ . As explained in Subsection \ref{quivers}, in this case, $I_{\text{odd}}=\{0\}$ and $R_n$ has non-trivial superstructure. We shall explain how to use Theorem \ref{semisimplicity for cyclotomic Hecke-Clifford} to derive semisimplicity criterion for $R^\Lambda_n$. 

We assume $p\nmid (2e+1)$. Therefore, we can choose $q^2$ to be a $(2e+1)$-th primitive root of unity. 
\begin{prop}\label{semisimplicty twised typeA }
Let $f= f^{(\bullet)}_{\undQ}$, where $\bullet\in\{\mathtt{0},\mathtt{s}\}$, $\undQ=( q^{2i_1},\cdots,q^{2i_m})$ and $i_s\in I=\{0,1,\cdots,e\}$, for $1\leq s\leq m$. Then  $P^{(\bullet)}_{n}(q^2,\undQ)\neq 0$ if and only if 
	\begin{enumerate}
		\item\label{criterion of twisted type A1} For any $1\leq s\leq m,$ we have $n \leq i_s\leq e-\frac{n-2}{2}$.
		\item \label{criterion of twisted type A2} For any $1\leq s_1\neq s_2\leq m,$ we have $n\leq |i_{s_1}-i_{s_2}|.$
	\end{enumerate}
	
\end{prop}
\begin{proof} We only prove the case when $\bullet=\mathtt{s}$, since the proof for the case $\bullet=\mathtt{0}$ is similar.  Using the fact that $q^2$ is a $(2e+1)$-th primitive root of unity, we deduce that $\prod\limits_{t=1}^{n}\bigl(q^{2t}-1\bigr)\neq 0$ holds if and only if $n\leq 2e.$ Moreover, when $n\leq 2e$, we have $\prod\limits_{t=1}^{n}\bigl(q^{2t}+1\bigr)\neq 0$.  A similar argument as in Proposition \ref{semisimplicty typeC } shows that $	\prod\limits_{t=1}^{n}\biggl(\bigl(q^{2t}-1\bigr)\bigl(q^{2t}+1\bigr)\biggr)\left(\prod\limits_{s=1}^m\prod\limits_{t=3-n}^{n-1}\bigl(Q^2_s-q^{-2t}\bigr)\right)\neq 0$ if and only if the following holds \begin{align}\label{criterion of twisted type A1'}
	\text{	For any $1\leq s\leq m,$ we have $\frac{n-2}{2} \leq i_s\leq e-\frac{n-2}{2}$.}
		\end{align} It is an easy exercise to compute that under the condition \eqref{criterion of twisted type A1'}, $\prod\limits_{s=1}^m\prod\limits_{t=1-n}^{n}\bigl(Q^2_s-q^{-4t}\bigr)\neq 0$ holds if and only if $i_s\geq n$ for any $1\leq s\leq m$. In conclusion, $	\prod\limits_{t=1}^{n}\biggl(\bigl(q^{2t}-1\bigr)\bigl(q^{2t}+1\bigr)\biggr)\prod\limits_{i=1}^m\biggl(\prod\limits_{t=3-n}^{n-1}\bigl(Q^2_i-q^{-2t}\bigr)\prod\limits_{t=1-n}^{n}\bigl(Q^2_i-q^{-4t}\bigr)\biggr)\neq 0$ holds if and only if \eqref{criterion of twisted type A1} holds. Again, a similar argument as in Proposition \ref{semisimplicty typeC } shows that under condition \eqref{criterion of twisted type A1}, $\prod\limits_{1\leq i<i'\leq m}\biggl(\prod\limits_{t=1-n}^{n-1}\bigl({Q_i}-Q_{i'}q^{-2t}\bigr)
		\bigl(Q_iQ_{i'}-q^{-2(t+1)}\bigr)\biggr)\neq 0$ holds if and only if \eqref{criterion of twisted type A2} holds. Hence we complete the proof of the proposition.
	\end{proof}
	
Now we consider the cyclotomic polynomial being the following form : $f= f^{(\bullet)}_{\undQ}$, where $\bullet\in\{\mathtt{0},\mathtt{s}\}$, $\undQ=( q^{2i_1},\cdots,q^{2i_m})$ and $i_s\in I=\{0,1,\cdots,e\}$, for $1\leq s\leq m$ and associate $f$ with a dominant integral weight as follows $$\Lambda(f):=\begin{cases}\sum\limits_{s=1}^m 2^{\delta_{i_s,0}}\Lambda_{k_s}\in P^+\qquad&\text{if $\bullet=\mathtt{0}$,}\\
		\Lambda_0+\sum\limits_{s=1}^m 2^{\delta_{i_s,0}}\Lambda_{i_s}\in P^+\qquad&\text{if $\bullet=\mathtt{s}$}.
	\end{cases}$$ 
	
	Then we have the following.
	
	\begin{thm}\cite[Corollary 4.8]{KKT}\label{twisted type A}
		$	\mathcal{H}^{f}_{\mathbb{K}}(n)\cong RC^{\Lambda(f)}_n(I)$.
	\end{thm}
	
	It is clear that for any $\Lambda\in P^+$, we can find some $f= f^{(\bullet)}_{\undQ}$, where $\bullet\in\{\mathtt{0},\mathtt{s}\}$, $\undQ=( q^{2i_1},\cdots,q^{2i_m})$ and $i_s\in I=\{0,1,\cdots,e\}$, for $1\leq s\leq m$ such that $\Lambda=\Lambda(f)$.

	\begin{cor}\label{semisimplicty twised typeA' }
		Assume $p\nmid (2e+1)$, then  $R^{\Lambda}_n(I)$ is semisimple if and only if 
		\begin{enumerate}
			\item For any $1\leq s\leq m,$ we have $n \leq i_s\leq e-\frac{n-2}{2}$.
			\item For any $1\leq s_1\neq s_2\leq m,$ we have $n\leq |i_{s_1}-i_{s_2}|.$
		\end{enumerate}
		
	\end{cor}
	
	\begin{proof}
		By Corollary \ref{semisimple-equivalence} and Theorem \ref{twisted type A}, $R^{\Lambda}_n(I)$ is semisimple if and only if the corresponding cyclotomic Hecke-Clifford superalgebra $\mathcal{H}^{f}_{\mathbb{K}}(n)$ is semisimple. Since $q^2$ is a $(2e+1)$-th primitive root of unity. We deduce that $\Gamma_n^{(\mathtt{s})}(q^2,\underline{Q})\neq 0$. 
		Then our corollary follows from Theorem \ref{semisimplicity for cyclotomic Hecke-Clifford} and Proposition \ref{semisimplicty twised typeA }.
		\end{proof}
		
	\begin{cor}\label{semisimple-equivalent 1}
		Assume $p\nmid (2e+1)$. Let $\Lambda\in P^+$ such that $\<h_0,\Lambda\>\in 2\N$. Then $R^{\Lambda}_n(I)$ is semisimple if and only if $R^{\Lambda+\Lambda_0}_n(I)$ is semisimple.
		\end{cor}
		
		\begin{proof}
			Note that when $\<h_0,\Lambda\>\in 2\N$, we can find a unique $f=f^{(\mathtt{0})}_{\undQ}$, where $\undQ=( q^{2i_1},\cdots,q^{2i_m})$ and $i_s\in I=\{0,1,\cdots,e\}$, for $1\leq s\leq m$ such that $\Lambda=\Lambda(f)$. By definition, we have $\Lambda+\Lambda_0=\Lambda(f^{(\mathtt{s})}_{\undQ})$. Hence the corollary follows from  Corollary \ref{semisimplicty twised typeA } immediately.
			\end{proof}
	\begin{example}
			Assume $e=5$ and $p\nmid (2e+1)$, and we consider $\Lambda=\Lambda_1+\Lambda_4$. Then by Proposition \ref{semisimplicty twised typeA },  $R^{\Lambda}_n(I)$ is semisimple if and only if $n\leq 2$. Both $R^{\Lambda}_2(I)$ and $R^{\Lambda+\Lambda_0}_2(I)$ are semisimple by Corollary \ref{semisimple-equivalent 1}.
		\end{example}
\subsection{Semisimplicity criterion for type $D^{(2)}_{e+1}$}

In this subsection, we consider the Dynkin quiver of type $D^{(2)}_{e+1}\, (e\geq 1)$ . As explained in Subsection \ref{quivers}, in this case, $I_{\text{odd}}=\{0,e\}$ and $R_n$ has non-trivial superstructure. We shall explain how to use Theorem \ref{semisimplicity for cyclotomic Hecke-Clifford} to derive semisimplicity criterion for $R^\Lambda_n$ where $\<h_0,\Lambda\>,\<h_e,\Lambda\>\in 2\N$.

We assume $p\nmid (e+1)$. Therefore, we can choose $q^2$ to be a $2(e+1)$-th primitive root of unity.  
 \begin{prop}\label{semisimplicty twised typeD }
 Let $f= f^{(\bullet)}_{\undQ}$, where $\bullet\in\{\mathtt{0},\mathtt{s},\mathtt{ss}\}$, $\undQ=( q^{2i_1},\cdots,q^{2i_m})$ and $i_s\in I=\{0,1,\cdots,e\}$, for $1\leq s\leq m$. Then  $P^{(\bullet)}_{n}(q^2,\undQ)\neq 0$ if and only if 
 \begin{enumerate}
 	\item \label{criterion of twisted type D1} For any $1\leq s\leq m,$ we have $n \leq i_s\leq e-n$.
 	\item \label{criterion of twisted type D2} For any $1\leq s_1\neq s_2\leq m,$ we have $n\leq |i_{s_1}-i_{s_2}|.$
 \end{enumerate}
 \end{prop}
\begin{proof}
This is exactly the same as Proposition \ref{semisimplicty twised typeA }.
 \end{proof}
 
Now we consider the cyclotomic polynomial being the following form : $f= f^{(\bullet)}_{\undQ}$, where $\bullet\in\{\mathtt{0},\mathtt{s},\mathtt{ss}\}$, $\undQ=( q^{2i_1},\cdots,q^{2i_m})$ and $i_s\in I=\{0,1,\cdots,e\}$, for $1\leq s\leq m$ and associate $f$ with a dominant integral weight as follows $$\Lambda(f):=\begin{cases}\sum\limits_{s=1}^m 2^{\delta_{i_s,0}+\delta_{i_s,e}}\Lambda_{k_s}\in P^+\qquad&\text{if $\bullet=\mathtt{0}$,}\\
 	\Lambda_0+\sum\limits_{s=1}^m 2^{\delta_{i_s,0}+\delta_{i_s,e}}\Lambda_{i_s}\in P^+\qquad&\text{if $\bullet=\mathtt{s}$},\\
 	\Lambda_0+\Lambda_e+\sum\limits_{s=1}^m 2^{\delta_{i_s,0}+\delta_{i_s,e}}\Lambda_{i_s}\in P^+\qquad&\text{if $\bullet=\mathtt{ss}$.}
 \end{cases}$$ 
 
 Then we have the following.
 
 \begin{thm}\cite[Corollary 4.8]{KKT}\label{twisted type D}
 	$	\mathcal{H}^{f}_{\mathbb{K}}(n)\cong RC^{\Lambda(f)}_n(I)$.
 \end{thm}
 
 Let $\Lambda\in P^+$ such that $\<h_0,\Lambda\>,\<h_e,\Lambda\>\in 2\N$. Then we can find some $f= f^{(\mathtt{0})}_{\undQ}$, where $\undQ=( q^{2i_1},\cdots,q^{2i_m})$ and $i_s\in I=\{0,1,\cdots,e\}$, for $1\leq s\leq m$ such that $\Lambda=\Lambda(f)$.
 
 \begin{cor}\label{semisimplicty twised typeD' }
 	Assume $p\nmid (e+1)$. Let $\Lambda\in P^+$ such that $\<h_0,\Lambda\>,\<h_e,\Lambda\>\in 2\N$. Then  $R^{\Lambda}_n(I)$ is semisimple if and only if \begin{enumerate}
 		\item For any $1\leq s\leq m,$ we have $n \leq i_s\leq e-n$.
 		\item For any $1\leq s_1\neq s_2\leq m,$ we have $n\leq |i_{s_1}-i_{s_2}|.$
 	\end{enumerate}
 \end{cor}
 \begin{proof}
By Corollary \ref{semisimple-equivalence} and Theorem \ref{twisted type D}, $R^{\Lambda}_n(I)$ is semisimple if and only if the corresponding cyclotomic Hecke-Clifford superalgebra $\mathcal{H}^{f}_{\mathbb{K}}(n)$ is semisimple. Then our corollary follows from Theorem \ref{semisimplicity for cyclotomic Hecke-Clifford} and Proposition \ref{semisimplicty twised typeD }.
 	\end{proof}
 	
 	\begin{rem}
 	Assume $p\nmid (e+1)$, then the ``if'' part of Corollary \ref{semisimplicty twised typeD' } is still true for arbitrary $\Lambda\in P^+$ by Theorem \ref{semisimple:non-dege} and Proposition \ref{semisimplicty twised typeD }. We conjecture the ``only if '' part also holds true for arbitrary $\Lambda\in P^+$. Motivated by Corollary \ref{semisimple-equivalent 1}, we also have the following Conjecture.
 		\end{rem}
 		
 	\begin{conjecture}\label{semisimple-equivalent 2}
 		Assume $p\nmid (e+1)$.
 		\begin{enumerate}
 		\item Let $\Lambda\in P^+$ such that $\<h_0,\Lambda\>\in 2\N$. Then $R^{\Lambda}_n(I)$ is semisimple if and only if $R^{\Lambda+\Lambda_0}_n(I)$ is semisimple.
 		\item  Let $\Lambda\in P^+$ such that $\<h_e,\Lambda\>\in 2\N$. Then $R^{\Lambda}_n(I)$ is semisimple if and only if $R^{\Lambda+\Lambda_e}_n(I)$ is semisimple.
 		\end{enumerate}
 		\end{conjecture}

	\end{document}